\documentclass[journal,draftcls,onecolumn,letterpaper,twoside,11pt]{IEEEtran}

\usepackage[margin=1in]{geometry}
\usepackage{amssymb, mathrsfs,verbatim} 
\usepackage[cmex10]{amsmath}
\usepackage{cite}
\RequirePackage[colorlinks,linkcolor=black,citecolor=black,urlcolor=blue]{hyperref}
\usepackage{enumerate}
\usepackage[dvips]{graphicx}
\usepackage{verbatim}
\usepackage{bbm}
\usepackage{subfig}
\usepackage{epsfig}
\usepackage{epstopdf}
\usepackage{graphicx} 

\usepackage{float,latexsym,amsfonts,amstext,times}

\usepackage{chngcntr}
\usepackage{apptools}

\usepackage{algorithm2e}
\usepackage{cases}

\usepackage{color}

\newtheorem{theorem}{Theorem}[section]
\newtheorem{definition}{Definition}[section]
\newtheorem{proposition}{Proposition}[section]
\newtheorem{corollary}{Corollary}[section]
\newtheorem{lemma}{Lemma}[section]

\newcommand{\cR}{\mathcal{R}}
\newcommand{\cS}{\mathcal{S}}

\newcommand{\cD}{\mathcal{D}}
\newcommand{\cC}{\mathcal{C}}
\newcommand{\cF}{\mathcal{F}}
\newcommand{\cP}{\mathcal{P}}
\newcommand{\cL}{\mathcal{L}}

\newcommand{\cT}{\mathcal{T}}
\newcommand{\cG}{\mathcal{G}}

\newcommand{\Pro}{\mathsf{P}}
\newcommand{\Exp}{\mathsf{E}}
\newcommand{\bN}{\mathbb{N}}

\newcommand{\bS}{\mathbb{S}}

\begin{document}

\title{Ordering sampling rules for sequential anomaly identification under sampling constraints
}

\author{Aristomenis Tsopelakos, Georgios Fellouris
}


\maketitle

\begin{abstract}
We consider the problem of sequential anomaly identification over multiple independent data streams, under the presence of a sampling constraint. The goal is to quickly identify those that  exhibit anomalous statistical behavior, when it is not possible to sample every source at each time instant. Thus, in addition to a stopping rule that  determines when to stop sampling, and a decision rule that indicates which sources to identify as anomalous upon stopping, one needs to specify a sampling rule that  determines which sources to sample at each time instant. We focus on the family of ordering sampling rules that select the sources to be sampled at each time instant based not only on the currently estimated subset of anomalous sources as the probabilistic sampling rules \cite{Tsopela_2022}, but also on the ordering of the sources' test-statistics. We show that under an appropriate design specified explicitly, an ordering  sampling rule leads to the optimal expected time for stopping among all policies that satisfy the same sampling and error constraints to a first-order asymptotic approximation as the false positive and false negative error thresholds go to zero. This is the first asymptotic optimality result for ordering sampling rules, when more than one sources can be sampled per time instant, and it is established under a general setup where the number of anomalous sources is not required to be known. A novel proof technique is introduced that encompasses all different cases of the problem concerning sources' homogeneity, and prior information on the number of anomalies. Simulations show that ordering sampling rules have better performance in finite regime compared to probabilistic sampling rules.
\end{abstract}


\section{Introduction}\label{sec:Intro} 
In many scientific and engineering problems with numerous data streams, it is  important to be able to  quickly identify those that  exhibit outlying statistical behavior.  For example, in navigation system integrity monitoring, it is critical to quickly identify a faulty sensor  
in order to remove it from the navigation system \cite[Section 11.1]{tartakovsky_book_2014}. 
For rapid intrusion detection in computer networks, anomaly-based detection systems are trained to recognize standard network behavior, detect deviations from the standard profile in real time, and identify those  deviations that can be  classified as potential network attacks \cite{tart1}. In  brain science, it is desirable to identify groups of cells with large vibration frequency, as this might be a symptom of a developing  malfunction \cite{Brain_2010}. Such applications motivate the study of \textit{sequential multiple testing} problems in which (i) there are multiple data  sequences generated by  distinct data sources, (ii) two hypotheses  are postulated for each of them, and (iii) the goal is to identify as quickly as possible  those data streams in which the alternative hypotheses hold, and which are often interpreted as  ``anomalous''. The data streams may be  observed continuously until a decision is reached for each (see, e.g.,  \cite{Bartroff_and_Song_2014,De_and_Baron_2012b,Song_and_Fellouris_2016,Song_and_Fellouris_2019,Bartroff_and_Song_2020,He_and_Bartroff_2021}), or there may be a constraint to sample only a fixed number of data  sources at each time instant (see, e.g., 
\cite{oddball_2018,Cohen2015active,huang2017active,
	Cohen2019nonlinearcost,Cohen2020composite,
	Tsopela_2019,Tsopela_2020,Cohen_hier,Cohen2023comp_hier,Prabhu2022}).
In the latter case, the  problem can be formulated as a sequential multi-hypothesis testing, where  at each time instant the action that influences the distribution of the observations is the  choice of sources to be sampled.  Thus, methods and  results from the literature of sequential design of experiments, or  sequential multi-hypothesis testing  \text{with controlled sensing} \cite{chernoff1959,albert1961,Bessler1960_I,Bessler1960_II,Kiefer_Sacks_1963,Keener_1984, Lalley_Lorden_1986,nitinawarat_controlled_2013,nitinawarat_controlled_2015,Aditya_2021} are applicable.

A weaker sampling constraint was proposed in \cite{Tsopela_2022}, according to which the number of sampled sources per time instant is not necessarily constant over time. In the same work it was shown that when using the stopping and decision rules proposed in \cite{Song_and_Fellouris_2016} for the full sampling case, the optimal expected time for stopping is achieved asymptotically if the long-run sampling frequency of each source is not smaller than a critical value that depends on the unknown subset of anomalous sources. Moreover, this criterion was shown to be satisfied simultaneously for every possible subset of anomalous sources by a \textit{probabilistic} sampling rule, according to which each source is sampled at each time instant with a probability greater than or equal to the aforementioned critical value.

The focus of the present paper is on a different sampling approach, which goes back to \cite[Remark 5]{chernoff1959}. Specifically, we consider a family of sampling rules that choose the sources to be sampled at each instant based not only on the currently estimated anomalous subset as the probabilistic sampling rules \cite{Tsopela_2022}, but also on the ordering of the sources' test-statistics, prioritizing those with the least evidence which in many cases are the ones with small absolute value. For this, we refer to them as  \textit{ordering} sampling rules. Such  sampling rules have been considered in the literature mainly in the case that \textit{the number of anomalous sources is known a priori}, and have been shown in simulation studies to be more  efficient than  probabilistic sampling rules (see e.g., \cite{Cohen2015active,huang2017active,Tsopela_2019}). Intuitively, this is because the ordering rules collect samples from the sources of least evidence, whereas the probabilistic rules assign a sampling probability to almost all sources. On the other hand, theoretical analysis for ordering sampling rules has been limited. A second-order asymptotic optimality analysis has been  conducted in \cite{Lalley_Lorden_1986} in the context of a  general controlled sensing problem. When translated to our  framework, this analysis requires that  a \textit{single} source can be  sampled at each time instant, and that it is a priori known that there is a \textit{single}  anomalous source.   Under the same setup, a first-order asymptotic analysis for the sequential anomaly identification problem has been conducted  when the testing problems in all sources are identical (homogeneous setup) in \cite{Cohen2015active,Cohen2020composite}, and under a specific non-homogeneous setup  in  \cite{huang2017active}.

In the present work we consider a general setup where \text{neither} the number of anomalies is required to be  a priori known, \text{nor} it is assumed that only one source can be sampled at each time instant. Specifically, as in \cite{Tsopela_2022}, (i) we do not make any  homogeneity assumption regarding the data sources, (ii) we  incorporate arbitrary lower and upper bounds on the  number of anomalous sources, (iii) we control arbitrary, distinct familywise error probabilities, (iv) we allow for an arbitrary upper bound on the conditional expected number of sampled sources given the past observations per time instant. Our main contribution in this work is that we establish the first-order asymptotic optimality of an ordering sampling rule in this general setup. Specifically, we show that with an appropriate design that is specified explicitly  an ordering sampling rule guarantees, under any possible unknown subset of anomalous sources, that the long-run sampling frequency of each source is equal to  or even larger than the critical value required for asymptotic optimality. To the best of our knowledge, this is  the first asymptotic optimality result  on \textit{ordering} sampling rules with  multiple  sampled sources per time instant, even when the number of anomalies is a priori known. Moreover, it unifies different setups regarding prior information on the number of anomalies and homogeneity/heterogeneity of the data sources, which have so far been  treated separately both methodologically and analytically.

The proposed sampling rule in this work differs from existing ordering sampling rules in the literature in two ways. First, (i) we have added a small, but critical, element of randomization. Specifically, at each time instant we allow \textit{at most one} source among those currently estimated as anomalous, and \textit{at most one} among those currently estimated as regular (non-anomalous), to be sampled with some probability. Second, (ii) we allow a subset of sources to be sampled with probability 1, depending on the current estimate of the anomalous sources. These features are not necessary in special cases such as when the number of anomalies is known a priori, and either the data sources are homogeneous, or exactly one source is sampled at each time instant. In fact, when the number of anomalies is known a priori and (ii) is not needed, we confirm the conjectured sufficient conditions for asymptotic optimality in \cite[(39)]{huang2017active}, by showing that they translate into cases where neither (i) is needed. 

Both parts of our proof differ substantially from existing approaches, even in the special case of a \textit{single} sampled source per time instant and a \textit{single} anomalous source in which previous approaches \cite{Cohen2015active,huang2017active} have been focused. To be specific, the first part in our proof is to show that the estimated subset of anomalous sources converges sufficiently fast to the true one. The second part is to show that the sampling frequency of each source converges sufficiently fast, and that the limits are greater than or equal to specified critical values. For this, we argue that it suffices to show that \textit{if the estimated anomalous subset is fixed at its true value, then for the anomalous (and respectively for the regular) data sources that are not sampled with probability $1$, it must hold that their test-statistics ``stay close"}. This is an intuitive property, since an ordering sampling rule prioritizes at each time instant the data sources whose current test-statistics are of low evidence. While the  proof  of this property is relatively simple when at most one source can be  sampled at each time instant, it turns out to be challenging when multiple sources can be sampled at each time instant.
 
The rest of the paper is organized as follows. In Section \ref{sec:formul}, we give the problem formulation. In Section \ref{sec:criterion}, we present the criterion for asymptotic optimality we will apply in our work. In Section \ref{sec:ordering}, we introduce the proposed family of ordering sampling rules, and we describe designs that achieve asymptotic optimality. Section \ref{consi} focuses on the consistency of ordering sampling rules, and Section \ref{asopt} on the asymptotic optimality. In Section \ref{sec:special}, we discuss special cases of the general theory and compare them to existing sampling rules. In Section \ref{sec: simulations}, we present a simulation study. In Section \ref{sec: conclusion} we have our conclusion, potential extensions of our work and future research directions. Technical parts of the proofs are organized in appendices presented in the Supplementary Material.

We end this section with some notation we use throughout the paper. We use  $:= $ to indicate the definition of a new quantity and $\equiv$ to indicate a duplication of notation. We set $\bN := \{1, 2 \ldots, \}$, $\bN_{0}:=\bN \cup \{0\}$, and  $[n] := \{1, \ldots, n\}$ for  $n \in \bN$. We denote by $A^c$ the complement, by  $|A|$ the size and by $2^A$ the powerset of a set $A$,  by $\lfloor  a  \rfloor$ the floor and by $\lceil a  \rceil$ the ceiling of a  positive number $a$, and by $\mathbf{1}$ the indicator of an event. The acronym iid stands for independent and identically distributed. We say that a sequence of positive numbers $\{a_n\, : \, n \in \bN \}$ is (i) \textit{summable} if $\sum_{n=1}^\infty a_n< \infty$, (ii) \textit{exponentially decaying} if there are  $c,d>0$ such that $a_n \leq c \exp\{-d\, n \}$ for every $n \in \bN$, and (iii) $p$-\textit{polynomially decaying} if there is  $c>0$ such that $a_n \leq c \, n^{-p}$ for every $n \in \bN$  . A sequence of random variables $\{ X(n),\, n \in \mathbb{N}\}$ converges $\Pro$-\textit{completely} to a real number $l$ if the sequence $\{\Pro(|X(n)-l|>\epsilon)\, : \,  n \in \bN \}$ is summable for all $\epsilon > 0$. 

\section{Problem formulation}\label{sec:formul}
Let $(\bS,  \cS)$ be an arbitrary  measurable space and  let $(\Omega, \cF, \Pro)$ be a probability space that hosts $M$ independent sequences of iid, $\bS$-valued random elements 
\begin{equation}\label{xi}
	X_{i}:=\{X_i(n)\, : \, n \in \bN\}, \qquad  i \in [M],
\end{equation}
generated by $M$ distinct data sources, and two independent sequences of iid Uniform$[0,1]$ random variables, 
\begin{equation*}
\hat{Z} := \{\hat{Z}_n \, :\, n \in \bN_0\}, \qquad  \check{Z} := \{\check{Z}_n \, :\, n \in \bN_0\},
\end{equation*}
which are used for randomization purposes. For each $i \in [M]$ and $n \in \bN$, $X_i(n)$ has density $f_i$ with respect to some $\sigma$-finite measure  $\mu_i$, that is equal to either $f_{0i}$ or $f_{1i}$, and  we refer to source $i$ as ``anomalous'' if $f_i=f_{1i}$, and as ``regular" if $f_i=f_{0i}$. We denote by $\Pro_A$ the underlying probability measure and by $\Exp_A$ the corresponding expectation when the subset of anomalous sources is  $$A:=\{ i \in [M] \,:\, f_i=f_{1i} \}.$$

We assume that it is \textit{a priori} known that there are at least $\ell$ and at most $u$ anomalous sources. That is, the  family of all possible subsets of anomalous sources is 
\begin{equation*}
\cP_{\ell, u}:=  \{D \subseteq [M]\, : \, \ell \leq |D| \leq u\},
\end{equation*}
where $\ell$ and $u$ are given, user-specified integers such that $0 \leq \ell \leq u \leq M$, with $\ell<M$, and $u >0$. Clearly, this encompasses the case where the number of anomalous data sources is a priori known ($\ell=u$), as well as the case where there is no prior information on the number of anomalies, i.e., $\ell=0, u=M$. 

We aim to identify all anomalous sources, if any, based on the sequentially acquired observations from all sources, \textit{under the constraint that it is not possible to observe all of them at every sampling instant}. For this, we need to specify a \textit{sampling} rule, a \textit{stopping} rule, and a \textit{decision} rule.
\begin{itemize}
	\item The sequence $R := \{R(n)\,:\, n \in \bN\}$ of $2^{[M]}$-valued random sets is a \textit{sampling} rule, if for every $n \in \bN$, $R(n)$ represents the subset of sources to be sampled at time $n$, and $R(n+1)$ is determined by the data collected up to the previous time instant, i.e., $R(n+1) \in \cF^R_n$, where $\cF^R_n$ is the $\sigma$-algebra generated by all available data up to time $n \in \bN$, i.e.,
	\begin{align*}
		\begin{split}
			\cF^R_n &:=
			\begin{cases} 
				\sigma(\hat{Z}_0,\,\check{Z}_0), \quad &\text{if} \quad  n=0,\\
				\sigma\left( \cF^R_{n-1},\,\hat{Z}_n,\,\check{Z}_n ,\, \{ X_{i}(n)\, :\, i \in R(n) \}  \right),  \quad &\text{if} \quad  n\in \bN.
			\end{cases}
		\end{split}
	\end{align*}  
	
	\item The random time $T$ is a \textit{stopping} rule, if $T$ is a $\Pro_A$-a.s. finite stopping time with respect to $\{\cF^R_n: n \in \bN\}$, for any set of anomalies $A \in \cP_{\ell,u}$, and we terminate sampling at time $n$ when $T=n$.
	 
	\item The sequence $\Delta := \{\Delta_n\,:\, n  \in \bN\}$ of $\cP_{l,u}$-valued random sets is a \textit{decision} rule, if for every $n \in \bN$, $\Delta_n \in \cF^R_n$, $\Delta_n$ represents the subset of sources that are estimated as anomalous at time $n$, and we declare them to be the anomalous sources when $T=n$.
\end{itemize}

The triplet $(R,T,\Delta)$ is called a policy, and it is clear that both stopping and decision rule depend on the sampling rule. We say that the policy $(R,T,\Delta)$ belongs to class $\cC(\alpha, \beta,\ell, u, K)$ if it satisfies the following two constraints.
\begin{enumerate}[(i)]
	\item The \textit{error constraint} according to which the probabilities of at least one \textit{false positive} and at least one \textit{false negative} upon stopping must not exceed the user-specified tolerance levels $\alpha$ and  $\beta$ in  $(0,1)$, respectively, i.e., for any set of anomalies $A \in \mathcal{P}_{\ell,  u}$,
	\begin{align} \label{err_const}
	  \Pro_{A} \left(\Delta_T \setminus A \neq \emptyset \right ) \leq\alpha,\quad \Pro_{A} \left(A \setminus \Delta_T \neq \emptyset  \right ) \leq \beta.
	\end{align}
	
	\item The \textit{sampling constraint} according to which the expected total number of observations up to stopping over the expected stopping time does not exceed a user-specified real number $K$ in $(0,M]$, i.e.,  for any set of anomalies $A \in \mathcal{P}_{\ell,  u}$,
	\begin{equation} \label{samp_const}
		\Exp_{A}\left[\sum_{n=1}^T |R(n)|\right] \leq K \; \Exp_{A}[T].
	\end{equation} 
\end{enumerate}

In  \cite[Theorem 5.2]{Tsopela_2022}, we proved that for any set of anomalies $A \in \mathcal{P}_{\ell,  u}$, and for any $\alpha,\, \beta \in (0,1)$, $\ell$, $u$, $K$, there is a policy $(R,T^R,\Delta^R)$ in $\cC(\alpha, \beta,\ell, u, K)$ that attains the minimum expected stopping time to a first-order asymptotic approximation as $\alpha, \beta \to 0$, i.e.,
\begin{equation}\label{AO1}
\lim_{\alpha,\, \beta \to 0}\, \frac{\Exp_{A}\left[T^R\right]}{\inf\limits_{(R,T,\Delta) \in  \cC(\alpha,\beta, \ell, u,K) } \; \Exp_{A}[T]}  = 1,
\end{equation}
assuming that $\lim\limits_{\alpha,\, \beta \to 0} |\log \alpha|/|\log \beta| = r \in(0, \infty)$ when $\ell < u$, as long as the sampling rule $R$ is designed so that each source $i \in [M]$ is sampled with a long run frequency greater than or equal to the specified critical value $c^{*}_{i}(A)$, in the sense that for all $i \in [M]$ and any $\epsilon > 0$,
\begin{equation}\label{ccs}
\sum_{n=1}^\infty \Pro_A \left( \pi_i^R(n)< c_i^*(A) -\epsilon\right)<\infty,
\end{equation}
where $\pi_{i}^R(n)$ is the empirical sampling frequency of source $i$ at time $n$, i.e.,
\begin{equation*}
\pi_i^R(n) := \frac{1}{n} \sum_{m=1}^n \mathbf{1}\{ i \in R(m)\}, \quad n \in \bN.
\end{equation*}
In \cite[Theorem 5.3]{Tsopela_2022}, we proved that the above condition is achieved by a probabilistic sampling rule which at time $n+1$ samples each source $i$ with probability $c_i^*(D)$, when $\Delta_n=D$. In this work, our \textit{goal} is to establish the analogous result for the family of ordering sampling rules, whose choice for the sources to be sampled at the following time instant does not depend only on $\Delta_n$, but also on the ordering of the sources' test-statistics. Our standing assumption throughout the paper is that, for each $i \in [M]$,  the Kullback-Leibler (KL) divergences of $f_{1i}$ and  $f_{0i}$ are positive and finite, i.e.,  
\begin{align*}
I_{i} := \int_{\bS} \log(f_{1i} / f_{0i}) \, f_{1i} \, d \nu_i \in (0, \infty), \quad  J_{i} := \int_{\bS} \log(f_{0i} / f_{1i}) \, f_{0i} \,  d \nu_i \in (0, \infty).
\end{align*}
However, for the main results of this work we will need to make the stronger assumption that
\begin{equation}\label{mc1}
	\sum_{i=1}^M \int_{\bS}  \left(  |\log(f_{1i}/f_{0i})|^{\mathfrak{p}} \, f_{1i}  + |\log(f_{0i}/ f_{1i})|^{\mathfrak{p}} \, f_{0i} \right)
	\; d\nu_i  < \infty.
\end{equation} 
for a particular $\mathfrak{p}>1$, sufficiently large. 

\section{A criterion for asymptotic optimality}\label{sec:criterion}
In this section, we review the criterion for asymptotic optimality, the formula of the test-statistic, and the stopping and decision rules we apply in this work. We further provide insightful properties of the minimum sampling frequencies in the long-run, i.e., $\mathbf{c^*}(A) := (c^*_1(A),\ldots,c^*_{M}(A))$.

\subsection{Log-Likelihood Ratio}\label{LLRR}
The test-statistic we apply for each source $i \in [M]$ is the Log-Likelihood Ratio (LLR) when we do sampling according to a sampling rule $R$, and we denote it by 
\begin{align} \label{LLR}
\Lambda^R_{i}(n) := \sum_{m=1}^n \log \left( \frac{f_{1i} (X_i(m))}{ f_{0i} (X_i(m))} \right) \, R_i(m), \quad n \in \bN,
\end{align} 
where $R_i(m)$ is the indicator of whether source $i$ is sampled at time $m$, i.e., $R_i(m):=\mathbf{1}\{ i \in R(m)\}$. For each $n \in \bN$, the decreasingly ordered LLRs are denoted by $\Lambda^R_{(1)}(n) \geq \ldots \geq  \Lambda^R_{(M)}(n)$, and we set $w^R_{i}(n)$ to be the corresponding index at time instant $n$, i.e.,
\begin{equation}\label{w_ord}
\Lambda^R_{(i)}(n):= \Lambda^R _{w^R_{i}(n)}(n), \quad i \in [M].
\end{equation}
 We set $\Lambda^R_{(0)}(n):= +\infty$, and $\Lambda^R_{(M+1)}(n) := -\infty$. When two LLRs are equal we arbitrarily break the tie.

\subsection{Stopping and decision rules}\label{subsec:stop_deci}
For any sampling rule $R$, the following stopping rule $T^R$ and decision rule $\Delta^R$ are defined such that the policy $(R,T^R,\Delta^R)$ satisfies the error constraint \eqref{err_const}, \cite[Theorems 3.1, 3.2]{Song_and_Fellouris_2016}.

When  the number of anomalous sources is known a priori, i.e., $\ell=u$,  $T^R$ stops  as soon as the  $\ell^{th}$ largest LLR  exceeds the next one by $\gamma {:=}|\log(\alpha \wedge \beta)|{+}\log(\ell(M{-}\ell))$, i.e., 
	\begin{align*}
		T^{R} &:=\inf \left\{ n \in \bN : \; \Lambda^R_{(\ell)}(n)-\Lambda^R_{(\ell+1)}(n) \geq \gamma \right\},
	\end{align*}  
	and $\Delta^R$ identifies as anomalous the sources with the $\ell$ largest LLRs, i.e., 
	\begin{align}\label{gap_decision rule}
		\Delta^R_n &:= \left\{w^R_1(n),\ldots, w^R_\ell(n) \right\}, \quad n \in \mathbb{N}. 
	\end{align}  
	
When the number of anomalous sources is completely unknown, i.e., $\ell=0$ and $u=M$, $T^R$ stops  as soon as every LLR is outside the interval $(-a, b)$, where $a:=|\log\beta|{+}\log M$, $b:=|\log \alpha|{+}\log M$, i.e., 
	\begin{align*}
		T^{R} := \inf \left\{n \in \bN : \;  \Lambda^R_{i}(n) \notin (-a,b)  \quad \text{for all}  \quad  i \in [M] \right\},
	\end{align*}
	and $\Delta^R$ identifies as anomalous the sources with positive LLRs, i.e.,  
	\begin{align}\label{intersection_dr}  
		\Delta^R_n:=\left\{ i \in [M] :  \; \Lambda^R_{i}(n)>0 \right\}, \quad n \in \mathbb{N}. 
	\end{align}
	
When $\ell<u$,  the stopping  and decision rules of the two previous cases are combined, and we have 
	\begin{align*}
		\begin{split}
			T^{R} := \inf  \{n \in \bN :  \quad &   \text{either} \quad  \Lambda^R_{(\ell+1)}(n){\leq}-a  
			\quad  \& \quad  \Lambda^R_{(\ell)}(n)-\Lambda^R_{(\ell+1)}(n) \geq c,  \\
			\quad & \; \text{or} \quad   \quad  \ell \leq p^{R}(n) \leq u \qquad  \& \qquad   \Lambda^R_{i}(n) \notin (-a,b)  \quad \forall \; 
			i \in [M],   \\
			\quad  & \; \text{or} \quad  \quad \;  \Lambda^R _{(u)}(n) \geq b \qquad  \& \quad  \Lambda^R_{(u)}(n)-\Lambda^R_{(u+1)}(n) \geq d \},  
		\end{split}
	\end{align*}
	where $a,b$ as in the previous case, $c:=|\log \alpha|+\log((M-\ell)M)$, $d:=|\log \beta|+\log(uM)$, and
	\begin{equation}\label{gi_decision_rule}
		\Delta^R_n :=\left\{ w^R_{i}(n) : \; i=1, \ldots, (p^R(n) \vee \ell) \wedge u \right\}, \quad n \in \mathbb{N},
	\end{equation}
	where $p^R(n)$ is the number of positive LLRs at time $n$. The sources that are identified as anomalous are the ones with positive LLRs, as long as their number is between $\ell$ and $u$. If their number is larger than $u$ (resp. smaller than $\ell$), then the anomalous are the ones  with the $u$ (resp. $\ell$) largest LLRs. 

\subsection{A criterion for asymptotic optimality}
In the rest of this paper, we consider the above stopping and decision rules, and we restrict our attention to the selection of a sampling rule that satisfies the sampling constraint \eqref{samp_const} with $T=T^R$, and it achieves the first-order asymptotic performance \eqref{samp_const} for an $A \in \cP_{\ell, u}$ in which case we call it \textit{asymptotically optimal} under $\Pro_{A}$. If $R$ satisfies the aforementioned conditions for any set of anomalies $A \in \cP_{\ell, u}$, we just call it \textit{asymptotically optimal}. In \cite[Theorem 5.2]{Tsopela_2022}, we proved that a sampling rule which satisfies the sampling constraint \eqref{samp_const} with $T=T^R$, is asymptotically optimal under $\Pro_{A}$, if it samples each source with a long run frequency greater than or equal to $c^{*}_{i}(A)$, in the sense described in \eqref{ccs}. Since, for any $c \geq c^{*}_{i}(A)$,
\begin{align*}
\Pro_A \left( \pi_i^R(n)< c_i^*(A) -\epsilon\right) \leq \Pro_A \left( \pi_i^R(n)< c -\epsilon\right) \leq \Pro_A \left( |\pi_i^R(n)-c| > \epsilon\right),
\end{align*}
based on the definition of complete convergence, we provide the following proposition from \cite[Theorem 5.2]{Tsopela_2022}, which states the criterion for asymptotic optimality under  $\Pro_{A}$  we use throughout this work.

\begin{proposition} \label{crit:AO} 
Let  $A {\in} \cP_{\ell, u}$ and let $R$ be a  rule that satisfies  \eqref{samp_const}  with $T{=}T^R$. If, for each  $i {\in} [M]$,  $\pi_i^R(n)$ converges completely to a number in  $[c^*_i(A), 1]$,  then $R$ is asymptotically optimal under $\Pro_A$.
\end{proposition}

The \textit{minimum long-run sampling frequencies} $c^*_i(A)$ are obtained from the solution of a max-min problem formulated and solved in \cite[Appendix]{Tsopela_2022}. In all cases of $\ell$, $u$, $K$, each $c^*_i(A)$ is inversely proportional to its KL divergence $I_i$ (resp. $J_i$) when $i \in A$ (resp. $i \notin A$), and 
\begin{equation}\label{A2}
c^*_{i}(A) \, I_{i}=c^*_{j}(A)\,I_{j}, \; \forall \, i,j\, \in A, \qquad \,  c^*_{i}(A) \, J_{i}=c^*_{j}(A)\,J_{j}, \; \forall \, i,j\, \notin A.
\end{equation}	
Thus, asymptotic optimality requires sampling more frequently the sources whose testing problems are harder, i.e., those of smaller KL numbers. In view of \eqref{A2}, in order to specify the vector $\mathbf{c^*}(A)$, it suffices to know the maximum element in $\{c_i^*(A): i \in A\}$, and the maximum element in $\{c_i^*(A): i \in A^c\}$, i.e.,
\begin{equation}\label{xy}
 	x(A):= \max_{i \in A} c_i^*(A),\qquad y(A):= \max_{i \in A^c} c_i^*(A),
\end{equation}
whose formulas were originally presented in \cite[Theorem 5.1]{Tsopela_2022}, and they are also given in Appendix \ref{app: defin x and y} for completeness. Hence,
\begin{equation}\label{c_bounds}
	c^*_{i}(A) = 	
	\begin{cases}
	x(A) \, I^*(A)/I_i \;\;	\leq I^*(A)/I_i,  \quad  &\text{if} \quad  i \in A,  \\
	y(A) \,  J^*(A^c)/J_i \leq J^*(A^c)/J_i,  \quad  &\text{if} \quad  i \in A^c,
	\end{cases}
\end{equation}	
where $I^*(A):= \min\limits_{i \in A}I_{i}$, $J^*(A^c):= \min\limits_{i \in A^c} J_{i}$, and the inequality follows by the fact that $x(A), y(A) \in [0,1]$.

The ratios $I^*(A) /I_i$ (resp.  $J^*(A^c) /J_i$) are the maximum sampling frequencies required for asymptotic optimality under $\Pro_A$ for a source $i \in A$ (resp. $i \in A^c$), when there are not sampling constraints, i.e., $K=M$, whereas  $x(A)$ (resp. $y(A)$) is the respective fraction of the maximum sampling frequency that is actually required for asymptotic optimality under the presence of sampling constraints. When $x(A)$ (resp. $y(A)$) is equal to $0$,  it is not necessary to sample any sources from $A$ (resp. $A^c$) in order to achieve asymptotic optimality under $\Pro_A$, but always one of $x(A)$, $y(A)$ is non-zero. When $x(A)$ (resp. $y(A)$) is equal to $1$, asymptotic optimality under $\Pro_A$ requires that  each  source in $A$ (resp. $A^c$) with KL divergence equal to $I^*(A)$ (resp. $J^*(A^c)$) to be sampled continuously, i.e. with probability $1$ in the long run,
\begin{equation} \label{excess_1} 
	\begin{split}
		x(A)=1  \quad  &\Leftrightarrow \quad 
		c_i^*(A)= 1, \quad \forall  \; i \in A \, \mbox{ s.t. } \, I_i=I^*(A)  
		\quad \;\; \Leftrightarrow \quad 
		\{i \in A: c_i^*(A)= 1\} \neq \emptyset,\\
		y(A)=1 \quad   &\Leftrightarrow \quad 
		c_i^*(A)=1,  \quad \forall  \;  i \notin A \, \mbox{ s.t. } \, J_i=J^*(A^c) 
		\quad  \Leftrightarrow \quad 
		\{i \notin A: c_i^*(A)= 1\} \neq \emptyset.
	\end{split}
\end{equation}
Last, we point out that from the formulas of $x(A)$, $y(A)$, we have
\begin{equation}\label{A1}
\sum_{i=1}^M  c_i^*(A) = x(A)  \, \sum_{i \in A}	\frac{I^*(A)}{I_i} + y(A) \, \sum_{i \notin A}  \frac{J^*(A^c)}{J_i} \leq K,
\end{equation}
which means that a sampling rule that samples each source $i \in [M]$ with probability $c_i^*(A)$ for every $i \in [M]$, would satisfy the sampling constraint \eqref{samp_const}, and as a result it would be asymptotically optimal. 

\section{Ordering sampling rules}\label{sec:ordering}
In this section, we introduce the family of \textit{ordering} sampling rules. Unlike a \textit{probabilistic} sampling rule, which specifies the probability with which each source is sampled at each time instant based only on the current estimate of the anomalous sources \cite[Section IV]{Tsopela_2022}, an ordering sampling rule takes also into account the ordering of the LLRs, and prioritizes the sources with the least evidence. Our ordering sampling rule has two fundamental differences compared to the sampling rule (termed as  ``deterministic'') in \cite{Cohen2015active,huang2017active}. The first one is that the proposed sampling rule is not fully deterministic, in the sense that we allow at most one source among those currently estimated as anomalous, and at most one among those currently estimated as regular, to be sampled with some probability. Second, we allow a subset of sources to be sampled with probability 1, depending on the current estimate of the anomalous sources. Both these novel features turn out to be critical for achieving asymptotic optimality. \textit{In what follows, $R$ denotes always an ordering sampling rule, and in order to lighten the notation we suppress the superscript $R$, e.g., we simply write $\Delta_n$, $\cF_n$, $\Lambda_{i}(n)$ instead of $\Delta_n^R$, $\cF^R_n$, $\Lambda^R_{i}(n)$.}

\subsection{Formulation of an ordering sampling rule}
Ordering sampling rules prioritize the sampling of the sources with the lowest statistical evidence. By definition of the stopping rule in Subsection \ref{subsec:stop_deci}, we observe that at each time instant, the sources with the lowest statistical evidence among those estimated as anomalous (resp. regular) are the ones with the smallest (resp. largest) LLRs. We first need to specify how many of the sources among those estimated as anomalous (resp. regular) we need to sample, at each time instant, and then choose the ones with the lowest statistical evidence from each subset. Thus, we need to define two functions $\hat{N}, \, \check{N} \, :\, \cP_{\ell, u} \, \to \, [0,K]$ such that $\hat{N}(D) \leq |D|$, $\check{N}(D) \leq |D^c|$, for all $D \in \cP_{\ell, u}$, and so that at each time $n+1$ we sample 
\begin{equation}\label{def_N}
	\begin{split}
		&\lfloor \hat{N}(\Delta_n) \rfloor  + \mathbf{1} \{ \hat{Z}_n\leq \hat{N}(\Delta_n)- \lfloor \hat{N}(\Delta_n) \rfloor \}\quad \;\, \text{sources from} \; \; \Delta_n, 
		\\
		\text{and} \quad  &\lfloor \check{N}(\Delta_n) \rfloor  + \mathbf{1} \left\{ \check{Z}_n \leq \check{N}(\Delta_n)  - \lfloor \check{N}(\Delta_n)  \rfloor \right\}
		\quad  \text{sources from} \; \;  \Delta_n^c,
	\end{split}
\end{equation}
where  $\{\hat{Z}_n: n \in \bN_0\}$ and $\{\check{Z}_n: n \in \bN_0\}$ are two independent sequences of independent, Uniform$[0,1]$ random variables, independent of the observations of the sources. For any true subset of anomalies $A \in \cP_{\ell, u}$, the expected number of sources we sample from $\Delta_n$ (resp. $\Delta_n^c$), at time n, are
\begin{equation}\label{expe_n}
	\begin{split}
		\sum_{i \in \Delta_n} \Exp_A\left[ R_{i}(n+1)   \, \big{|} \, \cF_{n} \right]  &=  \lfloor \hat{N}(\Delta_n) \rfloor  +  \hat{N}(\Delta_n)- \lfloor \hat{N}(\Delta_n) \rfloor =\hat{N}(\Delta_n), \; \mbox{ a.s.} \quad \forall \; A \in \cP_{\ell, u},\\
		\sum_{i \notin \Delta_n} \Exp_A\left[ R_i(n+1)   \, \big{|} \, \cF_{n} \right]  &= \lfloor \check{N}(\Delta_n) \rfloor  +  \check{N}(\Delta_n)  - \lfloor \check{N}(\Delta_n)  \rfloor = \check{N}(\Delta_n), \; \mbox{ a.s.} \quad \forall \; A \in \cP_{\ell, u}.
	\end{split}
\end{equation}
Hence, for each $D \in \cP_{\ell, u}$, $\hat{N}(D)$ (resp. $\check{N}(D)$) is  the expected  number of  sampled sources  among those estimated as anomalous (resp. regular) given all previously collected data, whenever the current estimate of the anomalous subset is equal to $D$, i.e. $\Delta_n=D$. We note that for all $D \in \cP_{\ell, u}$, $\hat{N}(D)$ (resp. $\check{N}(D)$) do not have to be integers, this is the reason why, whenever $\Delta_n=D$, we need to sample one source from $D$ \textit{with probability} $\hat{N}(D)- \lfloor \hat{N}(D) \rfloor$, and one from $D^c$ \textit{with probability} $\check{N}(D)- \lfloor \check{N}(D) \rfloor$.

In view of \eqref{excess_1}, we recall that in order to achieve asymptotic optimality we may need to sample some of the sources with probability $1$, regardless of the current ordering of their LLRs compared to the others. For this, we need to define two functions $\hat{G}, \, \check{G} \, :\, \cP_{\ell, u} \, \to \, 2^{[M]}$, such that 
\begin{align*}
	\begin{split}
		\hat{G}(D) &\subseteq D	 \qquad \; \text{and} \qquad |\hat{G}(D) |  \leq 
		\hat{N}(D), \quad \forall \; D \in \cP_{l,u},\\
		\check{G}(D) &\subseteq D^c	 \qquad
		\text{and} \qquad  |\check{G} (D) |  \leq  \check{N}(D), \quad \forall \; D \in \cP_{l,u},
	\end{split}
\end{align*}
so that the sources in $\hat{G}(\Delta_n)$ (resp. $\check{G}(\Delta_n)$) are sampled with probability $1$ at time $n+1$. 

\begin{definition}\label{definition_ordering}
We say that $R$ is an ordering sampling rule if there are functions $\hat{N}$, $\check{N}$, and $\hat{G}$, $\check{G}$, so that the sources to be sampled at time $n+1$, i.e., $R(n+1)$, are the sources in $\hat{G}(\Delta_n) \cup \check{G}(\Delta_n)$, the sources that correspond to the
\begin{equation*}
	\lfloor \hat{N}(\Delta_n) \rfloor -|\hat{G}(\Delta_n)| + \mathbf{1} \{ \hat{Z}_n \leq \hat{N}(\Delta_n)  - \lfloor \hat{N}(\Delta_n) \rfloor \}
\end{equation*}
\textit{smallest} LLRs in $\Delta_n \setminus \hat{G}(\Delta_n)$, and the sources that correspond to the
\begin{equation*}
	\lfloor \check{N}(\Delta_n) \rfloor -|\check{G}(\Delta_n)| + \mathbf{1} \{ \check{Z}_n \leq \check{N}(\Delta_n)  - \lfloor \check{N}(\Delta_n) \rfloor \}
\end{equation*}
\textit{largest} LLRs in $\Delta_n^c \setminus \check{G}(\Delta_n)$.
\end{definition}

We provide the condition that $\hat{N}$, $\check{N}$ must satisfy so that $R$ respects the sampling constraint \eqref{samp_const}.

\begin{proposition}\label{nn_kcons}
Let $R$ be an ordering sampling rule. If for the functions $\hat{N}$, $\check{N}$ it holds
\begin{equation}\label{cond_sampl_constraint}
	\hat{N}(D) + \check{N}(D) \leq K, \quad \forall \, D \in \cP_{l,u},
\end{equation}
then $R$ satisfies the sampling constraint \eqref{samp_const}.
\end{proposition}

\begin{IEEEproof}
Let us assume $\hat{N}$, $\check{N}$ that satisfy \eqref{cond_sampl_constraint}. In view of \eqref{expe_n}, for all $n \in \bN$, it holds
\begin{equation*}
	\begin{aligned}
		\Exp_{A}\left[|R(n+1)|\, \big{|}\, \cF_n \right] &=  \sum_{i \in \Delta_n} \Exp_A\left[ R_{i}(n+1)   \, \big{|} \, \cF_{n} \right] + \sum_{i \notin \Delta_n} \Exp_A\left[ R_i(n+1)   \, \big{|} \, \cF_{n} \right]\\
		&= \hat{N}(\Delta_n) + \check{N}(\Delta_n) \leq K.
	\end{aligned}
\end{equation*}
By the above inequality, and the fact that $T$ is an $\{\cF_n: n \in \bN\}$-stopping time, it follows that
\begin{equation*}
\begin{aligned}
\Exp_{A}\left[\sum_{n=1}^T |R(n)|\right] = \Exp_{A}\left[ \Exp_{A}\left[ \sum_{n=1}^{\infty} |R(n)| \mathbf{1} \{ T {\geq} n \}\, \big{|}\, \cF_{n-1} \right] \right]&= \Exp_{A}\left[ \sum_{n=1}^{\infty} \Exp_{A}\left[|R(n)|\, \big{|}\, \cF_{n-1} \right] \mathbf{1} \{ T {\geq} n \} \right] \\
&\leq K \, \Exp_{A}\left[ T \right].
\end{aligned}
\end{equation*}
\end{IEEEproof}

\subsection{The default design for asymptotic optimality}\label{default}
By Proposition \ref{crit:AO}, $R$ is asymptotically optimal under
$\Pro_A$ if it samples each source $i\in[M]$ with a long–run frequency at least $c_i^*(A)$. This suggests that, for $R$ to be asymptotically optimal, the expected number of sampled sources, among those in $D$ (resp.\ $D^c$), i.e., $\hat{N}(D)$ (resp.\ $\check{N}(D)$), when $\Delta_n =D$, should be at least equal to the sum of the respective minimum long–run sampling frequencies, i.e., 
\begin{equation} \label{N_opt}
	\begin{aligned}
	\hat{N}(D) \geq \sum_{i \in  D} c_i^*(D), \; \forall \, D \in \cP_{l,u}, \qquad \check{N}(D) \geq \sum_{i \notin D} c_i^*(D),  \; \forall \, D \in \cP_{l,u}.
	\end{aligned}
\end{equation}
In the same time, we want $\hat{N}$, $\check{N}$ to satisfy \eqref{cond_sampl_constraint} so that $R$ satisfies the sampling constraint \eqref{samp_const}. In view of \eqref{A1}, we verify that \eqref{cond_sampl_constraint} and \eqref{N_opt} are simultaneously satisfied when $\hat{N}$, $\check{N}$ are selected as
\begin{equation}\label{N_default}
 \hat{N}(D)  = \sum_{i \in  D} c_i^*(D), \; \forall \, D \in \cP_{l,u}, \qquad \check{N}(D)  = \sum_{i \notin D} c_i^*(D),  \; \forall \, D \in \cP_{l,u}.
\end{equation}
We note that \eqref{N_default} is not necessarily the only solution of \eqref{cond_sampl_constraint}-\eqref{N_opt}, especially when $K$ is large. As explained in \eqref{excess_1}, in some cases in order to achieve asymptotic optimality under $\Pro_A$, some of the sources should be sampled with probability $1$. This suggests that, when $\Delta_n =D$, $\hat G(D)$ (resp.\ $\check G(D)$) must contain at least the sources that need to be sampled with probability $1$ for $R$ to be asymptotically optimal, i.e.,
\begin{equation} \label{G_supset}
\hat{G}(D) \supseteq \{ i \in D :\, c^*_{i}(D)=1 \},  \; \forall \;  D \in \cP_{l,u},\qquad 
\check{G}(D) \supseteq \{ i \notin D :\, c^*_{i}(D)=1 \},  \; \forall \;  D \in \cP_{l,u}.
\end{equation}
In view of the default choice \eqref{N_default} for $\hat{N}$, $\check{N}$, the respective default choice for $\hat{G}$, $\check{G}$ is
\begin{equation} \label{G_default}
\hat{G}(D) = \{ i \in D :\, c^*_{i}(D)=1 \},  \; \forall \;  D \in \cP_{l,u},\qquad 
		\check{G}(D) = \{ i \notin D :\, c^*_{i}(D)=1 \},  \; \forall \;  D \in \cP_{l,u}.
\end{equation}
Next, we prove that selecting $\hat{N}$, $\check{N}$ according to \eqref{N_default}, and  $\hat{G}$,  $\check{G}$ according to \eqref{G_default} leads to asymptotic  optimality, which makes these choices the \textit{default design} for $R$. We will also show that any selection of $\hat{N}$ and  $\check{N}$  that  satisfies  \eqref{cond_sampl_constraint}-\eqref{N_opt} leads to asymptotic optimality as long as $\hat{G}$, $\check{G}$ are selected appropriately.

\subsection{A general design for asymptotic optimality}\label{general}
We describe a general design for an ordering sampling rule to be asymptotically optimal, beyond the default design. In the simple case, where $\hat{N}(D)=D$ (resp. $\check{N}(D)=D^c$), when $\Delta_n =D$ at time $n+1$ we sample all sources in $D$ (resp.\ $D^c$), and thus we choose $\hat G(D)=D$ (resp.\ $\check G(D)=D^c$). In general, the appropriate choice of $\hat G, \check G$ is more involved, and it is based on the following proposition, which describes the conditions that $\hat N$, $\check N$, $\hat G$, $\check G$ must satisfy in order to $R$ to be asymptotically optimality. 

\begin{proposition}\label{prop_ord_ao}
Let $R$ be an ordering sampling rule with $\hat{N}$, $\check{N}$ that satisfy constraint \eqref{cond_sampl_constraint}. Suppose that for all $D \in \cP_{\ell, u}$, the functions $\hat{N}$, $\hat{G}$ satisfy $\hat{G}(D)=D$, when $\hat{N}(D) = |D|$, and
\begin{equation}\label{ng_shat}
	\sum_{i \in 	D \setminus \hat{G}(D)} c^{*}_{i}(D)  \leq 	\hat{N}(D) - |\hat{G}(D)|  
	<  \sum_{i \in D \setminus \hat{G}(D)} \frac{ I^*(D \setminus \hat{G}(D))}{I_i}, \quad \; \mbox{when } \hat{N}(D) < |D|,
\end{equation}
and the functions $\check{N}$, $\check{G}$ satisfy $\check{G}(D)=D^c$, when $\check{N}(D) = |D^c|$, and
\begin{equation}\label{ng_scheck}			
	\sum_{i \in 	D^c \setminus \check{G}(D)} c^{*}_{i}(D) 
	\leq 	\check{N}(D) - |\check{G}(D)| < \sum_{i \in D^c \setminus \check{G}(D)} \frac{ J^*(D^c \setminus \check{G}(D))}{J_i}, \quad \mbox{when } \check{N}(D) < |D^c|.
\end{equation}		
If also, the moment condition \eqref{mc1} holds for
\begin{equation*}
\mathfrak{p} > \max_{D \in \cP_{\ell, u}}\left\{ 8, 3 \cdot 2^{ \lceil \hat{N}(D) \rceil - |\hat{G}(D)|-1} +1, 3 \cdot 2^{ \lceil \check{N}(D) \rceil - |\check{G}(D)|-1} +1\right\},
\end{equation*}
then $R$ is asymptotically optimal.
\end{proposition}

\begin{IEEEproof}
The proof is presented in Section \ref{asopt}.
\end{IEEEproof}

Given two functions $\hat N, \check N$ that satisfy  \eqref{cond_sampl_constraint}-\eqref{N_opt}, Proposition \ref{prop_ord_ao} suggests that the functions $\hat G, \check G$ should be selected so that \eqref{ng_shat} (resp. \eqref{ng_scheck}) are satisfied. For this, let us fix $D \in \cP_{\ell, u}$, and assume $\hat{N}(D) < |D|$, in order to determine $\hat G(D)$ that satisfies \eqref{ng_shat}, we focus on a strictly increasing sequence of sets,
\begin{equation}\label{st_g}
\emptyset = \hat \cG_0 \subset \hat \cG_1 \subset \cdots \subset \hat \cG_{|D|} = D
\end{equation} 
such that for each $u \in \{0, \dots, |D|\}$ it holds $|\hat{\cG}_u|=u$, and 
\begin{equation} \label{small_KL}
	I_i \leq I_j, \quad \forall \;\; i \in \hat{\cG}_u, \;\; \forall \;\; j  \in D \setminus \hat{\cG}_u,
\end{equation}
which means that the KL information numbers of all sources in $\hat{\cG}_u$ are less than or equal to the KL numbers of all sources in $D \setminus \hat{\cG}_u$. We choose our $\hat G(D)$ to be the \textit{minimum size} set among $\{\hat \cG_u : 0 \leq u \leq |D|\}$ that satisfies the right-hand side inequality of \eqref{ng_shat}, i.e., $\hat G(D)= \hat{\cG}_{u^*}$, where
\begin{equation}\label{st_u1}
	u^* := \arg\min_{u \in \{0, \dots, |D|\}} \bigg\{ \hat N(D) - u \leq \sum_{i \in D \setminus \hat \cG_u} I^*(D \setminus \hat \cG_u)/I_i 
	\bigg\}.
\end{equation}
When $\check{N}(D) < |D^c|$, we consider the strictly increasing sets  $\{\check \cG_u : 0 {\leq} u {\leq} |D^c|\}$, and $\check G(D)= \check{\cG}_{u^*}$ where
\begin{equation}\label{st_u2}
	u^* := \arg\min_{u \in \{0, \dots, |D^c|\}} \bigg\{ \check N(D) - u \leq \sum_{i \in D^c \setminus \check \cG_u} J^*(D^c \setminus \check \cG_u)/J_i \bigg\}.
\end{equation}

\begin{theorem}\label{thm_optG} 
 We fix $D \in \cP_{\l,u}$.
 \begin{enumerate}[(i)]
 	\item Any pair $\hat{N}(D)$, $\hat G(D)$ that satisfies \eqref{ng_shat} satisfies also \eqref{N_opt} and \eqref{G_supset} (resp. for $\check{N}(D)$, $\check G(D)$).
 	
 	\item For any $\hat{N}(D) < |D|$ (resp. $\check{N}(D) < |D^c|$) that satisfies \eqref{N_opt}, there is a $\hat G(D)$ (resp. $\check G(D)$) selected according to \eqref{st_u1} (resp. \eqref{st_u2}) that satisfies \eqref{ng_shat} (resp. \eqref{ng_scheck}).
 \end{enumerate}
\end{theorem}

\begin{IEEEproof}
(i) Let us consider $\hat{N}(D)$, $\hat G(D)$ that satisfy \eqref{ng_shat}. In order to show \eqref{N_opt} for $\hat{N}(D)$, we observe that by the left-hand side of \eqref{ng_shat},
\begin{equation*}
\hat{N}(D) \geq  \sum_{i \in D \setminus \hat{G}(D)} c^{*}_{i}(D) +|\hat{G}(D)| \geq \sum_{i \in D} c^{*}_{i}(D)
\end{equation*}
where the last inequality follows by the fact that $c^{*}_{i}(D) \leq 1$ for all $i \in [M]$. In order to show \eqref{G_supset}, we first note that by \eqref{excess_1}, if $x(D) < 1$ then $\{i \in D: c_i^*(D)= 1\}=\emptyset$, and thus \eqref{G_supset} holds trivially. Hence, we assume that $x(D) = 1$, and we will show that for any $j \in \{ i \in D : c_i^*(D) = 1 \}$, or equivalently any $j \in \{ i \in D : I_i = I^*(D) \}$, it holds $j \in \hat G(D)$. We prove this by contradiction. Suppose that there is a $j \in D$ such that $I_j = I^*(D)$ and $j \in D \setminus \hat G(D)$. Since $D \setminus \hat G(D) \subseteq D$, it holds $I^*(D) \leq I^*(D \setminus \hat G(D))$, but also $j \in D \setminus \hat G(D)$ and $I_j = I^*(D)$. Hence, $I^*(D) = I^*(D \setminus \hat G(D))$. By the left-hand side of \eqref{ng_shat},
\begin{equation*}
\hat{N}(D) - |\hat{G}(D)| \geq  \sum_{i \in D \setminus \hat{G}(D)} c^{*}_{i}(D) = \sum_{i \in D \setminus \hat{G}(D)} \frac{I^*(D)}{I_i} = \sum_{i \in D \setminus \hat{G}(D)} \frac{I^*(D \setminus \hat G(D))}{I_i},
\end{equation*}
where the first equality follows by \eqref{c_bounds} when $x(D)=1$, and the second by $I^*(D) = I^*(D \setminus \hat G(D))$. The above inequality contradicts the right-hand side inequality of \eqref{ng_shat}.

(ii) First, we note that for the set $\hat \cG_{u}$ with $u = \lfloor \hat N(D) \rfloor$ it holds $D \setminus \hat \cG_{u} \neq \emptyset$ because $\hat N(D) < |D|$, and
\begin{equation*}
\hat N(D) - u = \hat N(D) - \lfloor \hat N(D) \rfloor < 1\leq 
\sum_{i \in D \setminus \hat \cG_{u}} \frac{I^*(D \setminus \hat \cG_{u})}{I_i},
\end{equation*}
where the last inequality is true because there is always an $i \in D \setminus \hat \cG_{u}$ such that $I_i = I^*(D \setminus \hat \cG_{u})$. Thus, $\hat{\cG}_{u}$ satisfies the right–hand side inequality of \eqref{ng_shat}, and by checking the remaining subsets in \eqref{st_g} we can determine the $u^*$ in \eqref{st_u1}. Therefore, there is a set $\hat G(D)$ selected according to \eqref{st_u1} that satisfies the right–hand side inequality of \eqref{ng_shat}, and it remains to prove the left–hand side of \eqref{ng_shat}, i.e.,
\begin{equation*}
 \hat{N}(D) - |\hat{G}(D)| \geq  \sum_{i \in D \setminus \hat{G}(D)} c^{*}_{i}(D).
\end{equation*}
Since $\hat{G}(D) = \hat \cG_{u^*}$ is the minimum size set that satisfies right–hand side of \eqref{ng_shat}, for $\hat \cG_{u^* -1}$ it holds
\begin{equation}\label{dart1}
 \begin{aligned}
  \hat{N}(D) {-} (u^* {-}1) {\geq} \sum_{i \in D \setminus \hat \cG_{u^*-1}}\frac{I^*(D \setminus \hat \cG_{u^*-1})}{I_i}
\;\; \Leftrightarrow \;\; \hat{N}(D) {-} |\hat{G}(D)|  {\geq}  \sum_{i \in D \setminus \hat \cG_{u^*-1}} \frac{I^*(D \setminus \hat \cG_{u^*-1})}{I_i} {-}1,
 \end{aligned}
\end{equation}
where we used the fact that $u^* = |\hat G(D)|$. Let us denote by $j$ the source such that $\{j\} = \hat \cG_{u^*} \setminus \hat \cG_{u^*-1}$, and by property \eqref{small_KL} it holds $I_{j} = I^*(D \setminus \hat \cG_{u^*-1})$. Therefore, \eqref{dart1} becomes
\begin{equation*}
 \begin{aligned}
\hat{N}(D) - |\hat{G}(D)| \geq & \sum_{i \in D \setminus \hat \cG_{u^*-1}} \frac{I^*(D \setminus \hat \cG_{u^*-1})}{I_i} -1\\
 &= \sum_{i \in D \setminus \hat \cG_{u^*}} \frac{I^*(D \setminus \hat \cG_{u^*-1})}{I_i} + \frac{I^*(D \setminus \hat \cG_{u^*-1})}{I_j} -1
 \geq  \sum_{i \in D \setminus \hat{G}(D)} \frac{I^*(D \setminus \hat \cG_{u^*-1})}{I_i},
 \end{aligned}
\end{equation*}
where in the last inequality we used the fact that $\hat{G}(D) = \hat \cG_{u^*}$. Since $ I^*(D \setminus \hat \cG_{u^*-1}) \geq I^*(D)$, and by \eqref{c_bounds} $I^*(D)/I_i \geq c^{*}_{i}(D)$ for all $i \in  D$, we conclude the claim.
\end{IEEEproof}

Last, we prove that the default design described in Subsection \ref{default}, where $\hat N, \check N$ are chosen as in \eqref{N_default}, and $\hat G, \check G$ as in \eqref{G_default}, satisfies \eqref{ng_shat} (resp. \eqref{ng_scheck}). 

\begin{corollary}
The default ordering sampling rule is asymptotically optimal.  
\end{corollary}

\begin{IEEEproof}
Let us fix $D \in \cP_{\ell,u}$, and $\hat N(D)$, $\hat G(D)$ as in the default selection \eqref{N_default}, \eqref{G_default}. It suffices to show that for $\hat N(D) < |D|$, the pair $\hat N(D)$, $\hat G(D)$ satisfies \eqref{ng_shat}, and then the result follows by Proposition \ref{prop_ord_ao}. Indeed, for the default selection of  $\hat N(D)$, $\hat G(D)$, we have
\begin{equation}\label{tn}
\begin{aligned}
\hat N(D) - |\hat G(D)|= \sum_{i \in D} c^{*}_{i}(D) - \sum_{i \in \hat{G}(D)} c^{*}_{i}(D)= \sum_{i \in D \setminus \hat{G}(D)} c^{*}_{i}(D) \leq \sum_{i \in D \setminus \hat G(D)} \frac{I^*(D)}{I_i},
\end{aligned}
\end{equation}
where the last inequality follows by \eqref{c_bounds}. Since $\hat N(D) {<} |D|$, and $\hat G(D){=}\{ i {\in} D {:} I_i{=}I^*(D) \}$, by \eqref{excess_1} we have 
$I^*(D) {<} I_j$, $\forall j {\in} D {\setminus} \hat G(D)$, which implies $I^*(D) {<} I^*(D {\setminus} \hat G(D))$, and by \eqref{tn} we deduce \eqref{ng_shat}.
\end{IEEEproof}

\section{Consistency}\label{consi}
In this section, we establish the conditions on the functions $\hat{N}$, $\check{N}$, $\hat{G}$, $\check{G}$ so that an ordering sampling rule $R$ is \textit{quickly} consistent, in the sense that it guarantees that the estimated subset of anomalies $\Delta_n$ converges \textit{quickly} to the true subset of anomalies $A$. This property implies that we can recover the true subset of anomalies relatively fast, and it is the first step towards the proof of asymptotic optimality. 

\begin{definition} 
We say that a sampling rule $R$ is $z$-\textit{quickly} consistent under $\Pro_A$ for some $z \geq 1$, if $\Exp_{A} \left[(s_{A})^z \right] < \infty$, where $s_A$ is the random time starting from which the estimated subset of anomalies remains permanently equal to the true one, i.e.,
	\begin{equation}\label{sigmaA}
		s_{A} :=\inf\left\{n \in \bN : \Delta_m =A  \quad \text{for all} \; \;   m \geq n\right\}.
	\end{equation}  
When $R$ is $z$-\textit{quickly} consistent under $\Pro_A$, for any $A \in \cP_{\ell, u}$, we say that $R$ is $z$-\textit{quickly} consistent.
\end{definition}

The definition of a $z$-quickly consistent sampling rule $R$ is equivalent to the $z$-quick convergence of 
the estimated subset of anomalies $\Delta_n$ to the true subset $A$ under $P_A$ (see \cite[Def.~2.4.4]{tartakovsky_book_2014}), when sampling according to $R$. Next, we introduce a criterion for the $z$-quick consistency of a sampling rule $R$.
 
\begin{proposition}\label{propo1}
	Let us fix $A \in \mathcal{P}_{\ell,u}$ and $z \geq 1$. An ordering sampling rule $R$ is $z$-quickly consistent under $P_A$ 
	if there is a $\rho > 0$ such that
	\begin{equation}\label{series}
		\sum_{n=1}^{\infty} n^z \, \Pro_{A}\left(\pi_{i}(n)< \rho \right) < \infty,
	\end{equation}
	for every $i \in A$ when $x(A) > 0$, \textit{and} for every $i \notin A$ when $y(A) > 0$. In the special setup where $\ell=u$, it suffices that  \eqref{series} holds for every $i \in A$ when $x(A)>0$, \textit{or} for every $i \notin A$  when $y(A)>0$. 
\end{proposition}

\begin{IEEEproof}
  In order to prove that $R$ is $z$-quickly consistent, it suffices to show that $\Exp_{A} \left[(s_{A})^z \right] < \infty$. By the definition of $s_A$, we observe that $\{ s_{A} > n \}=\{ \exists\, m\geq n: \Delta_m \neq A \}$, and by the definition of $\Delta_n$ in \eqref{gap_decision rule}, \eqref{intersection_dr}, \eqref{gi_decision_rule}, we deduce the following inequalities,

  (i) if $\ell=u$, then
  \begin{equation*}
  	\Pro_{A}(s_{A} > n) \leq \sum\limits_{i \in A,\,j \notin A} \Pro_{A}\left(\exists\, m\geq n:\, \Lambda_{j}(m) \geq {\Lambda}_{i}(m)\right),
  \end{equation*}
  
  (ii) if $\ell<|A|<u$, then
  \begin{align*}
  	\begin{split}
  		\Pro_{A}(s_{A} > n)  \leq \sum\limits_{i \in A} \Pro_{A}\left(\exists\, m\geq n:\, \Lambda_{i}(m)<0\right) + \sum\limits_{j \notin A} \Pro_{A}\left(\exists\, m\geq  n:\, \Lambda_{j}(m) \geq 0 \right),
  	\end{split}
  \end{align*}	
  
  (iii) if $|A|=\ell$,  then
  \begin{align*}
  	\begin{split}
  		\Pro_{A}(s_{A} > n)  \leq  \sum\limits_{j \notin A} \Pro_{A}\left(\exists\, m\geq n:\, \Lambda_{j}(m)\geq 0\right) + \sum\limits_{i \in A,\,j \notin A} \Pro_{A}(\exists\, m\geq n:\, \Lambda_{j}(m) \geq {\Lambda}_{i}(m)),
  	\end{split}
  \end{align*}
  
  (iv) if $|A|=u$, then
  \begin{align*}
  	\begin{split}
  		\Pro_{A}(s_{A} > n)  \leq  \sum\limits_{i \in A} \Pro_{A}\left(\exists\, m\geq  n:\, \Lambda_{i}(m)<0\right) +  \sum\limits_{i \in A,\,j \notin A} \Pro_{A}\left(\exists \, m\geq  n:\, \Lambda_{j}(m) \geq {\Lambda}_{i}(m)\right).
  	\end{split}
  \end{align*}
  By inspection of the formulas $x(A)$, $y(A)$ in Appendix \ref{app: defin x and y}, we observe that in case (i) it holds $x(A)>0$ or $y(A)>0$,  in case (ii) both $x(A), y(A) >0$, in case (iii) it holds $y(A)>0$, and in case (iv) $x(A)>0$. Therefore, in order to prove $\Exp_{A} \left[(s_{A})^z \right] < \infty$, it suffices to show the following:
  
  (a) If $x(A)>0$ then for all $i \in A$ and $j \notin A$,
  \begin{align*}
  	\sum_{n=1}^{\infty} n^{z-1}\, \Pro_{A}\left(\exists\, m\geq n:\, \Lambda_{j}(m) \geq {\Lambda}_{i}(m)\right)\ < \infty, \quad
  	\text{and} \quad \sum_{n=1}^{\infty} n^{z-1}\, \Pro_{A}\left(\exists\, m\geq n:\, \Lambda_{i}(m)<0\right) \, <\, \infty.
  \end{align*}
  
  (b) If $y(A)>0$ then for all $i \in A$ and $j \notin A$,
  \begin{align*}
  	\sum_{n=1}^{\infty} n^{z-1}\, \Pro_{A}\left(\exists\, m\geq n:\, \Lambda_{j}(m) \geq {\Lambda}_{i}(m)\right) <\, \infty,  \quad
  	\text{and} \quad  \sum_{n=1}^{\infty} n^{z-1}\, \Pro_{A}\left(\exists\, m\geq n:\, \Lambda_{j}(m) \geq 0\right) <\, \infty.
  \end{align*}
   We prove case (a), as the proof for (b) is similar. We fix $i \in A$, $j \notin A$. For every $\rho > 0$, we have   
  \begin{equation}\label{eqeq_p}
  	\begin{aligned}
  		\Pro_{A}\left(\exists \, m\geq n:\, \Lambda_{j}(m) \geq {\Lambda}_{i}(m)\right)=&\Pro_{A}\left(\exists \, m\geq n:\, \Lambda_{j}(m) \geq {\Lambda}_{i}(m),\, \pi_{i}(m)\geq \rho \right)\\
  		&+ \Pro_{A}\left(\exists\, m\geq n:\, \pi_{i}(m)< \rho \right),
  	\end{aligned}
  \end{equation}
  and
  \begin{equation}\label{e_q}
  	\begin{aligned}
  		\Pro_{A}\left(\exists\,  m\geq n:\,\Lambda_{i}(m)<0\right) =& \Pro_{A}\left(\exists\, m\geq n:\,\Lambda_{i}(m)<0,\,\pi_{i}(m)\geq \rho\right)\\
  		&+\Pro_{A}\left(\exists\, m \geq n :\, \pi_{i}(m)< \rho\right).
  	\end{aligned}
  \end{equation}
  In \eqref{eqeq_p} (resp. \eqref{e_q}), the first term on the right hand side is exponentially decaying by Lemma \ref{lambda_tau}(i). For the second term we observe that
  \begin{equation}\label{pir}
  	\begin{aligned}
  		\sum_{n=1}^{\infty} n^{z-1} \, \Pro_{A}(\exists \, m \geq n :\, \pi_{i}(m)< \rho) \leq \sum_{n=1}^{\infty} n^{z-1} \, \sum_{m=n}^{\infty} \Pro_{A}(\pi_{i}(m)< \rho)=\sum_{n=1}^{\infty}  n^{z}\, \Pro_{A}(\pi_{i}(n)< \rho).
  	\end{aligned}
  \end{equation}
  Thus, it suffices to find $\rho >0$ such that \eqref{pir} is summable, which is possible by assumption \eqref{series}.
\end{IEEEproof}

Proposition \ref{propo1} states that in order for $R$ to be $z$-quickly consistent under $\Pro_{A}$, it suffices to sample each source $i \in A$ when $x(A) >0$, and each source $j \in A^c$ when $y(A) >0$, with a small frequency $\rho$ in the long-run that can be much smaller than the $c^*_{i}(A)$ (resp. $c^*_{j}(A)$) required for asymptotic optimality. When $x(A)$ (resp. $y(A)$) is equal to $0$, it is not necessary to sample any source in $A$ (resp. $A^c$) to achieve consistency under $\Pro_{A}$. In the special setup where $\ell=u$, even if both $x(A)$, $y(A)$ are positive it suffices to sample with some small frequency $\rho$ only the sources in $A$, or only the sources in $A^c$. Therefore, in order for $R$ to be $z$-quickly consistent, the $\hat{N}$, $\check{N}$, $\hat{G}$, $\check{G}$ must be chosen such that the sources in $D$ (resp. $D^c$) are sampled at least with a small frequency when $x(D) >0$ (resp. $y(D)>0$), when $\Delta_n = D$.

\begin{theorem}\label{thm_c} 
	Suppose that the moment condition \eqref{mc1} holds for some $\mathfrak{p} \geq 2$. Let $R$ be an ordering sampling rule. If the functions  $\hat{N}$, $\hat{G}$ satisfy
	\begin{equation}\label{NGh_cons}
		\hat{N}(D) > |\hat{G}(D)|  
		\quad \mbox{for all } \; D \in \cP_{l,u} \; \mbox{ such that }\; \hat{G}(D) \neq D \quad \text{and} \quad x(D)>0,
	\end{equation}	
	and  the functions $\check{N}$, $\check{G}$ satisfy
	\begin{equation}\label{NGch_cons}
		\check{N}(D) >|\check{G}(D)|   \quad \mbox{for all } \; D \in \cP_{l,u} \; \mbox{ such that }\; \check{G}(D) \neq D^c \quad \text{and} \quad  y(D)>0, 
	\end{equation}	
	then there exist  $\rho \in (0,1)$ and $C>0$ such that for any $A \in \cP_{\ell,u}$,
	\begin{equation}\label{sh_c}
		\Pro_{A}\left(\pi_{i}(n) < \rho \right) \leq C \, n^{-\mathfrak{p}/2},  \qquad \forall \, n \in \bN,
	\end{equation}	
	for every $i \in A$ when  $x(A)>0$, and for every  $i \notin A$ when  $y(A)>0$.
\end{theorem}

\begin{IEEEproof}
We provide a sketch of the proof of Theorem \ref{thm_c}. The detailed proof is presented in Appendix \ref{consist_orde}. For the purposes of the sketch, we can assume w.l.o.g. that that we sample exactly $K$ sources at each time instant. For each set $V \subseteq [M]$, we consider the event $ \left\{ \Pi_{V}(n)<\rho \right\}:=\{ \pi_{i}(n)<\rho,\, \forall\, i \in V \}$, and we prove that for any $V \subseteq [M]$,
\begin{equation}\label{tsh00}
  \Pro_{A}\left( \Pi_{V}(n)<\rho \right) \leq C\, n^{-q/2}, \quad \mbox{a.s.} \quad \forall \, n \in \bN,
\end{equation}
or equivalently $\Pro_{A}\left( \Pi_{V}(n)<\rho \right)$ is $q/2$-polynomially decaying. For this we apply a backwards induction argument on the size of $V$. We start by proving the basis of induction, i.e. $V=[M]$, and for the induction step we assume that \eqref{tsh00} holds for all $V \subseteq [M]$ of size $|V| = v+1$, and for an appropriately chosen $\rho$ we show that \eqref{tsh00} holds for all  $V \subseteq [M]$ of size $|V| = v$. In this way, we end up proving \eqref{tsh00} for all singleton subsets of $[M]$, which is \eqref{sh_c}. For the basis of induction, we choose $\rho < K/M$, and we have
\begin{equation*}
  \Pro_{A}\left( \Pi_{V}(n)<\rho \right) \leq \Pro_{A}\bigg( \sum_{i=1}^{M} \pi_{i}(n) < K \bigg) =0,
\end{equation*}
where the right-hand side is equal to $0$ because we sample exactly $K$ sources at each time instant. The induction step is based on the observation that the event on which the statistics in $A\setminus V$ are positive and greater than those in $V$,  and  the statistics in $A^c \setminus V$ are negative and smaller than those in $V$, i.e., 
\begin{equation*}
	E_{A,V}(n):= \bigcap_{i \in A\setminus V} \bigcap_{j \in V}  \bigcap_{z \in  A^c \setminus V }\{ \Lambda_{i}(n) \geq \max\{0,\Lambda_{j}(n) \}  \geq \min\{ 0,\Lambda _{j}(n) \} \geq  \Lambda_{z}(n)  \},
\end{equation*}
has high probability if at time $n$ there  have been collected relatively few samples from sources in $V$ and relatively many samples from sources not in $V$. Since
\begin{equation*}
	\Pro_{A}\left( \Pi_{V}(n)<\rho \right) = \Pro_{A}\bigg( \Pi_{V}(n)<\rho, \bigcup_{m =\lceil n/2 \rceil}^{n} E^c_{A,V}(m)\bigg) + \Pro_{A}\bigg( \Pi_{V}(n)<\rho, \bigcap_{m =\lceil n/2 \rceil}^{n} E_{A,V}(m)\bigg),
\end{equation*}
it suffices to show that both terms on the right-hand side are $q/2$-polynomially decaying. In Appendix \ref{consist_orde}, we show that the first term is bounded by 
\begin{equation*}
\Pro_{A}\bigg( \Pi_{V}(n)<\rho, \bigcup_{m =\lceil n/2 \rceil}^{n} E^c_{A,V}(m)\bigg) \leq \Pro\bigg( \bigcup_{m =\lceil n/2 \rceil}^{n} \{ \Pi_{V}(m)<2\rho \} \cap E^{c}_{A,V}(m)\bigg),
\end{equation*}
and we prove that it is $q/2$-polynomially decaying because for each $m$ the event $\{ \Pi_{V}(m){<}2\rho \} \cap E^{c}_{A,V}(m)$ has small probability. For the second term, we prove that it is exponentially decaying because on the event $\left\{ \bigcap_{m =\lceil n/2 \rceil}^{n} E_{A,V}(m) \right\}$ the sources in $V$ are the ones with the smallest positive (resp. largest negative) LLRs for each $m \in [ n/2 , n]$, and by definition an ordering sampling rule $R$ samples at least one of the sources in $V$ at each time $m \in [ n/2 , n]$, which makes the probability of $\{ \Pi_{V}(n)<\rho \}$ very small.
\end{IEEEproof}

Next, we show that the conditions \eqref{NGh_cons}-\eqref{NGch_cons} suffice to guarantee the $z$-quick consistency of $R$.

\begin{corollary}\label{consi_deter}
	Suppose that the moment condition \eqref{mc1} holds for some $\mathfrak{p} > 4$ and that the conditions \eqref{NGh_cons}-\eqref{NGch_cons} hold, then $R$ is $z$-quickly consistent for any $z \in [1, \mathfrak{p}/2 -1)$.
\end{corollary}

\begin{IEEEproof}
 We fix $A \in \cP_{\ell, u}$, and $z {\in} [1, \mathfrak{p}/2 -1)$. In view of Proposition \ref{propo1}, it suffices to show that there is a $\rho  {\in} (0,1)$ such that \eqref{series} holds for every $i \in A$ when $x(A)>0$, and for every $i \notin A$ when $y(A)>0$, which encompasses the looser requirement we have for the case $\ell=u$. By Theorem \ref{thm_c}, it follows that there exist $\rho {\in} (0,1)$ and $C{>}0$ such that for every  $i {\in} A$ when $x(A){>}0$, and every  $i {\notin} A$ when  $y(A){>}0$,
 \begin{equation}
 	n^z \, \Pro_{A}\left( \pi^R_{i}(n)<\rho  \right) \leq C \, n^{z -\mathfrak{p}/2}, \quad \forall \, n \in \bN,
 \end{equation}	
 and since $z-\mathfrak{p}/2 < -1$ the bounding sequence is summable, which proves the claim. 
\end{IEEEproof}

\section{Asymptotic optimality}\label{asopt}
In this section, we prove Proposition \ref{prop_ord_ao} which describes the conditions that the functions $\hat{N}$, $\check{N}$, $\hat{G}$, $\check{G}$ must satisfy so that an ordering sampling rule $R$ is asymptotically optimal. Our assumption for the development of the results of this section is that $R$ is $z$-quickly consistent for all $z \in [1, \mathfrak{p}/2 - 1)$. Intuitively, when $R$ is $z$-quickly consistent and $A$ is the true subset of anomalies, then in the long run the rule $R$ will sample at each sampling instant all sources in $\hat{G}(A)$ (resp. $\check{G}(A)$) with probability $1$, and the expected number of sampled sources in $A \setminus \hat{G}(A)$ (resp. $A^c \setminus \check{G}(A)$) will be $\hat{N}(A)$ (resp. $\check{N}(A)$). Thanks to the symmetry of the problem (anomalous/regular sources), we can show the claim only for the sources in $A$, as the result for the sources in $A^c$ follows in the same way. The following theorem provides the conditions for the complete convergence of the empirical sampling frequencies under $\Pro_A$.

\begin{theorem}\label{pi_conv}
	Fix $A \in \cP_{\ell,u}$ and let $R$ be an ordering sampling rule that is $z$-quickly consistent under $\Pro_A$ for all $z \in [1, \mathfrak{p}/2 - 1)$. Suppose that the moment condition \eqref{mc1} holds for $\mathfrak{p} > 4$. 
    \begin{enumerate}[(i)]
    	\item Then, for all $i \in \hat{G}(A)$ it holds that $\pi_{i}(n) \to 1$, $\Pro_{A}$-completely.
    	
    	\item If $|A\setminus \hat{G}(A)|=1$ and $\hat{N}(A) - |\hat{G}(A)| < 1$, then for the single source $i \in A \setminus \hat{G}(A)$ it holds that $\pi_{i}(n) \to \hat{N}(A) - |\hat{G}(A)|$, $\Pro_{A}$-completely.
    	
    	\item If $|A\setminus \hat{G}(A)|\geq 2$, the $\hat{N}(A)$, $\hat{G}(A)$ satisfy
    	      \begin{equation}\label{mchat}
    	      	\hat{N}(A) - |\hat{G}(A)| < \sum_{i \in A \setminus \hat{G}(A)} I^*(A \setminus \hat{G}(A))/I_i,
    	      \end{equation}
    	      and the moment condition \eqref{mc1} holds for
    	      \begin{equation}\label{epp1}
    	      	\mathfrak{p} > \max \left\{8,3\cdot 2^{\lceil \hat{N}(A) \rceil - |\hat{G}(A)|-1}+1\right\},
    	      \end{equation}
    	      then $\pi_{i}(n) \to c_{i}(A)$, $\Pro_{A}$-completely for all $i \, \in \, A \setminus \hat{G}(A)$, and the $\{ c_{i}(A) \, :\, i \in A\setminus \hat{G}(A) \}$ satisfy
    	      \begin{equation}\label{i0}
    	      	\begin{aligned}
    	      		c_{i}(A)\,I_{i}=c_{j}(A)\,I_{j}, \; \forall \, i,j\, \in A \setminus \hat{G}(A),\quad \mbox{and} \quad
    	      		\sum_{ i \in A \setminus \hat{G}(A)}c_{i}(A) =\hat{N}(A)-|\hat{G}(A)|.
    	      	\end{aligned}
    	      \end{equation}
    \end{enumerate}
\end{theorem}

\begin{IEEEproof}
The proof of Theorem \ref{pi_conv} is given in the end of Subsection \ref{stab_ord}.
\end{IEEEproof}  

We note that when $|A\setminus \hat{G}(A)|=1$ the right-hand side of condition \eqref{mchat} is equal to $1$, and thus \eqref{mchat} reduces to the requirement of case (ii). However, the proof of case (iii) requires the stronger moment condition \eqref{epp1}. Theorem \ref{pi_conv} provides conditions for the complete  convergence under $\Pro_A$ of the empirical sampling frequency of each source $i$ to a number $c_i(A)$ that is not necessarily greater than or equal to the value $c_i^*(A)$ that is required for asymptotic optimality under $\Pro_A$ according to Proposition \ref{crit:AO}. If we further impose the following lower bound on $\hat{N}(A)-|\hat{G}(A)|$, i.e.,
\begin{equation*}
  \sum_{i \in A \setminus \hat{G}(A)} c^{*}_{i}(A) \leq \hat{N}(A)-|\hat{G}(A)|,
\end{equation*}
then by the property \eqref{A2} for $\{ c^{*}_{i}(A) \,:\, i \in A \setminus \hat{G}(A) \}$, and relation \eqref{i0} for $\{ c_{i}(A) \, :\, i \in A \setminus \hat{G}(A) \}$, we deduce that $c_{i}(A) \geq c^*_{i}(A)$ for all $i \in A \setminus \hat{G}(A)$. We proceed to the proof of Proposition \ref{prop_ord_ao}.

\begin{IEEEproof}[Proof of Proposition \ref{prop_ord_ao}]
It suffices to show that under conditions \eqref{ng_shat}-\eqref{ng_scheck} the rule $R$ is asymptotically optimal for any $A \in \cP_{\ell, u}$. Let us fix $A \in \cP_{\ell, u}$. By Theorem \ref{pi_conv}, in order to show that $\pi_{i}(n) \to c_i(A)$ $\Pro_{A}$-completely, it suffices to show that $R$ is $z$-quickly consistent for all $z \in [1, \mathfrak{p}/2 - 1)$. Then, by the lower bound in \eqref{ng_shat}, the property \eqref{A2} and the relation \eqref{i0} we deduce that $c_{i}(A) \geq c^*_{i}(A)$ for all $i \in A \setminus \hat{G}(A)$. In order to show that $R$ is $z$-quickly consistent for all $z \in [1, \mathfrak{p}/2 - 1)$, we observe that for any  $D \in \cP_{\ell, u}$ such that $\hat{G}(D) \neq D$ and $x(D)>0$, it holds $c^*_{i}(D) > 0$ for all $i \in D$ by \eqref{xy}, which further implies that the lower bound in \eqref{ng_shat} is positive, i.e., $\hat{N}(D)-|\hat{G}(D)| >0$ (resp. $\check{N}(D)-|\check{G}(D)| >0$). Therefore, the conditions \eqref{NGh_cons}, \eqref{NGch_cons} are satisfied, and by Corollary \ref{consi_deter} we deduce the claim.
\end{IEEEproof}

\subsection{Stabilized ordering sampling rules}\label{stab_ord}
One of the requirements of Theorem \ref{pi_conv} is that the ordering rule $R$ is $z$-quickly consistent for all $z \in [1, \mathfrak{p}/2 - 1)$ for sufficiently large $\mathfrak{p}$. Under this assumption, we can show that the empirical sampling frequency of a source $i \in A$ for rule $R$, i.e., $\pi^{R}_{i}(n)$, converges $\Pro_{A}$-completely to the same limit as that of a \textit{stabilized} ordering sampling rule that behaves in the same way as $R$ on the event $\{ \Delta^{R}_{n} = A, \;  \forall \, n \in \bN \}$. In this subsection, we fix $A \in \cP_{\ell, u}$, and an ordering sampling rule $R$. We also restore the superscript $R$ to point out the dependence of the respective quantities on $R$. 

By the definition of an ordering sampling rule it follows that for every $n \in \bN_{0}$, $R(n+1)$ is conditionally independent of $\cF^{R}_{n}$ given the estimate of anomalies $\Delta^R_n$, and the vector $\mathbf{w}^R(n):=( w_{1}^{R}(n), \ldots, w^{R}_{M}(n))$ defined in \eqref{w_ord} which keeps the ordering of the LLRs $\{\Lambda_i^{R}(n) \, :\, i \in [M]\}$. By \cite[Prop. 8.20]{kal02}, the former property implies that there is a measurable function $h\,:\, \mathcal{P}_{\ell,u}\times \mathfrak{s}_{M} \times [0,1]^2\to 2^{[M]}$, where $\mathfrak{s}_{M}$ is the set of all permutations of $[M]$, such that
\begin{equation}\label{repre}
	R(n+1)=h \left(\Delta^R_n,\mathbf{w}^{R}(n), \mathbf{Z}_{n} \right), \quad \forall \; n \in \bN_{0},
\end{equation}
where $\mathbf{Z}_{n}:=(\hat{Z}_{n},\check{Z}_{n})$ contains the two independent Uniform[0,1] random variables used for randomization purposes. We define the stabilized ordering rule $R^{A,m}$ as a rule which is identical to $R$ after time $m$ on the event $\{s_{A}=m\} = \{ \Delta_u^R = A, \; \forall \, u \geq m\}$, i.e., $\Delta_u^R$ is ``stabilized" to $A$ for all $u \geq m$. Specifically,
\begin{equation}\label{repreA}
 R^{A,m}(u+1) := h \left(A,\mathbf{w}^{A,m}(u), \mathbf{Z}_{m+u} \right), \quad \forall \; u \in \bN_{0},
\end{equation}
where $\mathbf{w}^{A,m}(u)$ keeps the ordering of the LLRs $\{ \Lambda_i^{A,m}(u) \, :\, i \in [M] \}$ defined as
\begin{equation}\label{llrA}
\Lambda_i^{A,m}(u):= \Lambda^{R}_i(m) + \sum_{k=1}^u  \log \left( \frac{f_{1i} (X_i(m+k))}{ f_{0i} (X_i(m+k))} \right) \, R^{A,m}_i(k), \quad i \in [M],
\end{equation}
where we perform sampling according to $R^{A,m}$, and the $\{\Lambda^{R}_i(m) : i \in [M]\}$ play the role of the initial values of $\{ \Lambda_i^{A,m}(u) \, :\, i \in [M] \}$ for $u=0$. By the definition of $R^{A,m}$, we deduce that on the event $\{s_{A} \leq m\}$ the sampling rule $R$ is equal to the rule $R^{A,m}$ at any time $n \geq m$,
\begin{equation*}
 \begin{aligned}
  \{s_{A} \leq m\} \subseteq \{\Delta^{R}_{n} = A, \;  \forall \, n\geq m  \} \subseteq \{R(n)= R^{A,m}(n-m), \;  \forall \, n\geq m  \}.
 \end{aligned}
\end{equation*}
In the following proposition, we show that if $R$ is $z$-quickly consistent for all $z \in [1, \mathfrak{p}/2 - 1)$, and the ``stabilization" happens early compared to $n$, i.e., $m << n$, then $\pi^{R}_{i}(n)$ is approximately equal to $\pi^{A,m}_{i}(n-m)$ for large $n$, where 
\begin{equation}\label{piA}
 \pi_i^{A,m}(u):= \frac{1}{u} \sum_{k=1}^u  R^{A, m}_i(k), \quad u \in \bN,
\end{equation}
is the empirical sampling frequency for the rule $R^{A,m}$. In particular, we consider the increasing sequence of integers $\{\zeta_n : n \in \bN\}$ such that 
\begin{equation}\label{zn1}
 \zeta_n \leq n, \;\; \forall \, n \in \bN, \quad  \zeta_n/n \to 0, \quad \mbox{and} \quad \exists \; z \in [1, \mathfrak{p}/2 - 1)\; \mbox{ s.t. }\; \sum_{n=1}^\infty \frac{1}{\zeta_n^{z}} < \infty
\end{equation}
e.g., $\zeta_n = \lceil n^a \rceil$ for some $a \in (1/z,1)$, and we prove the following result.

\begin{proposition}\label{pr_tpa}
Suppose that the moment condition \eqref{mc1} holds for $\mathfrak{p} > 4$, and let $R$ be a $z$-quickly consistent under $\Pro_A$ for all $z \in [1, \mathfrak{p}/2 - 1)$. Then,
\begin{align*}
	\left|\pi^{R}_i(n)  -\pi_i^{A, \zeta_n}(n-\zeta_n) \right|  \to 0 \quad \Pro_A-\text{completely}, \quad \forall \; i \in [M],
\end{align*}
where $\{\zeta_n : n \in \bN\}$ is an increasing sequence of integers that satisfies \eqref{zn1}.
\end{proposition}

\begin{IEEEproof}
Let us fix $\epsilon >0$. By Boole’s inequality,
\begin{equation*}
 \begin{aligned}
   \Pro_{A}\left(|\pi^{R}_i(n)  -\pi_i^{A, \zeta_n}(n-\zeta_n)| > \epsilon \right) \leq \Pro_{A}\left(|\pi^{R}_i(n)  -\pi_i^{A, \zeta_n}(n-\zeta_n)| > \epsilon, s_A < \zeta_n \right) + \Pro_{A}(s_A \geq \zeta_n ).
 \end{aligned}
\end{equation*}
By Markov's inequality, for any $z \in [1, \mathfrak{p}/2 - 1)$ such that \eqref{zn1} holds, we have
\begin{equation*}
 \Pro_{A}(s_A \geq \zeta_n ) \leq \frac{\Exp_{A}[s^z_A]}{\zeta^{z}_n}
\end{equation*}
which is summable because $\Exp_{A}[s^z_A] < \infty$, and \eqref{zn1}. Therefore, in order to show the claim it suffices to show that $\{ \Pro_{A}(|\pi^{R}_i(n)  -\pi_i^{A, \zeta_n}(n-\zeta_n)| > \epsilon, s_A < \zeta_n) \, :\, n \in \bN \}$ is summable. We observe that 
\begin{equation*}
   \pi^{R}_i(n) = \frac{1}{n}  \sum_{u=1}^{n} R_i(u)=  
   \frac{1}{n}\sum_{u=1}^{\zeta_n} R_i(u) + \left( 1- \frac{\zeta_n}{n}  \right) \; \frac{1}{n-\zeta_n}
   \sum_{u=\zeta_n+1}^{n} R_i(u)
\end{equation*}
and on the event $\{ s_A < \zeta_n \}$, it holds 
\begin{equation*}
\frac{1}{n-\zeta_n}\sum_{u=\zeta_n+1}^{n} R_i(u)=\frac{1}{n-\zeta_n}\sum_{u=1}^{n-\zeta_n} R^{A, \zeta_n}_i(u)= \pi_i^{A, \zeta_n}(n-\zeta_n)
\end{equation*}
which implies that
\begin{equation}\label{zetn}
  \begin{aligned}
  	|\pi^{R}_i(n) - \pi_i^{A, \zeta_n}(n-\zeta_n)| \leq \frac{\zeta_n}{n} + \frac{\zeta_n}{n}\pi_i^{A, \zeta_n}(n-\zeta_n) \leq \frac{2\zeta_n}{n},
  \end{aligned}
\end{equation}
where for the first inequality we used the fact that $R_i(u) \leq 1$ for all $u \in [1, \zeta_n]$. Therefore,
\begin{equation*}
\Pro_{A}\left(|\pi^{R}_i(n)  -\pi_i^{A, \zeta_n}(n-\zeta_n)| > \epsilon, s_A < \zeta_n \right) \leq \Pro_{A}\left(\frac{\zeta_n}{n}  > \epsilon/2 \right),
\end{equation*}
and by \eqref{zn1} there is $\mathcal{M}>0$ such that for all $n \geq \mathcal{M}$ it holds $\zeta_n/n  < \epsilon/2$, which proves the claim.
\end{IEEEproof}

Proposition \ref{pr_tpa} suggests that in order to prove Theorem \ref{pi_conv}, it suffices to show that $\pi_i^{A, \zeta_n}(n-\zeta_n) \to c_i(A)$ $\Pro_A$-completely, for all $i \in A$. The first two cases are simpler, and we show them first.

\begin{IEEEproof}[Proof of Theorem \ref{pi_conv} (i), (ii)]
For case (i), by definition the sampling rule $R^{A,\zeta_n}$ samples for sure each source $i \in \hat{G}(A)$ at each instant $m \in [\zeta_n, n]$, which implies that $\pi_i^{A, \zeta_n}(n-\zeta_n)=1$ for all $n \in \bN$, for each $i \in \hat{G}(A)$. For case (ii), by definition the sampling rule $R^{A,\zeta_n}$ samples the single source $i \in A \setminus \hat{G}(A)$ with probability $\hat{N}(A) - |\hat{G}(A)|$ at each instant $m \in [\zeta_n, n]$, i.e.,
\begin{equation*}
\pi_i^{A, \zeta_n}(n-\zeta_n) = \frac{1}{n-\zeta_n}	\sum_{u=\zeta_n +1}^{n} \mathbf{1}\left\{ \hat{Z}_u \leq \hat{N}(A)- |\hat{G}(A)|\right\},
\end{equation*}
and since $\{ \hat{Z}_u \,: \, u \in \bN_0 \}$ is a sequence of iid Uniform[0,1] random variables, by Chernoff bound we show that $\pi_i^{A, \zeta_n}(n{-}\zeta_n) {\to} \hat{N}(A){-} |\hat{G}(A)|$, $\Pro_A$-completely, which by Proposition \ref{pr_tpa} proves the claim.
\end{IEEEproof}

For part (iii) of Theorem \ref{pi_conv}, the property \eqref{i0} suggests that we need to show that for all $i,j \in A \setminus \hat{G}(A)$,
\begin{equation*}
 \Big|I_i \,\pi_i^{A, \zeta_n}(n-\zeta_n) -  I_j \, \pi_j^{A, \zeta_n}(n-\zeta_n) \Big| \to 0  \quad \Pro_{A}-\mbox{completely}.
\end{equation*}
Since $\Lambda_i^{A,\zeta_n}(n-\zeta_n) \simeq (n-\zeta_n) I_i  \pi_i^{A, \zeta_n}(n-\zeta_n)$ for large $n$, it suffices to show that
\begin{equation*}
	\frac{1}{n-\zeta_n} \,  |\Lambda^{A,\zeta_n}_{i} (n-\zeta_n)  - \Lambda^{A,\zeta_n}_j (n-\zeta_n) | \to 0 \quad \Pro_A- \text{completely} \quad \forall \; i, j \in	A \setminus \hat{G}(A).
\end{equation*}
The key step towards that direction is to show that the LLRs of the rule $R^{A,m}$ of any two sources in $A \setminus \hat{G}(A)$ stay close, as described in the following theorem. This property follows from the fact that an ordering sampling rule prioritizes the sampling of the sources with small LLRs among those in $A \setminus \hat{G}(A)$.  

\begin{theorem}\label{prV}
	Let $A \in \cP_{\ell, u}$, $m \in \bN$, and let $R$ be an ordering sampling rule. Suppose that $A \setminus \hat{G}(A) \geq 2$, the $\hat{N}(A)$, $\hat{G}(A)$ satisfy \eqref{mchat}, and the moment condition \eqref{mc1} holds for 
	\begin{equation}\label{pfrk}
	 \mathfrak{p}>3\cdot 2^{\lceil \hat{N}(A) \rceil - |\hat{G}(A)|-1}+1.
	\end{equation}
	If 
	\begin{equation}\label{ccd}
		\max_{i,j \in A \setminus \hat{G}(A)} |\Lambda_i^R(m) - \Lambda_j^R(m)| \in \mathcal{L}^{\mathfrak{p}-1},
	\end{equation}
	then there exists a strictly increasing sequence of $\Pro_{A}$-a.s. finite stopping times $\{ \sigma_l\,:\, l \in \bN_{0}\}$, with $\sigma_0:=0$, such that for the distance of the LLRs of $R^{A,m}$ during each $[\sigma_l,\sigma_{l+1})$, i.e.,
	\begin{equation}\label{i1}
		V_l := \max\limits_{\sigma_l\leq u< \sigma_{l+1}} \, \max\limits_{i,j\in A\setminus \hat{G}(A)}\big{|}\Lambda^{A,m}_{i}(u)-\Lambda^{A,m}_{j}(u)\big{|} , \quad \; l  \in \bN_{0},
	\end{equation}
	there is a constant $C>0$ independent of $m,l$ such that
	\begin{equation}\label{vlth}
	 \sup_{l \in \bN_0} \Exp_A \left[ V^{3+\theta}_l \right] \leq C\, \left(1 + \Exp_A \left[  \max_{i,j \in A \setminus \hat{G}(A)} \left| \Lambda^R_i(m) - \Lambda^R_j(m) \right|^{3+\theta^+} \right] \right)
	\end{equation}
	for any $\theta ,\theta^+ >0$, such that $\theta < \theta^+ < (\mathfrak{p}-1)/2^{\lceil \hat{N}(A) \rceil - |\hat{G}(A)|-1} -3$.
\end{theorem}

\begin{IEEEproof}
 The proof is presented in Appendix \ref{vlbg11}. Since $R$ prioritizes the sampling of the sources with small LLRs, we provide a recursive definition of $\{ \sigma_l\,:\, l \in \bN_{0}\}$, which enables us to prove the claim.
\end{IEEEproof}

For the following lemma, we fix $\theta ,\theta^+$ such that 
\begin{equation*}
0<\theta < \theta^+ < (\mathfrak{p}-1)/2^{\lceil \hat{N}(A) \rceil - |\hat{G}(A)|-1} -3,\;\; \mbox{ and }\;\; 3>(3+\theta^+)/(1+\theta),
\end{equation*}
and we impose the following stronger assumption on $\{\zeta_n : n \in \bN\}$, which implies \eqref{zn1}, i.e.,
\begin{equation}\label{zn2}
	\zeta_n \leq n, \;\; \forall \, n \in \bN, \quad \sum_{n=1}^{\infty} \frac{\zeta_n^{3+\theta^+}}{n^{2+\theta}} < \infty, \quad \mbox{and} \quad \exists \; z \in [1, \mathfrak{p}/2 - 1)\; \mbox{ s.t. }\; \sum_{n=1}^\infty \frac{1}{\zeta_n^{z}} < \infty.
\end{equation}

\begin{lemma}\label{sublin}
Let $A \in \cP_{\ell, u}$, and let $R$ be a $z$-quickly consistent ordering rule under $\Pro_A$ for all $z \in [1, \mathfrak{p}/2 - 1)$. Suppose that $|A \setminus \hat{G}(A)| \geq 2$, the $\hat{N}(A)$, $\hat{G}(A)$ satisfy \eqref{mchat}, and the moment condition \eqref{mc1} holds for \eqref{epp1}. Then, for $\{\zeta_n : n \in \bN\}$ that satisfies \eqref{zn2}, it holds
\begin{equation} \label{close}
	\frac{1}{n-\zeta_n} \,  |\Lambda^{A,\zeta_n}_{i} (n-\zeta_n)  - \Lambda^{A,\zeta_n}_j (n-\zeta_n) | \to 0 \quad \Pro_A- \text{completely} \quad \forall \; i, j \in	A \setminus \hat{G}(A).
\end{equation} 
\end{lemma}

\begin{IEEEproof}
By \eqref{epp1}, we have $\mathfrak{p}>8 \Leftrightarrow \mathfrak{p}/2 - 1 >3$, and thus we can fix $z \in ((3+\theta^+)/(1+\theta) , 3)$. The sequence $\{ \zeta_n = \lceil n^{\delta} \rceil \, :\, n \in \bN \}$, where $\delta \in (1/z, (1+\theta)/(3+\theta^+))$ satisfies \eqref{zn2}. To show \eqref{close}, we fix $\epsilon >0$, and $n \in \bN$. By Markov's inequality, we have
\begin{equation}\label{o}
 \Pro_{A}\left( \frac{|\Lambda^{A,\zeta_n}_{i} (n-\zeta_n) - \Lambda^{A,\zeta_n}_j (n-\zeta_n)|}{n-\zeta_n} > \epsilon \right) \leq \frac{\Exp_{A}\left[\big{|}\Lambda^{A,\zeta_n}_{i}(n-\zeta_n)-\Lambda^{A,\zeta_n}_{j}(n-\zeta_n)\big{|}^{3+\theta} \right]}{ \epsilon^{3+\theta} (n-\zeta_n)^{3+\theta}},
\end{equation}
and it suffices to provide an upper bound on the expectation so that the right-hand side is summable. In view of Theorem \ref{prV}, we first note that condition \eqref{ccd} is satisfied, because by Lemma \ref{lem:C5} there is a constant $C_0 >0$ such that for all $n \in \bN$,
\begin{equation*}
\Exp_{A}\left[\left|\Lambda_{i}^{R}(\zeta_n)-\Lambda_{j}^{R}(\zeta_n)\right|^{\mathfrak{p}-1}\right] \leq C_0 \, (\zeta_n)^{\mathfrak{p}-1} < \infty.
\end{equation*}
Therefore, by definition of $\{V_l \, :\, l \in \bN \}$ in \eqref{i1}, we deduce that
\begin{equation*}
|\Lambda^{A,\zeta_n}_{i} (n-\zeta_n) - \Lambda^{A,\zeta_n}_j (n-\zeta_n)| \leq V_{l^*_n},
\end{equation*}
where $l^*_n$ is the increasing number of the interval $[\sigma_l, \sigma_{l+1})$ where $n-\zeta_n$ belongs, i.e. $l^*_n:=\max\{l\in \bN_0 :\; \sigma_l \leq n-\zeta_n\}$. Since $l^*_n \leq n-\zeta_n$ a.s., we further have
\begin{equation}\label{isx}
	\begin{aligned}
		\Exp_{A}\left[\big{|}\Lambda^{A,\zeta_n}_{i}(n{-}\zeta_n){-}\Lambda^{A,\zeta_n}_{j}(n{-}\zeta_n)\big{|}^{3+\theta}\right] 
		&{\leq} \sum_{l=0}^{n-\zeta_n} \Exp_{A}\left[V^{3+\theta}_l\right]\\
		&{\leq} (n{-}\zeta_n)C\left(1{+}E_A \big[\max_{i,j \in A \setminus \hat{G}(A)} \left| \Lambda_i^R(\zeta_n) {-} \Lambda_j^R(\zeta_n) \right|^{3+\theta^+} \big] \right) \\
		&{\leq} (n{-}\zeta_n)C\left(1+\, C_0\, \zeta^{3+\theta^+}_n \right),
	\end{aligned}
\end{equation}
where the second inequality follows by Theorem \ref{prV}, and the third by Lemma \ref{lem:C5}, and $C,C_0 >0$ are constants independent of $n$. The summability of \eqref{o} follows by assumption \eqref{zn2}.
\end{IEEEproof}

Based on Lemma \ref{sublin}, we provide the proof of Theorem \ref{pi_conv}(iii).

\begin{IEEEproof}[Proof of Theorem \ref{pi_conv}(iii)]
	In view of Proposition \ref{pr_tpa}, it suffices to that 
		\begin{equation}
		\pi_i^{A,\zeta_n}(n - \zeta_n) \to c_i(A) \quad \Pro_A\text{-completely},
	\end{equation}
	where $\{c_{i}(A)\,:\, i \in A \setminus \hat{G}(A)\}$ are given by \eqref{i0}, and $\{\zeta_n \, :\, n \in \bN \}$ satisfies \eqref{zn2}. By the Definition \ref{definition_ordering} of an ordering rule, we observe that for the rule $R^{A,\zeta_n}$, for every $n \in \bN$ we have  
	\begin{equation}\label{su_con}
		\sum_{i \in A \setminus \hat{G}(A)} \pi_i^{A,\zeta_n}(n-\zeta_n) =  \lfloor \hat{N}(A) \rfloor-|\hat{G}(A)| + \frac{1}{n-\zeta_n} \sum_{u=1}^{n-\zeta_n} \mathbf{1}\left\{ \hat{Z}_{u+\zeta_n} \leq \hat{N}(A)-\lfloor \hat{N}(A) \rfloor \right\}.
	\end{equation}
	Since  $\{\hat{Z}_m\,:\, m \in \bN_0\}$ is a sequence of iid Uniform[0,1] random variables, by the Chernoff bound it follows that the average term converges completely under $\Pro_A$ to  $\hat{N}(A)  - \lfloor \hat{N}(A) \rfloor$, and consequently,
	\begin{equation}\label{n}
		\sum_{ i \in A \setminus \hat{G}(A)} \pi_i^{A,\zeta_n}(n-\zeta_n) \to \hat{N}(A)-|\hat{G}(A)| \quad \Pro_{A}-\mbox{completely}.
	\end{equation}
	Therefore, to prove the claim it remains to show that for any $\epsilon>0$ and  $i,j \in A \setminus \hat{G}(A)$, the sequence
	\begin{equation*}
		\left\{\Pro_{A}\left(|I_{i}\pi_i^{A,\zeta_n}(n-\zeta_n)-I_{j}\pi^{R^{A,\zeta_n}}_{j}(n-\zeta_n)|> \epsilon \right)\,:\; n \in \bN \right\} 
	\end{equation*}
	is summable. We fix $\epsilon >0$, $i,j \in A \setminus \hat{G}(A)$, and for any $n \in \bN$ we have
	\begin{equation}\label{ieq_decomp_2}
		\begin{aligned}
			\Pro_{A}\left(|I_{i}\pi_{i}^{A,\zeta_n}(n{-}\zeta_n){-}I_{j}\pi_{j}^{A,\zeta_n}(n{-}\zeta_n)|{>}\epsilon\right)
			\leq & \Pro_{A}\left(|{\Lambda}^{A,\zeta_n}_{i}(n{-}\zeta_n){-}{\Lambda}^{A,\zeta_n}_{j}(n{-}\zeta_n)|{>}\frac{\epsilon}{2}  (n-\zeta_n)  \right)\\
			&+\Pro_{A}\left(|\widetilde{\Lambda}^{A,\zeta_n}_{i}(n{-}\zeta_n){-}\widetilde{\Lambda}^{A,\zeta_n}_{j}(n{-}\zeta_n)|{>}\frac{\epsilon}{2} (n{-}\zeta_n) \right),
		\end{aligned}
	\end{equation}
	where
	\begin{equation*}
	 \widetilde{\Lambda}^{A,\zeta_n}_k(n-\zeta_n) := \Lambda^{A,\zeta_n}_k(n-\zeta_n) - (n-\zeta_n) I_i \, \pi^{A,\zeta_n}_k(n-\zeta_n),\quad \forall\, k \in A \setminus \hat{G}(A).
	\end{equation*}
	The first term of \eqref{ieq_decomp_2} is summable by Lemma \ref{sublin}, and the second by Lemma \ref{tilA} since
	\begin{equation*}
	 \begin{aligned}
	 &\Pro_{A}\left(|\widetilde{\Lambda}^{A,\zeta_n}_{i}(n-\zeta_n)-\widetilde{\Lambda}^{A,\zeta_n}_{j}(n-\zeta_n)|>\frac{\epsilon}{2} \, (n-\zeta_n) \right)\\
	 &\leq \Pro_{A}\left(|\widetilde{\Lambda}^{A,\zeta_n}_{i}(n-\zeta_n)|>\frac{\epsilon}{4} \, (n-\zeta_n) \right) + \Pro_{A}\left(|\widetilde{\Lambda}^{A,\zeta_n}_{j}(n-\zeta_n)|>\frac{\epsilon}{4} \, (n-\zeta_n) \right).
	 \end{aligned}
	\end{equation*}
\end{IEEEproof}

\section{Comparison with existing ordering sampling rules}\label{sec:special}
In this section, we compare our default ordering sampling rule introduced in Subsections \ref{default}, with the ordering rules presented in \cite{Cohen2015active}, \cite{huang2017active}. In both papers \cite{Cohen2015active}, \cite{huang2017active}, the authors assume that the number of anomalous sources is known a priori, i.e., $\ell =u$, and that $K$ is an integer, and they focus their analysis on the special case where $\ell =u=1$, and $K=1$. In \cite{Cohen2015active}, the authors assume homogeneous sources, i.e.,
\begin{equation}\label{hmg}
	I_i = I \qquad \text{ and }  \qquad J_i = J, \qquad \forall \; i \in [M],
\end{equation}
and for any integer $K \geq 1$, they introduce an asymptotically optimal ordering rule. In \cite{huang2017active}, they consider heterogeneous sources, i.e., \eqref{hmg} does not hold, and for $K=1$ they introduce an asymptotically optimal ordering rule, whereas for any integer $K>1$ they provide a conjecture \cite[(39)]{huang2017active}. 

For $\ell=u$, and $K$ integer in the homogeneous setup, and for $\ell=u=1$ in the heterogeneous setup, we do not need to define a set of sources that are sampled with probability $1$ because for all $D \in \cP_{\ell, u}$ either $\hat{G}(D)=\emptyset$ or $\hat{N}(D)=|D| \Leftrightarrow \hat{G}(D)=D$ (resp. for $\check{G}(D)$), i.e.,
\begin{equation}\label{ghd}
\hat{G}(D) \in \{\emptyset, D\}, \quad \check{G}(D) \in \{\emptyset, D^c\}, \quad \forall \, D \in \cP_{\ell, u},
\end{equation}
and there is no need for randomization because $\hat{N}(D)$, $\check{N}(D)$ are integers for all $D \in \cP_{\ell, u}$, as shown in the following subsections. In view of Definition \ref{definition_ordering} our ordering sampling rule becomes
\begin{equation}\label{rn1}
 R(n+1) = \{ w_{\ell -\hat{N}(\Delta_n) +1}(n),\ldots, w_{\ell +\check{N}(\Delta_n)}(n) \}, \quad n \in \bN_0,
\end{equation}
where $\Delta_n$ is defined in \eqref{gap_decision rule}, which means that at time $n+1$ we sample the sources with the $\hat{N}(\Delta_n)$ smallest LLRs in $\Delta_n$, and the $\check{N}(\Delta_n)$ largest LLRs in $\Delta^c_n$. We show that the ordering sampling rule suggested in \cite{Cohen2015active}, \cite{huang2017active} coincides with \eqref{rn1} in each case.

\subsubsection{Homogeneous setup}
In case $\ell=u$, $K$ is an integer, and \eqref{hmg} holds, then for all $D \in \cP_{\ell, u}$ by the form of $x(D)$, $y(D)$ in \eqref{kn_c1}-\eqref{kn_c2} and for the default choice \eqref{N_default} for $\hat{N}(D)$, $\check{N}(D)$ we have
\begin{itemize}
	\item if  $(M-\ell) I \geq  J \ell$, then $\hat{N}(D)=\min \{K, \ell\}$ and $\check{N}(D)=(K-\ell)^+$,
	\item if   $(M-\ell) I <  J \ell$,  then $\hat{N}(D)= (K-M+\ell)^+$ and $\check{N}(D)= \min \{K, M-\ell\}$,
\end{itemize}
and our rule \eqref{rn1} takes the following form,
\begin{itemize}
	\item  if  $(M-\ell) I \geq  J \ell$, then 
	\begin{equation}\label{r1}
		R(n+1)=
		\begin{cases}
			\{ w_{l-K+1}(n), \ldots,w_{l}(n) \}, \quad &\mbox{if } K < \ell,\\
			\{ w_{1}(n), \ldots, w_{K}(n) \}, \quad &\mbox{if } K \geq \ell.
		\end{cases}
	\end{equation} 
	which means that when $K < \ell$, our rule $R$ samples the $K$ sources in $\Delta_n$ with the smallest LLRs, and when $K\geq \ell$, it samples all sources in $\Delta_n$ and the $K-\ell$ sources in $\Delta^c_n$ with the largest LLRs.
	
	\item If  $(M-\ell) I <  J \ell$, then 
	\begin{equation}\label{r2}
		R(n+1)= \begin{cases}
			\{ w_{l+1}(n), \ldots,w_{l+K}(n) \}, \quad &\mbox{if } K \leq M-l,\\
			\{ w_{M-K+1}(n),\ldots,w_{M}(n) \}, \quad &\mbox{if } K > M-l,
		\end{cases}
	\end{equation} 
	which means that when $K {\leq} M {-} \ell$, $R$ samples the $K$ sources in $\Delta^c_n$ with the largest LLRs, and when $K {>} M{-}\ell$, it samples all sources in $\Delta^c_n$ and the  $K {-} (M {-} \ell)$ sources in $\Delta_n$ with the smallest LLRs.
\end{itemize}
The rule $R$ in \eqref{r1}-\eqref{r2} is the same as the ordering rule presented in \cite[(17)-(18)]{Cohen2015active}.

\subsubsection{Heterogeneous setup}
In case $\ell=u=1$, and \eqref{hmg} does not hold, we consider the subcases $K=1$, and $K>1$ separately. In subcase $K=1$, for any $D \in \cP_{\ell, u}$ by the form of $x(D)$, $y(D)$ in \eqref{kn_c1}-\eqref{kn_c2} and for the default choice \eqref{N_default} for $\hat{N}(D)$, $\check{N}(D)$ we have
\begin{equation*}
\begin{aligned}
	\hat{N}(D) &=1 \quad \text{and} \quad \check{N}(D)=0, \quad \text{if} \quad 
	\sum_{i \in D} 1/I_i \leq 	 \sum_{i \notin D} 1/J_i, \\
	\hat{N}(D) &=0 \quad \text{and} \quad \check{N}(D)=1,  \quad \text{otherwise},
\end{aligned}
\end{equation*}
and our rule \eqref{rn1} takes the following form,
\begin{equation}\label{r4}
	R(n+1)=
	\begin{cases}
		w_{1}(n), \quad &\mbox{if } \sum_{i \in D} 1/I_i \leq 	 \sum_{i \notin D} 1/J_i,\\
		w_{2}(n), \quad &\mbox{otherwise,} 
	\end{cases}
\end{equation}  
which means that when $\sum_{i \in D} 1/I_i \leq \sum_{i \notin D} 1/J_i$, our rule $R$ samples the single source in $\Delta_n$, otherwise it samples the source in $\Delta^c_n$ with the largest LLR. The rule $R$ in \eqref{r4} is the same as the one in \cite[(12)]{huang2017active}.

For $K>1$, the authors in \cite[(39)]{huang2017active} conjecture that, if for all $D \in \cP_{\ell, u}$, $\hat{N}(D)$, $\check{N}(D)$ are integers, then the rule \eqref{rn1} is asymptotically optimal. Indeed, if we further assume \eqref{ghd}, then their conjecture is true, however their assumption is very restrictive and in most cases it is not satisfied by heterogeneous sources. In our design asymptotic optimality is achieved, even if $\hat{N}(D)$, $\check{N}(D)$ are not integers for some $D \in \cP_{\ell, u}$ thanks to the randomization introduced by $\hat{Z}_n$, $\check{Z}_n$.

\section{Simulation Study}\label{sec: simulations}
In this section, we present a simulation study in which we compare the expected stopping time, i.e., $E_{A}[T]$, of the default ordering sampling rule presented in Subsection \ref{default} where  $\hat{N}, \check{N}$, $\hat{G}, \check{G}$ satisfy \eqref{N_default}, \eqref{G_default}, with the expected stopping time of a probabilistic sampling rule that samples each source $i \in [M]$ with probability $c^*_i\left(D\right)$, whenever the estimated subset of anomalous sources is $D$, i.e. $\Delta_n =D$.

For every $i \in [M]$, we  set  $f_{0i}:=\mathcal{N}(0,1)$ and $f_{1i}:=\mathcal{N}(\mu_i,1)$,  i.e., all observations from source $i$ are Gaussian with  variance $1$   and mean equal to  $\mu_i$ if the source is anomalous, and 0 otherwise, and as a result,  $I_{i}=J_{i}=(\mu_{i})^2 /2$.  We consider a  homogeneous setup, i.e., $\mu_{i}=\mu$ for all $i \in [M]$, as well as a heterogeneous setup, where
\begin{equation*}
	\mu_{i}=
	\begin{cases}
		\mu,\quad & 1 \leq i\leq M/2,\\
		2 \mu ,\quad & M/2 < i \leq M.
	\end{cases}
\end{equation*}
In both setups, we set $\mu=0.5$, $\alpha=\beta=10^{-3}$, $M=10$,  $K=5$, and $\ell =1$,  $u=6$, i.e., $1\leq |A| \leq 6$. The thresholds of the stopping times are selected, via simulations, so that the familywise error probability of each kind in \eqref{err_const} is approximately equal to $10^{-3}$. For each possible value of the underlying (but unknown) number of anomalous sources $|A| \in \{1, \ldots, 6\}$, we compute the expected stopping time for each one of the two sampling rules, when the true (but unknown) subset  of anomalous sources is of the form $A:=\{1, \ldots, |A|\}$. In all cases, the  Monte Carlo error for each estimated expected value is $10^{-1}$. 
\vspace{-2.5 em}
	\begin{figure}[h]
		\subfloat[Homogeneous setup]{
			\includegraphics[width=0.5\linewidth]{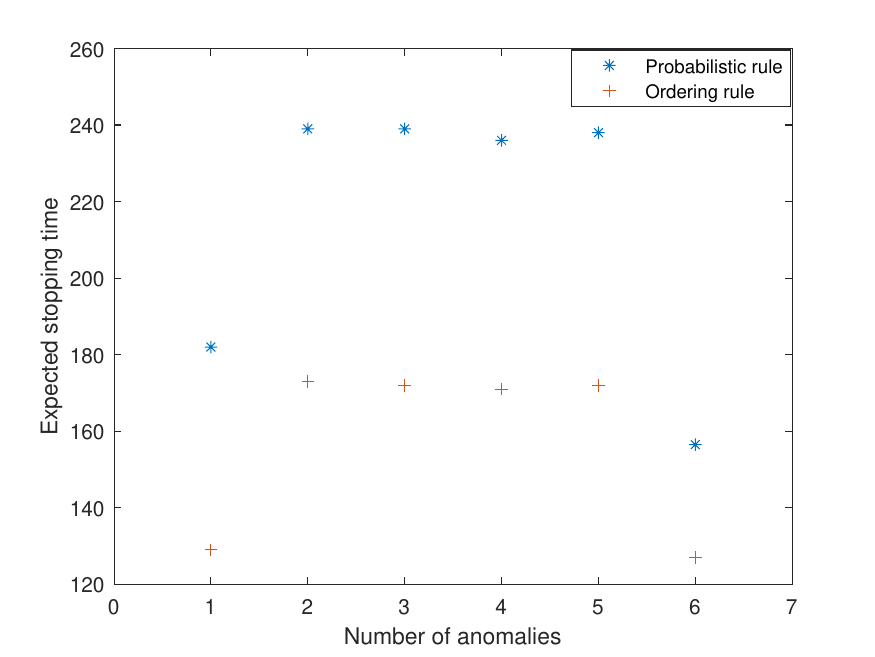}  
			\label{fig:homog_nonasym}
		}
		\subfloat[Heterogeneous setup]{
			\includegraphics[width=0.5\linewidth]{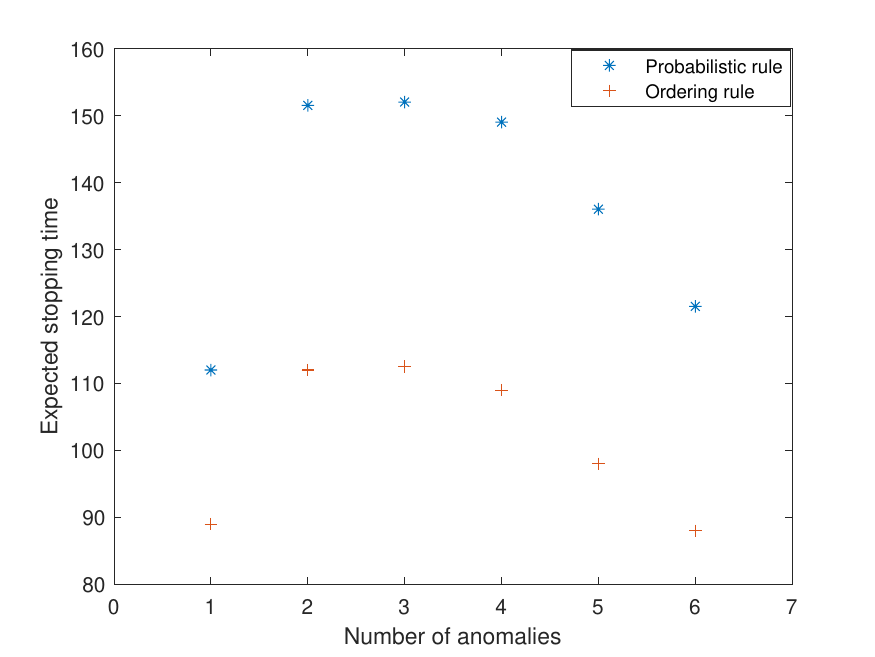}
			\label{fig:nonhomog_nonasym} 
		} 
		\captionsetup{justification=raggedright, singlelinecheck=false}
		\caption{Expected stopping time that corresponds to each sampling rule versus the size of $|A|$.}
	\end{figure}

In Figures \ref{fig:homog_nonasym}  and  \ref{fig:nonhomog_nonasym}, we plot the expected stopping times against $|A| \in \{1, \ldots, 6\}$ in the homogeneous and heterogeneous setup. We observe that in both setups and for both sampling rules, the expected stopping time  is much smaller when the number of anomalous sources is equal to either $\ell$ or $u$ than when it is between  $\ell$ and $u$. This is because we encode the prior information on the number of anomalies in the definition of the stopping rule \eqref{gi_decision_rule}, which enables us to stop faster. Moreover, we observe that in all cases, the ordering sampling rule leads to significantly smaller expected stopping time than the corresponding probabilistic sampling rule. Intuitively, this is because the ordering rule prioritizes the sampling of the sources with small LLRs without wasting samples for the sources with already large LLRs, which is impossible with the probabilistic sampling rule as it samples each source with some probability.

\section{Conclusion} \label{sec: conclusion}
We consider the problem of sequential identification of anomalies under a general sampling constraint, where it is not necessary to sample only one source per time instant, and we allow for arbitrary bounds on the number of anomalous data sources. We 
introduce an  ordering sampling rule, whose main feature is that the sources that are selected for sampling at each time are not determined solely based on the currently estimated subset of anomalous sources, but also on the ordering of the LLR statistics. We show that with an  appropriate design, such a sampling rule, combined with appropriate stopping and decision rules, leads to asymptotic optimality. That is, it minimizes the expected time for stopping, under any possible anomalous subset, to a first-order asymptotic approximation as the error rates go to 0. A novel proof technique is developed for this result that covers, for the first time, the case of multiple sampled sources per time  instant, and the case that the number of anomalies is not necessarily known a priori.

We focused our presentation on the case of the usual familywise error metrics \eqref{err_const}, however this assumptions can be relaxed by considering generalized error metrics. According to \cite{He_and_Bartroff_2021}, the asymptotic optimality theory of the present paper remains valid for any error metrics that are bounded up to a multiplicative constant by the corresponding familywise error rates, which we considered in this work. These include, among others, the false discovery rate (FDR) and the false non-discovery rate (FNR) \cite{Bartroff_and_Song_2020}, as well as the respective generalized FDR, FNR \cite{Bartroff_2018}. However, this direct extension is not possible in the case of  \textit{generalized} familywise error rates in \cite{LehmannRomano2005}, where we tolerate a certain number of  false positives and false negatives, respectively.  The problem considered in this paper but for \textit{generalized} familywise error rates, was studied in \cite{Song_and_Fellouris_2019} in the absence of sampling constraints (full sampling), and in \cite{ts2024sequential} under sampling constraints for probabilistic sampling rules. In \cite{ts2024sequential}, the authors compute the appropriate minimum sampling frequencies in the long-run, i.e. $\{c^*_{i}(A)\, :\, i\in [M]\}$, for the problem under \textit{generalized} familywise error rates, and thus considering $\hat{N},\, \check{N}, \hat{G},\, \check{G}$ as in Subsections \ref{default}, \ref{general} but for the adapted $\{c^*_{i}(A)\, :\, i\in [M]\}$ we can extend our results to ordering sampling rules. 

Future directions of the present work include the incorporation of composite hypotheses. To be specific,  suppose  that for each  $i \in [M]$, the density of the observations in source $i$ is assumed to belong to a family
$\{f_{\theta}: \theta \in \Theta_i\}$,  
and that  the source $i$  is  anomalous (resp. regular)  if its density is
$f_{\theta}$ for some $\theta $ in  $\Theta_{i,1}$ (resp. $\Theta_{i,0}$), where $\Theta_{i,0}$ and $\Theta_{1,i}$ are two disjoint subsets of $\Theta_i$. In this case, we must choose an appropriate test statistic compatible to composite hypotheses \cite[Section 5]{tartakovsky_book_2014}. Also, the minimum long-run sampling frequencies, and as a result the functions $\hat{N},\, \check{N}, \hat{G},\, \check{G}$, depend not only on the true anomalous subsets, but also on the true parameter vector, $\boldsymbol{\theta}:=(\theta_{1},\ldots,\theta_{M})$. An asymptotic optimality analysis for our problem, for composite hypotheses and without sampling constraints was presented in \cite{song_fel_supple}, and the special case of a single sampled source per time instant, and a single anomalous source in \cite{Cohen2020composite}. Other research directions include non-uniform sampling cost per observation across different sources as in \cite{Cohen2019nonlinearcost}, hierarchical structure on the sources as in \cite{Cohen_hier}, as well as the application of ordering sampling rules to non-sequential testing problems with adaptive design, as in \cite{kartik}. The proof of second-order asymptotic optimality of an ordering sampling rule for the general setup, as in \cite{Lalley_Lorden_1986} for a single sampled source at each time instant and a single anomalous source, is also an interesting open problem. 

\section*{Acknowledgments}
This work was supported in part by the NSF under grants CIF 1514245 and DMS 1737962, through the University of Illinois at Urbana–Champaign.

\bibliographystyle{ieeetr}
\bibliography{biblio_new}

@article{Aditya_2021,
author = {Aditya Deshmukh and Venugopal V. Veeravalli and Srikrishna Bhashyam},
title = {Sequential controlled sensing for composite multihypothesis testing},
journal = {Sequential Analysis},
volume = {40},
number = {2},
pages = {259-289},
year  = {2021},
publisher = {Taylor & Francis},
doi = {10.1080/07474946.2021.1912525},

URL = { 
        https://doi.org/10.1080/07474946.2021.1912525
    
},
eprint = { 
        https://doi.org/10.1080/07474946.2021.1912525
    
}

}

@article{albert1961,
author = {Arthur E. Albert},
title = {{The Sequential Design of Experiments for Infinitely Many States of Nature}},
volume = {32},
journal = {The Annals of Mathematical Statistics},
number = {3},
publisher = {Institute of Mathematical Statistics},
pages = {774 -- 799},
year = {1961},
doi = {10.1214/aoms/1177704973},
URL = {https://doi.org/10.1214/aoms/1177704973}
}

@Article{Bartroff_and_Song_2014,
  Title                    = {Sequential tests of multiple hypotheses controlling type I and II familywise error rates},
  Author                   = {Bartroff, Jay and Song, Jinlin},
  Journal                  = {Journal of statistical planning and inference},
  Year                     = {2014},
  Pages                    = {100--114},
  Volume                   = {153},
  Publisher                = {Elsevier}
}

@article{Bartroff_2018,
 ISSN = {10170405, 19968507},
 URL = {http://www.jstor.org/stable/26384246},
 author = {Jay Bartroff},
 journal = {Statistica Sinica},
 number = {1},
 pages = {363--398},
 publisher = {Institute of Statistical Science, Academia Sinica},
 title = {Multiple Hypothesis Tests controlling generalized error rates for sequential data},
 urldate = {2023-09-21},
 volume = {28},
 year = {2018}
}

@article{Bartroff_and_Song_2020,
author = {Jay Bartroff and Jinlin Song},
title = {Sequential tests of multiple hypotheses controlling false discovery and nondiscovery rates},
journal = {Sequential Analysis},
volume = {39},
number = {1},
pages = {65-91},
year  = {2020},
publisher = {Taylor & Francis},
doi = {10.1080/07474946.2020.1726686},
URL = { 
        https://doi.org/10.1080/07474946.2020.1726686    
},
eprint = { 
        https://doi.org/10.1080/07474946.2020.1726686
    
}
}

@Misc{Hall_Heyde,
 Author = {Hall, P. and Heyde, C. C.},
 Title = {Martingale limit theory and its application},
 Year = {1980},
 Language = {English},
 HowPublished = {Probability and {Mathematical} {Statistics}. {New} {York} etc.: {Academic} {Press}, {A} {Subsidiary} of {Harcourt} {Brace} {Jovanovich}, {Publishers}. {XII}, 308 p. \$ 36.00 (1980).},
 zbMATH = {3723610},
 Zbl = {0462.60045}
}

@article{He_and_Bartroff_2021,
title = {Asymptotically optimal sequential FDR and pFDR control with (or without) prior information on the number of signals},
journal = {Journal of Statistical Planning and Inference},
volume = {210},
pages = {87-99},
year = {2021},
issn = {0378-3758},
doi = {https://doi.org/10.1016/j.jspi.2020.05.002},
url = {https://www.sciencedirect.com/science/article/pii/S0378375820300604},
author = {Xinrui He and Jay Bartroff}
}

@TechReport{Bessler1960_I,
  author      = {Bessler, Stuart A.},
  title       = {Theory and Applications of the Sequential Design of Experiments, k-Actions and Infinitely Many Experiments, Part I Theory.},
  institution = {Department of Statistics, Stanford University},
  year        = {1960},
  type        = {Technical Report},
  number      = {55},
}

@TechReport{Bessler1960_II,
  author      = {Bessler, Stuart A.},
  title       = {Theory and Applications of the Sequential Design of Experiments, k-Actions and Infinitely Many Experiments, Part II Applications.},
  institution = {Department of Statistics, Stanford University},
  year        = {1960},
  type        = {Technical Report},
  number      = {56},
}

@Article{chernoff1959,
  author    = {Chernoff, Herman},
  title     = {Sequential Design of Experiments},
  journal   = {Ann. Math. Statist.},
  year      = {1959},
  volume    = {30},
  number    = {3},
  pages     = {755--770},
  month     = {09},
  doi       = {10.1214/aoms/1177706205},
  fjournal  = {The Annals of Mathematical Statistics},
  publisher = {The Institute of Mathematical Statistics},
  url       = {http://dx.doi.org/10.1214/aoms/1177706205},
}

@article{Cohen2015active,
  title={Active hypothesis testing for anomaly detection},
  author={Cohen, Kobi and Zhao, Qing},
  journal={IEEE Transactions on Information Theory},
  volume={61},
  number={3},
  pages={1432--1450},
  year={2015},
  publisher={IEEE}
}

@ARTICLE{huang2017active,  author={Huang, Boshuang and Cohen, Kobi and Zhao, Qing},  journal={IEEE Transactions on Information Theory},   title={Active Anomaly Detection in Heterogeneous Processes},   year={2019},  volume={65},  number={4},  pages={2284-2301},  doi={10.1109/TIT.2018.2866257}}

@ARTICLE{Cohen2019nonlinearcost,  author={Gurevich, Andrey and Cohen, Kobi and Zhao, Qing},  journal={IEEE Transactions on Signal Processing},   title={Sequential Anomaly Detection Under a Nonlinear System Cost},   year={2019},  volume={67},  number={14},  pages={3689-3703},  doi={10.1109/TSP.2019.2918981}}

@ARTICLE{Cohen2020composite,  author={Hemo, Bar and Gafni, Tomer and Cohen, Kobi and Zhao, Qing},  journal={IEEE Transactions on Signal Processing},   title={Searching for Anomalies Over Composite Hypotheses},   year={2020},  volume={68},  number={},  pages={1181-1196},  doi={10.1109/TSP.2020.2971438}}

@ARTICLE{Cohen2023comp_hier,
  author={Gafni, Tomer and Wolff, Benjamin and Revach, Guy and Shlezinger, Nir and Cohen, Kobi},
  journal={IEEE Transactions on Signal Processing}, 
  title={Anomaly Search Over Discrete Composite Hypotheses in Hierarchical Statistical Models}, 
  year={2023},
  volume={71},
  number={},
  pages={202-217},
  doi={10.1109/TSP.2023.3242074}}

@Article{De_and_Baron_2012b,
  Title                    = {Step-up and step-down methods for testing multiple hypotheses in sequential experiments},
  Author                   = {De, Shyamal K and Baron, Michael},
  Journal                  = {Journal of Statistical Planning and Inference},
  Year                     = {2012},
  Number                   = {7},
  Pages                    = {2059--2070},
  Volume                   = {142},
  Publisher                = {Elsevier}
}

@ARTICLE{Cohen_hier,
  author={Gafni, Tomer and Cohen, Kobi and Zhao, Qing},
  journal={IEEE Signal Processing Letters}, 
  title={Searching for Unknown Anomalies in Hierarchical Data Streams}, 
  year={2021},
  volume={28},
  number={},
  pages={1774-1778},
  doi={10.1109/LSP.2021.3106587}}

@ARTICLE{kartik,
  author={Kartik, Dhruva and Nayyar, Ashutosh and Mitra, Urbashi},
  journal={IEEE Transactions on Automatic Control}, 
  title={Fixed-Horizon Active Hypothesis Testing}, 
  year={2022},
  volume={67},
  number={4},
  pages={1882-1897},
  doi={10.1109/TAC.2021.3090742}}

@article{Keener_1984,
author = {Robert Keener},
title = {{Second Order Efficiency in the Sequential Design of Experiments}},
volume = {12},
journal = {The Annals of Statistics},
number = {2},
publisher = {Institute of Mathematical Statistics},
pages = {510 -- 532},
year = {1984},
doi = {10.1214/aos/1176346503},
URL = {https://doi.org/10.1214/aos/1176346503}
}

@article{Kiefer_Sacks_1963,
author = {J. Kiefer and J. Sacks},
title = {{Asymptotically Optimum Sequential Inference and Design}},
volume = {34},
journal = {The Annals of Mathematical Statistics},
number = {3},
publisher = {Institute of Mathematical Statistics},
pages = {705 -- 750},
year = {1963},
doi = {10.1214/aoms/1177704000},
URL = {https://doi.org/10.1214/aoms/1177704000}
}

@article{lorden1970excess,
  author={Lorden, Gary},
  title={{On excess over the boundary}},
  volume={41},
  journal={The Annals of Mathematical Statistics},
  number={2},
  pages={520--527},
  year={1970},
  publisher={Institute of Mathematical Statistics}
}

@article{Lalley_Lorden_1986,
author = {S. P. Lalley and G. Lorden},
title = {{A Control Problem Arising in the Sequential Design of Experiments}},
volume = {14},
journal = {The Annals of Probability},
number = {1},
publisher = {Institute of Mathematical Statistics},
pages = {136 -- 172},
year = {1986},
doi = {10.1214/aop/1176992620},
URL = {https://doi.org/10.1214/aop/1176992620}
}

@Article{LehmannRomano2005,
  Title                    = {Generalizations of the familywise error rate},
  Author                   = {Lehmann, E. L. and Romano, Joseph P.},
  Journal                  = {Ann. Statist.},
  Year                     = {2005},

  Month                    = {06},
  Number                   = {3},
  Pages                    = {1138--1154},
  Volume                   = {33},

  Doi                      = {10.1214/009053605000000084},
  Fjournal                 = {The Annals of Statistics},
  Publisher                = {The Institute of Mathematical Statistics},
  Url                      = {http://dx.doi.org/10.1214/009053605000000084}
}

@Article{nitinawarat_controlled_2013,
  author    = {Nitinawarat, Sirin and Atia, George K and Veeravalli, Venugopal V},
  title     = {Controlled sensing for multihypothesis testing},
  journal   = {IEEE Transactions on Automatic Control},
  year      = {2013},
  volume    = {58},
  number    = {10},
  pages     = {2451--2464},
  publisher = {IEEE},
}

@article{nitinawarat_controlled_2015,
author = { Sirin   Nitinawarat  and  Venupogal V.   Veeravalli },
title = {Controlled Sensing for Sequential Multihypothesis Testing with Controlled Markovian Observations and Non-Uniform Control Cost},
journal = {Sequential Analysis},
volume = {34},
number = {1},
pages = {1-24},
year  = {2015},
publisher = {Taylor & Francis},
doi = {10.1080/07474946.2014.961864},

URL = { 
        https://doi.org/10.1080/07474946.2014.961864
    
},
eprint = { 
        https://doi.org/10.1080/07474946.2014.961864
    
}

}

@ARTICLE{oddball_2018,
 author={Vaidhiyan, Nidhin Koshy and Sundaresan, Rajesh},  journal={IEEE Transactions on Information Theory},   title={Learning to Detect an Oddball Target},   year={2018},  volume={64},  number={2},  pages={831-852},  doi={10.1109/TIT.2017.2778264}}

@article{Prabhu2022,
  title={Sequential multi-hypothesis testing in multi-armed bandit problems: An approach for asymptotic optimality},
  author={Prabhu, Gayathri R and Bhashyam, Srikrishna and Gopalan, Aditya and Sundaresan, Rajesh},
  journal={IEEE Transactions on Information Theory},
  year={2022},
  publisher={IEEE}
}

@article{Song_and_Fellouris_2016,
author = {Song, Yanglei and Fellouris, Georgios},
year = {2016},
month = {03},
pages = {},
title = {Asymptotically optimal, sequential, multiple testing procedures with prior information on the number of signals},
volume = {11},
journal = {Electronic Journal of Statistics},
doi = {10.1214/17-EJS1223}
}

@article{Song_and_Fellouris_2019,
author = {Yanglei Song and Georgios Fellouris},
title = {{Sequential multiple testing with generalized error control: An asymptotic optimality theory}},
volume = {47},
journal = {The Annals of Statistics},
number = {3},
publisher = {Institute of Mathematical Statistics},
pages = {1776 -- 1803},
year = {2019},
doi = {10.1214/18-AOS1737},
URL = {https://doi.org/10.1214/18-AOS1737}
}

@article{song_fel_supple,
author = {Yanglei Song and Georgios Fellouris},
title = {{Supplementary file to: Sequential multiple testing with generalized error control: An asymptotic optimality theory}},
volume = {47},
journal = {The Annals of Statistics},
number = {3},
publisher = {Institute of Mathematical Statistics},
pages = {1776 -- 1803},
year = {2019},
doi = {10.1214/18-AOS1737},
URL = {https://doi.org/10.1214/18-AOS1737}
}

@article{Tsopela_2022,
  title={Sequential anomaly detection under sampling constraints},
  author={Tsopelakos, Aristomenis and Fellouris, Georgios},
  journal={IEEE Transactions on Information Theory},
  volume={69},
  number={12},
  pages={8126--8146},
  year={2022},
  publisher={IEEE}
}

@INPROCEEDINGS{Tsopela_2019,  author={Tsopelakos, Aristomenis and Fellouris, Georgios and Veeravalli, Venugopal V.},  booktitle={2019 IEEE International Symposium on Information Theory (ISIT)},   title={Sequential anomaly detection with observation control},   year={2019},  volume={},  number={},  pages={2389-2393},  doi={10.1109/ISIT.2019.8849555}}

@INPROCEEDINGS{Tsopela_2020,  author={Tsopelakos, Aristomenis and Fellouris, Georgios},  booktitle={2020 IEEE International Symposium on Information Theory (ISIT)},   title={Sequential anomaly detection with observation control under a generalized error metric},   year={2020},  volume={},  number={},  pages={1165-1170},  doi={10.1109/ISIT44484.2020.9174081}}

@Book{durrett2010probability,
  title     = {Probability: theory and examples},
  publisher = {Cambridge university press},
  year      = {2010},
  author    = {Durrett, Rick},
}

@Book{tartakovsky_book_2014,
  title     = {Sequential analysis: Hypothesis testing and changepoint detection},
  publisher = {CRC press},
  year      = {2015},
  author    = {Tartakovsky, Alexander and Nikiforov, Igor and Basseville, Mich{\`e}le},
}

@ARTICLE{tart1,
  author={Tartakovsky, Alexander G. and Polunchenko, Aleksey S. and Sokolov, Grigory},
  journal={IEEE Journal of Selected Topics in Signal Processing}, 
  title={Efficient Computer Network Anomaly Detection by Changepoint Detection Methods}, 
  year={2013},
  volume={7},
  number={1},
  pages={4-11},
  doi={10.1109/JSTSP.2012.2233713}}

@article{Brain_2010,
author = {Stiles, Joan and Jernigan, Terry},
year = {2010},
month = {11},
pages = {327-48},
title = {The Basics of Brain Development},
volume = {20},
journal = {Neuropsychology review},
doi = {10.1007/s11065-010-9148-4}
}

@book{kal02,
  title={Foundations of Modern Probability},
  author={Kallenberg, O.},
  isbn={9783030618704},
  series={Probability Theory and Stochastic Modelling},
  url={https://books.google.fi/books?id=0LLZzQEACAAJ},
  year={2021},
  publisher={Springer International Publishing}
}

@article{ts2024sequential,
  title={Sequential anomaly identification under sampling constraints for generalized error metrics},
  author={Tsopelakos, Aristomenis and Fellouris, Georgios},
  journal={IEEE Transactions on Information Theory},
  year={2025}
}

@article{ren2003rosenthal,
  title={On the Rosenthal's inequality for locally square integrable martingales},
  author={Ren, Yao-Feng and Tian, Fan-Ji},
  journal={Stochastic processes and their applications},
  volume={104},
  number={1},
  pages={107--116},
  year={2003},
  publisher={Elsevier}
}

\appendices
\renewcommand{\thetheorem}{ \Alph{section}.\arabic{theorem}}
\renewcommand{\thelemma}{ \Alph{section}.\arabic{lemma}}

\section{Definition of $x(A)$ and $y(A)$}\label{app: defin x and y}

In Appendix \ref{app: defin x and y}, we present the formulas of $x(A)$ and $y(A)$ as originally given in \cite[Theorem 5.1]{Tsopela_2022}. We also provide Proposition \ref{corl}, which reveals the fact that by $x(A)=0$ (resp. $y(A)=0$) we can deduce that $|A|=\ell$ (resp. $|A|=u$), a property that will be used in proof of our results about consistency. For the better presentation of the formulas of $x(A)$ and $y(A)$ we consider the following auxiliary quantities
\begin{equation*}
	\begin{aligned}
\theta_{A}:= I^{*}(A) / J^{*}(A^c),\qquad z_A :=\theta_{A}/(r-1), \;\; \mbox{if } \, r>1,\qquad w_A := (1/\theta_A)/( 1/r -1), \;\; \mbox{if } \, r<1,
	\end{aligned}
\end{equation*}
and
\begin{align}\label{K_hat}
	\begin{split}
		\hat{K}(A):=
		\begin{cases}
			\sum\limits_{i \in  A} I^{*}(A)/I_i,& \qquad  A \neq \emptyset, \\
			0,& \qquad  A=\emptyset, 
		\end{cases} \qquad \qquad  
		\check{K}(A^c):=
		\begin{cases}
			\sum\limits_{i \in  A^c} J^{*}(A^c)/J_i,& \quad  A \neq [M],  \\
			0,& \quad A=[M].
		\end{cases} 
	\end{split}
\end{align}

A. We start with the case where the number of anomalies is known, i.e. $1 \leq \ell = u \leq M-1$, and we distinguish two cases:
\begin{itemize}
	\item  If  $\hat{K}(A)  \leq \theta_{A}\, \check{K}(A^c) $, then
	\begin{equation}\label{kn_c1} 
	\begin{aligned}
		x(A):=( K/ \hat{K}(A) ) \wedge 1, \qquad  y(A):= \left( ( K-\hat{K}(A) )^{+} / \check{K}(A^c) \right) \wedge 1.
	\end{aligned}
	\end{equation}
	
	\item If $\hat{K}(A)  > \theta_{A} \, \check{K}(A^c) $, then  
	\begin{equation}\label{kn_c2} 
	\begin{aligned}
		x(A):= \left(( K-\check{K}(A^c) )^{+} / \hat{K}(A) \right)  \wedge 1, \qquad y(A):=( K/ \check{K}(A^c) ) \wedge 1.
	\end{aligned}
	\end{equation}
\end{itemize}

B. We continue with the case where the number of anomalies is unknown, i.e., $0 \leq \ell < u \leq M$, and we distinguish the following three cases.

\begin{enumerate}[(a)]
	\item If $\ell < |A|<u$, then 
	\begin{align*}
		x(A) := \frac{K}{\hat{K}(A) + (\theta_{A}/r) \check{K}(A^c)} \wedge (r/\theta_{A}) \wedge  1, \qquad  y(A)&:=(\theta_{A}/r)\, x(A). 
	\end{align*}	
	
	\item If $|A|=\ell$, then we distinguish three subcases. \\
	
	\begin{itemize}
		\item [1.] If $\ell =0$ or $r \leq 1$, then  
		\begin{align*}
			x(A):=0, \qquad y(A):=(K/ \check{K}(A^c) ) \wedge 1.
		\end{align*}
		
		\item [2.] If $\ell >0$, $r >1$, $z_A <1$,  and   $K >\hat{K}(A)+z_A \, \check{K}(A^c)$, then
		\begin{align*}
			x(A):=1, \qquad y(A):= \bigl( (K-\hat{K}(A) )/  \check{K}(A^c) \bigr) \wedge 1.
		\end{align*}
		
		\item [3.] If $\ell >0$,  $r >1$, and either $z_A  \geq 1$ or $K \leq \hat{K}(A)+z_A \, \check{K}(A^c) $, then 
		\begin{align*}
			x(A):= \frac{K}{\hat{K}(A)+z_A \, \check{K}(A^c)}   \wedge  (1/z_A)\wedge  1,  \qquad y(A):= \frac{K}{\check{K}(A^c)+(1/z_A) \, \hat{K}(A)}   \wedge  z_A \wedge  1. 
		\end{align*}					
		
	\end{itemize}
	
	\item If $|A|=u$, then we distinguish three subcases.\\
	
	\begin{itemize}
		\item [1.] If $u=M$ or $r \geq 1$, then 
		\begin{align*}
			x(A):=(K/ \hat{K}(A) ) \wedge 1, \qquad y(A):=0.
		\end{align*}
		\item [2.] If $u<M$, $r < 1$, $w_A<1$, and $K >  \check{K}(A^c)+ w_A \hat{K}(A)$, then  
		\begin{align*} 
			x(A) :=\bigl( (K-\check{K}(A^c)) / \hat{K}(A) \bigr)  \wedge 1, \qquad y(A):=1.
		\end{align*}						
		\item [3.] If $u<M$, $r < 1$, and either	$w_A\geq 1$ or $K  \leq   \check{K}(A^c)+ w_A \hat{K}(A)$, then
		\begin{align*} 
			x(A) := \frac{K}{ \hat{K}(A)+ (1/w_A) \check{K}(A^c)} \wedge w_A \wedge 1, \qquad
			y(A) := \frac{K}{ \check{K}(A^c)+ w_A \hat{K}(A)} \wedge (1/w_A) \wedge 1. 
		\end{align*}		    
	\end{itemize}
\end{enumerate}	

By definition, at least one of $x(A)$, $y(A)$ is positive in anyone of the above cases. However, when $K$ is relatively small one of the $x(A)$, $y(A)$ can be equal to zero. By inspection of the formulas, we verify that the fact that $x(A)=0$ (resp. $y(A)=0$) reveals information about the size of $|A|$, as summarized in the following proposition.

\begin{proposition}\label{corl}
	For any $A \, \in \mathcal{P}_{\ell,u}$, the following implications hold.
	\begin{enumerate}
		
		\item[(i)] If $x(A)=0$, then we know that $|A|=\ell$.
		
		\item[(ii)] If $y(A)=0$, then we know that $|A|=u$.
	\end{enumerate}
\end{proposition}

\begin{IEEEproof}
By inspection of the formulas, we observe that $x(A)=0$ can hold only in the following two cases: (1) when $|A|=0 < u$ (subcase b.1), or (2) when $|A|=\ell =u$, $\hat{K}(A)  > \theta_{A} \, \check{K}(A^c)$ and $K \leq \check{K}(A^c)$. Similarly,  $y(A)=0$ only in the following two cases: (1) when $|A|=M > \ell$ (subcase c.1), or (2) when $|A|=\ell =u$, $\hat{K}(A)  \leq \theta_{A} \, \check{K}(A^c)$ and $K \leq \hat{K}(A)$.
\end{IEEEproof}

\section{}\label{intrLLR}
In Appendix \ref{intrLLR}, we state and prove auxiliary lemmas that are used in the proof of the main results of this paper presented in Appendices \ref{consist_orde}, \ref{vlbg11}. Throughout this appendix, we fix a subset of anomalies $A \in \cP_{\ell, u}$, and a sampling rule $R$. In order to lighten the notation we do not emphasize the dependence of the various statistics on the sampling rule $R$, thus we write  $\Lambda_i(n), \pi_i(n), \cF_{n}$ instead of  $\Lambda^R_i(n), \pi_i^R(n), \cF^R_{n}$. For the establishment of our results, we introduce the following notation.

For any $i \in [M]$, we set 
\begin{align}\label{tilde}
	\bar{\Lambda}_i(n) :=
	\begin{cases}
		 \bar{\Lambda}_i(n-1) + \Bigl( \log\left( \frac{f_{i1}(X_i(n))}{f_{i0}(X_i(n))} \right)  - \Exp_A\left[ \log\left( \frac{f_{i1}(X_i(n))}{f_{i0}(X_i(n))} \right)\right] \Bigr) \, R_i(n),\quad &n \in \mathbb{N}, \\
		 0, & n=0,
	\end{cases}
\end{align} 
and comparing with \eqref{LLR}, we observe that for all $n \in \bN_{0}$,
\begin{align}\label{decompose}
	\bar{\Lambda}_i(n) &= 
	\begin{cases}
		\Lambda_i(n) - I_i \,n\, \pi_i(n), \quad i \in A, \\ 
		\Lambda_i(n) + J_i \, n\, \pi_i(n), \quad i \notin A.
	\end{cases}
\end{align}
For each $i \in [M]$ and $n,m \in \bN_0$, we denote by $\Lambda_i(n:m)$  the LLR statistic based on the measurements from source $i$ during  $[m+1,m+n]$, i.e., $\Lambda_{i}(n:m):=\Lambda_{i}(n+m) - \Lambda_{i}(m)$, respectively,
\begin{equation}\label{lnmb}
	\bar{\Lambda}_{i}(n:m):=\bar{\Lambda}_{i}(n+m) - \bar{\Lambda}_{i}(m),
\end{equation}
and by $\pi_{i}(n:m)$ the empirical sampling frequency of source $i$ during $[m+1, m+n]$, i.e.,
\begin{equation}\label{pnm}
	\pi_{i}(n:m):= \frac{1}{n} \sum_{u=m+1}^{m+n} R_{i}(u).
\end{equation}
In the following results, in place of $m$ we have an a stopping time $\tau \in \cT$, where $\cT$ is the family of all a.s.-finite stopping times with respect to $\{ \cF_{n}\,:\, n \in \bN \}$. The Lemmas \ref{exp_tau}, \ref{polyn}, \ref{lambda_tau} show the rate of decay of the probability of an event $E(n:\tau)$ with respect to $n$, which is defined based on the measurements collected after a stopping time $\tau \in \cT$ up to time $\tau+n$, i.e., during the random interval $[\tau+1, \tau+n]$. For this, we introduce the following terminology. Given a family of events $\{E(n:\tau): n \in \bN_0, \tau \in \cT\}$,  we say that the conditional probability of  $E(n:\tau)$ given $\cF_{\tau}$ is
\begin{itemize}
	\item  \textit{uniformly exponentially decaying} if there are $C,\, c >0$  independent of $\tau \in \mathcal{T}$ and $n\in \bN_{0}$, so that 
	\begin{equation}
		\Pro_{A}\left( E(n,\tau) \, |\, \cF_{\tau} \right) \leq C \, e^{-c n}, \quad \mbox{ a.s.}
	\end{equation}
	\item  \textit{uniformly $q$-polynomially decaying}, for some $q>0$, if there is $C >0$ independent of $\tau \in \mathcal{T}$ and $n\in \bN_{0}$, so that 
	\begin{equation}\label{pol}
		\Pro_{A}\left( E(n,\tau) \, |\, \cF_{\tau} \right) \leq C \, n^{-q}, \quad \mbox{ a.s.}
	\end{equation}
\end{itemize}

\textit{In what follows, when we refer to a constant we also imply independent of any stopping time in $\mathcal{T}$.} The following Lemma \ref{exp_tau} is a generalization of \cite[Lemma A.1]{Tsopela_2022}, which  corresponds to the special case that $\tau=0$.

\begin{lemma}\label{exp_tau}
	Let  $\zeta \in (0,1]$,  $\epsilon >0$.  Then,
	\begin{equation}\label{abs_m_tau_yp}
		\begin{aligned}
			\Pro_{A}&\left(\exists\, m \geq \zeta n: \, \bar{\Lambda}_i(m:\tau)<-\epsilon \, m \, |\, \cF_{\tau} \right), \quad\, \forall \;  i \in A, \\
			\Pro_{A}&\left(\exists\, m \geq \zeta n:\, \bar{\Lambda}_i(m:\tau)>\epsilon \, m \, | \, \cF_{\tau} \right),  \qquad \forall \;  i \notin A,
		\end{aligned}
	\end{equation}
	are	uniformly exponentially decaying. 
\end{lemma}

\begin{IEEEproof}
	We only prove the inequality for $i \in A$, as the proof for $i \notin A$ is similar. We fix $i \in A$, $m \in \bN$, $\tau \in \cT$, and it suffices to show that there is a constant $c>0$ such that
	\begin{equation}\label{abs_m_tau_yp2}
		\Pro_{A}\left(\bar{\Lambda}_i(m:\tau)<-\epsilon \, m,\, B \right) \leq e^{- c\, m} \, \Pro_{A}(B), \quad \forall \; 
		B \in \cF_{\tau}.
	\end{equation}
	Then, the claim follows by application of the law of total probability over all $m \geq \zeta n$. We fix $B \in \cF_{\tau}$. Since $\Pro_{A}(\tau < \infty)=1$, by the law of total probability we have
	\begin{equation*}
		\Pro_{A}\left(\bar{\Lambda}_i(m:\tau)<-\epsilon \, m ,B \right)= \sum_{u=0}^{\infty} \Pro_{A}\left(\bar{\Lambda}_i(m:u)<-\epsilon \, m ,B, \tau=u \right).
	\end{equation*}
	Since $B \cap\{\tau=u\} \in \cF_{u}$ for every $u \in \bN_{0}$, working as in \cite[Lemma A.1]{Tsopela_2022} we can show that  there is a constant $c>0$ so that for every $u \in \bN_{0}$,
	\begin{equation*}
		\begin{aligned}
			\Pro_{A}\left(\bar{\Lambda}_i(m:u)<-\epsilon \, m ,B, \tau=u \right) \leq e^{- c\, m} \, \Pro_{A}(\tau=u,B).
		\end{aligned}
	\end{equation*}
	and summing over all $u \in \bN_{0}$, we show  \eqref{abs_m_tau_yp2}.\\
\end{IEEEproof}

\begin{lemma}\label{polyn}
	Let  $\zeta \in (0,1]$, $\epsilon >0$. Suppose condition \eqref{mc1} holds for some $\mathfrak{p} \geq 2$. Then, 
	\begin{align*}
		\Pro_{A}&\left(\exists \,  m \in [\zeta n,n] :\, \bar{\Lambda}_i(m:\tau) \geq \epsilon \, m \, |\,  \cF_{\tau}\right)  \quad \;\;\, \forall \;  i \in A, \\
		\Pro_{A}&\left(\exists \, m \in [\zeta n,n] :\, \bar{\Lambda}_i(m:\tau) \leq -\epsilon \,  m \, | \,  \cF_{\tau}\right)
		\quad \forall \;  i \notin A,
	\end{align*}
	are uniformly $\mathfrak{p}/2$-polynomially decaying.
\end{lemma}

\begin{IEEEproof}
	We prove the lemma for  $i \in A$, as the proof for $i \notin A$ is similar. We fix $i \in A$, $n \in \bN$, $\tau \in \cT$, and it suffices to show that there is a constant $C>0$ such that
	\begin{equation}\label{tsh1}
	 \Pro_{A}\left(\exists \, m \in [\zeta n,n]:\,\bar{\Lambda}_i(m:\tau)\geq \epsilon \, m, B \right) \leq C\, n^{-\mathfrak{p}/2}\, \Pro_{A}(B), \quad \forall \; 
	 B \in \cF_{\tau}.
	\end{equation}
	We fix $B \in \cF_{\tau}$. Since $\Pro_{A}(\tau < \infty)=1$, by  the law of total probability we have 
	\begin{equation}\label{stc}
		\begin{aligned}
			\Pro_{A}\left(\exists \, m \in [\zeta n,n]:\,\bar{\Lambda}_i(m:\tau)\geq \epsilon \, m, B \right)
			= \sum_{u=0}^{\infty} \Pro_{A}\left(\exists \, m \in [\zeta n,n]:\bar{\Lambda}_i(m:u)\geq \epsilon \, m, B, \tau=u \right),
		\end{aligned}
	\end{equation}
	and each term in the sum is further bounded by
	\begin{align*}
		\Pro_{A}\left(\exists \, m \in [\zeta n,n]:\,\bar{\Lambda}_i(m:u) \geq \epsilon \, m, B, \; \tau=u \right) 
		&\leq \Pro_A\left(\exists \, m \in [\zeta n,n]:\,\bar{\Lambda}_i(m:u) \geq \epsilon \, \zeta\,  n,  B, \; \tau=u \right) \\
		&\leq \Pro_{A}\left( \max_{m\in [0,n]}\big{|}\bar{\Lambda}_i(m:u)\big{|}^{\mathfrak{p}} \geq 
		(\epsilon \, \zeta \, n)^{\mathfrak{p}},B, \; \tau=u \right).
	\end{align*}
	For each $u \in \bN_{0}$, $\{\bar{\Lambda}_{i}(m:u)\,:\, m \in \bN_{0}\}$ is a $\{\cF_{m+u}\,:\, m \in \bN_{0}\}$-martingale. Thus, by Doob's submartingale inequality we have 
	\begin{equation*}
		\Pro_{A}\left(\max_{m\in [0,n]}|\bar{\Lambda}_i(m:u)|^{\mathfrak{p}} \geq \epsilon^{\mathfrak{p}} (\zeta\, n)^{\mathfrak{p}}, B, \tau=u  \right) \leq \frac{\Exp_{A}\left[|\bar{\Lambda}_{i}(n:u)|^{\mathfrak{p}}; \{B, \tau=u\} \right]}{(\zeta\epsilon)^{\mathfrak{p}} \, n^{\mathfrak{p}}}, 
	\end{equation*}
	and by  Rosenthal's inequality \cite[Theorem 2.12]{Hall_Heyde} there is a $C_{0}>0$ such that 
	\begin{equation*}
		\Exp_{A}\left[|\bar{\Lambda}_{i}(n:u)|^{\mathfrak{p}}; \,\{ B, \tau=u\} \right] \leq C_0 \, n^{\mathfrak{p}/2} \, \Pro_{A}(B,\tau=u).
	\end{equation*}
	which implies that 
	\begin{equation*}
		\Pro_{A}\left(\exists \, m \in [\zeta n,n]:\, \bar{\Lambda}_i(m:u) \geq \epsilon m, B, \; \tau=u \right)  \leq C\, n^{-\mathfrak{p}/2}\, \Pro_{A}(B,\tau=u),
	\end{equation*}
	where $C:= C_{0}/(\zeta\epsilon)^{\mathfrak{p}}$, and summing over $u \in \mathbb{N}_{0}$ we show \eqref{tsh1}.\\
\end{IEEEproof}
For Lemma \ref{lambda_tau}, we introduce the quantity 
\begin{equation}\label{qa}
	q_{A}:=\max\bigg\{ \max\limits_{i\in A}I_{i}/I^{*}(A), \, \max\limits_{i\notin A}J_{i}/J^{*}(A^c) \bigg\}.
\end{equation} 

\begin{lemma}\label{lambda_tau} 
	Let $\zeta,\, \rho,\, \lambda \in (0,1]$.
	
	\begin{itemize}
		\item[(i)] For all $i \in A$, $j \notin A$,  
		\begin{align*}
			\Pro_{A}\left(\exists\, m \geq \zeta n:\, \Lambda_{i}(m:\tau)<0,\, \pi_{i}(m:\tau) > \rho
			\, | \,\cF_{\tau} \right)  & \\
			\Pro_{A}\left(\exists\, m \geq \zeta n:\, \Lambda_{j}(m:\tau)>0,\, \pi_{j}(m:\tau) > \rho
			\, |\, \cF_{\tau} \right)  &\\
			\Pro_{A}\left(\exists\, m \geq \zeta n:\, \Lambda_{j}(m:\tau)> \Lambda_{i}(m:\tau),\, \pi_{i}(m:\tau) > \rho \,| \, \cF_{\tau} \right)  &\\
			\Pro_{A}\left(\exists\, m \geq \zeta \,  n:\, \Lambda_{j}(m:\tau)> \Lambda_{i}(m:\tau),\, \pi_{j}(m:\tau) > \rho \,| \, \cF_{\tau} \right)  &
		\end{align*}  
		are uniformly exponentially decaying.\\
		
		\item [(ii)]  If  $\lambda/\rho > q_{A}$ and condition \eqref{mc1} holds for some $\mathfrak{p}\geq 2$, then for all $i, j \in A$,
		\begin{equation}\label{rf1}
			\begin{aligned}
				\Pro_{A}&\left(\exists \, m \in [\zeta n,n]:\, \Lambda_{i}(m:\tau){>}\Lambda_{j}(m:\tau),\, \pi_{i}(m:\tau) {<} \rho,\, \pi_{j}(m:\tau) {>} \lambda \, | \, \cF_{\tau}\right)		
			\end{aligned}
		\end{equation}
		is uniformly $\mathfrak{p}/2$-polynomially decaying.\\
		
		\item [(iii)] If  $\lambda/\rho < 1/ q_{A}$ and condition \eqref{mc1} holds for some $\mathfrak{p}\geq 2$, then for all $i, j \notin A$,
		\begin{equation*}
			\begin{aligned}
				\Pro_{A}&\left(\exists \, m \in [\zeta n,n]:\, \Lambda_{i}(m:\tau)> \Lambda_{j}(m:\tau),\, \pi_{i}(m:\tau) > \rho,\, \pi_{j}(m:\tau) < \lambda \, | \, \cF_{\tau}\right)
			\end{aligned}
		\end{equation*}
		is 	uniformly $\mathfrak{p}/2$-polynomially decaying.
	\end{itemize}
\end{lemma}

\begin{IEEEproof}
	(i) We only prove the claim for the  first and the third  conditional probability,   as the proofs  for the second and the fourth, respectively, are similar. By decomposition \eqref{decompose} we obtain 
	\begin{equation*}
		\begin{aligned}
			&\Pro_{A}\left(\exists\, m \geq \zeta n:\, \Lambda_{i}(m:\tau)<0,\, \pi_{i}(m:\tau) > \rho
			\, | \,\cF_{\tau} \right)\\
			&=\Pro_{A}(\exists \,  m \geq \zeta n :\, \bar{\Lambda}_{i}(m:\tau)<-I_{i} \, \pi_{i}(m:\tau) \, m, \, \pi_{i}(m:\tau) > \rho
			\,  | \, \cF_{\tau}) \\
			&\leq \Pro_{A}\left(\exists \,  m \geq \zeta \, n:\, \bar{\Lambda}_{i}(m:\tau)<-I_{i}\, \rho  \,  m \, | \, \cF_{\tau} \right).
		\end{aligned}
	\end{equation*}
	Similarly, 
	\begin{equation*}
		\begin{aligned}
			&\Pro_{A}\left(\exists\, m \geq \zeta  \,  n:\, \Lambda_{j}(m:\tau)> \Lambda_{i}(m:\tau),\, \pi_{i}(m:\tau) > \rho \,| \, \cF_{\tau} \right)\\
			&=\Pro_{A}\bigg{(} \exists \,  m \geq \zeta \,  n:  \, \bar{\Lambda}_{j}(m:\tau) -\bar{\Lambda}_{i}(m:\tau) >(I_{i}\pi_{i}(m:\tau)+J_{j}\pi_{j}(m:\tau))\, m, \pi_{i}(m:\tau) > \rho \, |\,  \cF_{\tau} \bigg{)} \\ 
			&\leq  \Pro_{A}\left(\exists \,  m \geq \zeta \, n:\, \bar{\Lambda}_{j}(m:\tau) -\bar{\Lambda}_{i}(m:\tau) >I_{i} \, \rho \, m
			\, | \, \cF_{\tau} \right)\\
			& \leq \Pro_{A}\left(\exists  \, m \geq \zeta \, n:\,\bar{\Lambda}_{j}(m:\tau)>I_{i}\, (\rho/2)\,  m \, |\, \cF_{\tau} \right)+\Pro_{A}\left(\exists \,  m \geq \zeta \,  n:\, -\bar{\Lambda}_{i}(m:\tau) >I_{i} \, (\rho/2)\,  m \,  |\,  \cF_{\tau} \right).
		\end{aligned}
	\end{equation*}
	In both cases, each term in the upper bound is uniformly exponentially decaying by Lemma \ref{exp_tau}, which proves the claim.
	
	For (ii) (similarly for (iii)), by decomposition \eqref{decompose},  the conditional probability in  \eqref{rf1} can be expressed as 
	\begin{equation*}
		\begin{aligned}
			\Pro_{A}&\Big{(}\exists \, m \in [\zeta n,n] : \bar{\Lambda}_{i}(m:\tau) {-}\bar{\Lambda}_{j}(m:\tau) >(I_{j} \pi_{j}(m:\tau) {-}I_{i}  \pi_{i}(m:\tau)\big) m, \, \pi_{i}(m:\tau) {<} \rho, \, \pi_{j}(m:\tau) {>} \lambda \,|\, \cF_{\tau} \Big{)} 
		\end{aligned}
	\end{equation*}
	and is bounded by
	\begin{equation*}
		\begin{aligned}
			&\Pro_{A}(\exists \, m \in [\zeta n,n]:\,\bar{\Lambda}_{i}(m:\tau) -\bar{\Lambda}_{j}(m:\tau) >(I_{j}\lambda -I_{i}\rho) m \, | \, \cF_{\tau})\\
			&\leq \Pro_{A}(\exists \, m \in [\zeta n,n]:\,\bar{\Lambda}_{i}(m:\tau) > (I_{j}\lambda -I_{i}\rho) \, m/  2  \,| \, \cF_{\tau} ) \\
			&\qquad + \Pro_{A}(\exists \, m \in [\zeta n,n]:\, -\bar{\Lambda}_{j}(m:\tau) >(I_{j} \, \lambda -I_{i} \, \rho) \,m/2 \, |\, \cF_{\tau} ),
		\end{aligned}
	\end{equation*}
	where $I_{j}\lambda -I_{i}\rho >0$ because $\lambda / \rho >q_{A}$. The first term on the right hand side is uniformly $\mathfrak{p}/2$-polynomially decaying by Lemma \ref{polyn}, and the second term is  uniformly exponentially decaying by Lemma \ref{exp_tau}, which proves the claim.\\
\end{IEEEproof}

In Lemma \ref{dd_l}, we provide a result about how large can be the maximum draw-down of an LLR after a stopping time $\tau \in \cT$.

\begin{lemma}\label{dd_l}
	Let  $\tau \in \cT$ and $x>0$, then  
	\begin{equation}
		\begin{aligned}
			\Pro_{A}\left( \Lambda_{i}(\tau) - \inf_{n \geq \tau }\Lambda_{i}(n)  > x\right) &\leq e^{-x}, \quad \forall \, i \in A, \\
			\Pro_{A}\left( \sup_{n \geq \tau}\Lambda_{i}(n) - \Lambda_{i}(\tau) > x\right) &\leq e^{-x}, \quad \forall \, i \notin A. 
		\end{aligned}
	\end{equation}
	Consequently, the random variables $\Lambda_{i}(\tau) - \inf_{n \geq \tau }\Lambda_{i}(n)$ and $\;\sup_{n \geq \tau}\Lambda_{i}(n) - \Lambda_{i}(\tau)$ have each moment bounded by a constant that depends only on the moment.
\end{lemma}

\begin{IEEEproof}
	We only prove the claim for $i \in A$, as the proof for $i \notin A$  is similar. We note that
	\begin{equation}
		\Lambda_{i}(\tau) - \inf_{n \geq \tau}\Lambda_{i}(n) = -\inf_{n \in \bN}\Lambda_{i}(n:\tau)= \sup_{n \in \bN}\left\{-\Lambda_{i}(n:\tau)\right\},
	\end{equation}
	and thus it suffices to show that for any $x >0$,
	\begin{equation}
		\Pro_{A}\left( \sup_{n \in \bN}\{-\Lambda_{i}(n:\tau)\} > x\right) = \Pro_{A}\left( \sup_{n \in \bN}\left\{ \exp( -\Lambda_{i}(n:\tau))  \right\} > e^x\right) \leq e^{-x}.
	\end{equation}
	Fix $x>0$. Since $\Pro_{A}(\tau < \infty)=1$, by the law of total probability we have
	\begin{equation}\label{a_tot}
		\Pro_{A}\left( \sup_{n \in \bN}\{-\Lambda_{i}(n:\tau)\} > x\right) =\sum_{m=0}^{\infty} \Pro_{A}\left( \sup_{n \in \bN}\left\{ \exp( -\Lambda_{i}(n:m))  \right\} > e^x, \tau=m\right).
	\end{equation}
	We note that for any $n,\,m \in \bN_{0}$, $X_{i}(n+1+m)$ is independent of $\cF_{n+m}$, and it has the same distribution as $X_{i}(1)$.
	Since $R_{i}(n+1+m) \in \{0,1\}$ and is $\cF_{n+m}$-measurable, it holds
	\begin{equation}
		\begin{aligned}
			\Exp_{A}&\left[\exp\{-\log\left( \frac{f_{i1}(X_{i}(n+1+m))}{f_{i0}(X_{i}(n+1+m))} \right)\, R_{i}(n+1+m)  \} \,| \,\cF_{n+m}\right] \\
			&=\Exp_{A} \left[\exp\left\{-\log\left( \frac{f_{i1}(X_{i}(1))}{f_{i0}(X_{i}(1))} \right)\right\} \right]^{R_{i}(n+1+m)}=1.
		\end{aligned}
	\end{equation}
	This implies that for fixed $m \in \bN_{0}$,  $\{\exp\{-\Lambda_{i}(n:m) \}\,:\, n\in \bN_{0}\}$ is a martingale with respect to $\{\cF_{n+m}\,:\,n \in \bN_{0}\}$. Thus, by Ville's supermartingale inequality, for any $m \in \bN_{0}$ we have
	\begin{equation*}
		\Pro_{A}\left(\sup_{n \in \bN}\left( \exp\{-\Lambda_{i}(n:m)\} \right)>e^x,\,\tau=m\right) \leq e^{-x}\, \Pro_{A}(\tau=m),
	\end{equation*}
	and summing over all $m \in \mathbb{N}_{0}$ we show the claim.  
\end{IEEEproof}	

In the following lemma, we provide an upper bound on the moments of the difference of two LLRs.

\begin{lemma}\label{lem:C5}
Suppose that the moment condition \eqref{mc1} holds for some $\mathfrak{p}\geq 2$, and let us fix $p \in [2, \mathfrak{p}]$. Then, there is a constant $C>0$ such that for all $i,j \in [M]$,
\begin{equation*}
 \Exp_{A}\left[ | \Lambda_i(n) - \Lambda_j(n) |^p \right]\leq C \, n^p, \quad \forall\, n \in \bN.
\end{equation*}
\end{lemma}

\begin{IEEEproof}
Without loss of generality, we fix  $i,j \in A$ and $n \in \bN$. By the definition of $\bar{\Lambda}_{i}$ in \eqref{decompose}, and the fact that $\pi_{i}(n) \leq 1$ for all $i \in [M]$, we have
\begin{equation*}
| \Lambda_i(n) - \Lambda_j(n) | \leq |\bar{\Lambda}_{i}(n)|+|\bar{\Lambda}_{j}(n) |+ ( I_i + I_j) n,
\end{equation*}
which further implies that
\begin{equation}\label{jens}
 \begin{aligned}
   \Exp_{A}\left[|\Lambda_i(n) - \Lambda_j(n) |^p \right] &\leq \Exp_{A}\left[\left(|\bar{\Lambda}_{i}(n)|+|\bar{\Lambda}_{j}(n) |+ ( I_i + I_j) n\right)^p \right]\\
   &\leq 3^{p-1} \left( \Exp_{A}\left[|\bar{\Lambda}_{i}(n)|^p \right] + \Exp_{A}\left[|\bar{\Lambda}_{j}(n)|^p \right] + ( I_i + I_j)^p n^p \right),
 \end{aligned}
\end{equation}
where the last inequality follows by Jensen's inequality. Since for each $k \in [M]$, $\{\bar{\Lambda}_{k}(n)\, :\, n \in \bN\}$ is a $\{\cF_{n} : n \in \bN\}$-martingale and $p \geq 2$, by Rosenthal's inequality \cite[Theorem 2.12]{Hall_Heyde} it follows that there is a $C_0 > 0$ such that
\begin{equation*}
	\Exp_{A}\left[\left|\bar{\Lambda}_{k}(n)\right|^{p}\right] \leq C_0 \, n^{p/2}, \quad \forall \; n \in \bN. 
\end{equation*}
Thus, by bounding the right-hand side of \eqref{jens}, we prove the claim.\\
\end{IEEEproof}

For our last lemma, we restore the superscript $R$ and recall the definition of a stabilized ordering rule $R^{A,m}$, for $m \in \bN$, defined in \eqref{repreA}, of the associated LLR $\Lambda_i^{A,m}$ defined in \eqref{llrA}, of the associated empirical sampling frequency $\pi_i^{A,m}$ defined in \eqref{piA}, and of the quantity
\begin{equation*}
	\widetilde{\Lambda}^{A,m}_i(u) := \Lambda^{A,m}_i(u) - u I_i \, \pi^{A,m}_i(u),\quad \forall \; u \in \bN_{0}, \quad \forall \; i \in A.
\end{equation*}
By the definition of $\Lambda^{A,m}_i(u)$, for all $u \in \bN_0$ we have
\begin{equation}\label{tilb}
  \begin{aligned}
  	\widetilde{\Lambda}^{A,m}_i(u) = \Lambda^{R}_i(m) + \sum_{k=1}^u  \log \left( \frac{f_{1i} (X_i(m+k))}{ f_{0i} (X_i(m+k))} \right) \, R^{A,m}_i(k) - u I_i \, \pi^{A,m}_i(u) = \Lambda^{R}_i(m) + \bar{\Lambda}^{A,m}_i(u),
  \end{aligned}
\end{equation}
where 
\begin{equation*}
\bar{\Lambda}^{A,m}_i(u) := \sum_{k=1}^u \left( \log \left( \frac{f_{1i} (X_i(m+k))}{ f_{0i} (X_i(m+k))} \right) - \Exp_{A}\left[ \log \left( \frac{f_{1i} (X_i(m+k))}{ f_{0i} (X_i(m+k))} \right) \right] \right) \, R^{A,m}_i(k),
\end{equation*}
and $\bar{\Lambda}^{A,m}_i(0):=0$. We note that $\{\bar{\Lambda}^{A,m}_i(u) \,:\, u \in \bN \}$ is a $\cF^{A,m}_{u}$-martingale, where
\begin{align*}
	\begin{split}
		\cF^{A,m}_{u} &:=
		\begin{cases} 
			\sigma(\{\Lambda^{R}_i(m) \,:\, i \in A \},\hat{Z}_m,\,\check{Z}_m), \quad &\text{if} \quad  u=0,\\
			\sigma\left(\cF^{A,m}_{u-1},\,\hat{Z}_{m+u},\,\check{Z}_{m+u},\, \{ X_{i}(m+u)\, :\, i \in R^{A,m}(u) \}  \right),  \quad &\text{if} \quad  u\in \bN.
		\end{cases}
	\end{split}
\end{align*}  

\begin{lemma}\label{tilA}
Suppose that the moment condition \eqref{mc1} holds for $\mathfrak{p} > 4$. Then, for an increasing sequence of integers $\{\zeta_n \, :\, n \in \bN\}$ that satisfies \eqref{zn2}, and for any $\epsilon >0$ it holds 
\begin{equation*}
 \sum_{n=1}^{\infty} \Pro_{A}\left(|\widetilde{\Lambda}^{A,\zeta_n}_{i}(n-\zeta_n)|>\epsilon \, (n-\zeta_n) \right) < \infty.
\end{equation*}
\end{lemma}

\begin{IEEEproof}
In view of \eqref{tilb}, we have
\begin{equation*}
  \begin{aligned}
   \Pro_{A}\left(|\widetilde{\Lambda}^{A,\zeta_n}_{i}(n-\zeta_n)|>\epsilon \, (n-\zeta_n) \right) \leq \, \Pro_{A}\left(|\Lambda^{R}_i(\zeta_n)|>\frac{\epsilon}{2} \, (n-\zeta_n) \right)+ \Pro_{A}\left(|\bar{\Lambda}^{A,\zeta_n}_i(n-\zeta_n)|>\frac{\epsilon}{2} \, (n-\zeta_n) \right).
  \end{aligned}
\end{equation*}
For the first term, by Markov's inequality
\begin{equation*}
 \begin{aligned}
  \Pro_{A}\left(|\Lambda^{R}_i(\zeta_n)|>\frac{\epsilon}{2} \, (n-\zeta_n) \right) \leq \frac{\Exp_{A}\left[|\Lambda^{R}_i(\zeta_n)|^{2+\theta}\right]}{(\epsilon/2)^{2+\theta}(n-\zeta_n)^{2+\theta}}
   \leq \frac{\zeta^{2+\theta}_n}{(\epsilon/2)^{2+\theta}(n-\zeta_n)^{2+\theta}}	
 \end{aligned}
\end{equation*}
where $\theta$ as in \eqref{zn2}, and the second inequality follows by Lemma \ref{lem:C5}. The bounding sequence is summable by assumption of \eqref{zn2}. For the second term, by Markov's inequality
\begin{equation}\label{sst1}
		\Pro_{A}\left(|\bar{\Lambda}^{A,\zeta_n}_i(n-\zeta_n)|>\frac{\epsilon}{2} \, (n-\zeta_n) \right) \leq \frac{\Exp_{A}\left[|\bar{\Lambda}^{A,\zeta_n}_i(n-\zeta_n)|^{\mathfrak{p}}\right]}{(\epsilon/2)^{\mathfrak{p}}(n-\zeta_n)^{\mathfrak{p}}}
\end{equation}
Since $\{ \bar{\Lambda}^{A,\zeta_n}_i(n-\zeta_n) \,:\, n \in \bN \}$ is a $\cF^{A,\zeta_n}_{n-\zeta_n}$-martingale, by Rosenthal's inequality \cite[Theorem 2.12]{Hall_Heyde} there is a constant $C>0$ such that
\begin{equation*}
 \Exp_{A}\left[|\bar{\Lambda}^{A,\zeta_n}_i(n-\zeta_n)|^{\mathfrak{p}}\right] \leq C \, (n-\zeta_n)^{\mathfrak{p}/2}
\end{equation*}
and since $\mathfrak{p} > 4$ we prove that \eqref{sst1} is summable.
\end{IEEEproof}

\section{}\label{consist_orde}
In this appendix, we provide the proof of a generalized version of Theorem \ref{thm_c}, which is also used in the proof of the asymptotic optimality results in Appendix \ref{vlbg11}. Throughout this appendix, we fix a subset of anomalies $A \in \cP_{\ell, u}$, and a sampling rule $R$. In order to lighten the notation we do not emphasize the dependence of the various statistics on the sampling rule $R$, thus we write  $\Lambda_i(n), \pi_i(n), \cF_{n}$ instead of  $\Lambda^R_i(n), \pi_i^R(n), \cF^R_{n}$. We recall from Appendix \ref{intrLLR} that  $\cT$ is the family of all $\Pro_{A}$-a.s. finite stopping times with respect to filtration $\left\{ \cF_{n} \,:\, n \in \bN \right\}$, and $\pi_{i}(n:m)$ is the empirical sampling frequency of source $i$ during $[m+1, m+n]$ defined in \eqref{pnm}.\\ 

\begin{theorem}\label{gen_thmc}
 Suppose that the moment condition \eqref{mc1} holds for some $\mathfrak{p} \geq 2$. Let $R$ be an ordering sampling rule. If the functions  $\hat{N}$, $\hat{G}$ satisfy
 \begin{align}\label{NGh_cons1}
 	\hat{N}(D) > |\hat{G}(D)|  
 	\quad \mbox{for all } \; D \in \cP_{l,u} \; \mbox{ such that }\; \hat{G}(D) \neq D \quad \text{and} \quad x(D)>0,
 \end{align}	
 and  the functions $\check{N}$, $\check{G}$ satisfy
 \begin{align}\label{NGch_cons1}
 	\check{N}(D) >|\check{G}(D)|   \quad \mbox{for all } \; D \in \cP_{l,u} \; \mbox{ such that }\; \check{G}(D) \neq D^c \quad \text{and} \quad  y(D)>0, 
 \end{align}	
 then there exist  $\rho \in (0,1)$ and $C>0$ such that for any $A \in \cP_{\ell,u}$, and any $\tau \in \cT$
 \begin{equation}\label{sh_c1}
 	\Pro_{A}\left(\pi_{i}(n:\tau) < \rho \, \big{|} \, \cF_{\tau} \right) \leq C \, n^{-\mathfrak{p}/2},  \qquad \forall \, n \in \bN,
 \end{equation}	
 for every $i \in A$ when  $x(A)>0$, and for every  $i \notin A$ when  $y(A)>0$.\\
\end{theorem}

Theorem \ref{gen_thmc} coincides with Theorem \ref{thm_c} for $\tau=0$. For the purposes of the proof, for any $n \in \bN$ and non-empty set $V\subseteq [M]$, we introduce the event on which the statistics in $A\setminus V$ are positive and greater than those in $V$,  and  the statistics in $A^c \setminus V$ are negative and smaller than those in $V$, i.e., 
\begin{equation}\label{eav}
	E_{A,V}(n):= \bigcap_{i \in A\setminus V} \bigcap_{j \in V}  \bigcap_{z \in  A^c \setminus V }\{ \Lambda_{i}(n) \geq \max\{0,\Lambda_{j}(n) \}  \geq \min\{ 0,\Lambda_{j}(n) \} \geq  \Lambda_{z}(n)  \}. 
\end{equation}
The proof of Theorem \ref{thm_c} relies on two results regarding this event, which are presented in Lemmas  \ref{e_c_g} and \ref{V_Z}. \\

According to the first one, the event $E_{A,V}(n)$ has high probability if at time $n$ there  have been collected relatively few samples from sources in $V$ and relatively many samples from sources not in $V$. To be precise,  for some $\rho,\, \zeta \in (0,1)$, that will  be selected appropriately later, we introduce the event
\begin{equation}\label{ga_v}
	\Gamma_{V}(n):=\left\{ \Pi_{V}(n)<\rho \right\}\bigcap \left\{ \Pi_{V^c}(n)> \zeta  \right\},
\end{equation}
where 
\begin{align*}
\left\{ \Pi_{V}(n)<\rho \right\}:=\{ \pi_{i}(n)<\rho,\, \forall\, i \in V \}, \qquad \left\{ \Pi_{V^c}(n)> \zeta  \right\} := \{ \pi_{i}(n)>\zeta,\, \forall\, i \notin V \}.
\end{align*}

\begin{lemma}\label{e_c_g} 
	Let  $V \subseteq [M]$ and $\rho,\, \zeta \in (0,1)$  such that $\zeta >q_{A}\, \rho$, where $q_{A}$ is defined in \eqref{qa}. Suppose condition \eqref{mc1} holds for some $\mathfrak{p} \geq 2$.  Then, 
	\begin{equation}
		\Pro_{A}\left( \bigcup_{m =\lceil n/2 \rceil}^{n}  \, E^{c}_{A,V}(m:\tau) \cap \Gamma_{V}(m:\tau)  \, \Big{|} \, \cF_{\tau} \right)
	\end{equation}
	is 	uniformly $\mathfrak{p}/2$-polynomially decaying.
\end{lemma}

\begin{IEEEproof}
	For  $m \in [n/2,n]$,  one of the following holds on the event $E^{c}_{A,V}(m:\tau)$:
	\begin{itemize}
		\item $\exists \; i \in A \setminus V$ and $j \in V\cap A$ such that $\Lambda_{i}(m:\tau)<\Lambda_{j}(m:\tau)$,  \\
		
		\item $ \exists \; i \in A \setminus V$ and $k \in V\setminus A$ such that $\Lambda_{i}(m:\tau)<\Lambda_{k}(m:\tau)$,\\
		
		\item  $ \exists \;i \in A \setminus V$ such that $\Lambda_{i}(m:\tau)<0$, \\ 
		
		\item $ \exists \; j \in V\cap A$ and $z \in A^c \setminus V$ such that $\Lambda_{j}(m:\tau)<\Lambda_{z}(m:\tau)$, \\
		
		\item  $\exists \; k \in V\setminus A$ and $z \in A^c \setminus V$ such that $\Lambda_{k}(m:\tau)<\Lambda_{z}(m:\tau)$, \\
		
		\item $\exists \; z \in A^c \setminus V$ such that $\Lambda_{z}(m:\tau)>0$. 
	\end{itemize}
	As a result, by Boole's inequality, it suffices to show that for any $i\, \in\, A\setminus V$, $j\, \in\, V\cap A$, $k\, \in \, V\setminus A$, $z \, \in A^c\setminus V$, 	
	\begin{equation}
		\begin{aligned}
			&\Pro_{A}\left(\exists\, m\in[n/2, n]:\, \Lambda_{i}(m:\tau)< \Lambda_{j}(m:\tau),\, \Gamma_{V}(m:\tau) \, \Big{|}  \, \cF_{\tau} \right) \\
			&\Pro_{A}\left(\exists\, m\in[n/2, n]:\, \Lambda_{i}(m:\tau)< \Lambda_{k}(m:\tau),\, \Gamma_{V}(m:\tau) \, \Big{|} \, \cF_{\tau} \right) \\
			&\Pro_{A}\left(\exists\, m\in[n/2,n]:\,  \Lambda_{i}(m:\tau)< 0,\,\Pi_{V^c}(m:\tau)> \zeta \Big{|} \, \cF_{\tau} \, \right) \\
			&\Pro_{A}\left(\exists\, m\in[n/2, n]:\, \Lambda_{j}(m:\tau) <  \Lambda_{z}(m:\tau),\, \Gamma_{V}(m:\tau) \, \Big{|}  \, \cF_{\tau} \right)\\
			&\Pro_{A}\left(\exists\, m\in[n/2, n]:\, \Lambda_{k}(m:\tau) <  \Lambda_{z}(m:\tau),\, \Gamma_{V}(m:\tau) \, \Big{|} \, \cF_{\tau} \right) \\
			&\Pro_{A}\left( \exists\, m\in[n/2,n]:\, \Lambda_{z}(m:\tau)> 0,\,\Pi_{V^c}(m:\tau)> \zeta \,  \Big{|} \, \cF_{\tau} \right)
		\end{aligned}
	\end{equation}
	are	uniformly $\mathfrak{p}/2$-polynomially decaying. This is indeed the case by Lemma \ref{lambda_tau}.\\
\end{IEEEproof}

According to the second result, if the sampling rule $R$ satisfies conditions \eqref{NGh_cons1}-\eqref{NGch_cons1}, then for any $V \subseteq [M]$ that either intersects $A$ when $x(A)>0$ or $A^c$ when $y(A)>0$, it is very unlike to have few samples from $V$ when the event $E_{A,V}(m)$ occurs for every $m \in [n/2,n]$, especially for large $n$. For the following lemma, we consider the quantity 
\begin{equation*}
 \delta:= 1 \wedge \min\limits_{D}\{\hat{N}(D) - |\hat{G}(D)|\} \wedge \min\limits_{D}\{ \check{N}(D) - |\check{G}(D)| \},
\end{equation*}
where the minimum is considered over all $D \in \cP_{\ell, u}$ such that $\hat{G}(D) \neq D$ and $x(D)>0$ (resp. $\check{G}(D) \neq D^c$ and $y(D)>0$).

\begin{lemma}\label{V_Z}
	Suppose the sampling rule $R$ satisfies conditions \eqref{NGh_cons1}-\eqref{NGch_cons1}, and 
	\begin{equation}\label{cond_vz}
		\mbox{either}  \;\; A\cap V \neq \emptyset 
		\quad \mbox{when} \quad  x(A)>0 \; \;\
		\qquad \mbox{or} \qquad A^c \cap V \neq \emptyset \quad \mbox{when} \; y(A)>0.
	\end{equation} 
	\begin{enumerate}
		\item[(i)] For any $m \in \bN$, on the event $E_{A,V}(m)$, at least one source in $V$ is sampled at time $m+1$ with probability at least  $\delta$.
		
		\item[(ii)] If $\rho < \delta/(2|V|)$, then 
		\begin{equation}\label{cla}
			\Pro_{A}\left( \Pi_{V}(n:\tau)<\rho,\, \bigcap_{m =\lceil n/2 \rceil}^{n} E_{A,V}(m:\tau) \, | \, \cF_{\tau} \right) 
		\end{equation}
		is uniformly exponentially decaying.
	\end{enumerate}
\end{lemma} 

\begin{IEEEproof}
	We prove the lemma when $x(A)>0$ and  $A\cap V \neq \emptyset$. If this  is not true, then by \eqref{cond_vz} it follows that $y(A)>0 $ and $A^c \cap V \neq \emptyset$ and the proof follows in the same way.\\
	
	(i) We fix $m \in \bN$. It is clear that at least one source in $V$ is sampled at time $m+1$ when either $\hat{G}(\Delta_m)$ or  $\check{G}(\Delta_m)$ intersects with $V$. Therefore, it suffices to show that at least one source in $V$ is sampled at time $m+1$, at least with probability $\delta>0$, also on the event 
	\begin{equation}\label{ev1}
		E_{A,V}(m) \cap \{ V \subseteq \hat{G}(\Delta_m)^c \cap \check{G}(\Delta_m)^c \}.
	\end{equation}
	
	By the definition of the ordering sampling rule, it suffices to show that on this event one of the following holds:
	\begin{itemize}
		\item $\hat{N}(\Delta_m)-| \hat{G}(\Delta_m) | \geq \delta$ and the source with the smallest  LLR in $\Delta_m \setminus \hat{G}(\Delta_m) \neq \emptyset$ is in $V$,\\
		\item 
		$\check{N}(\Delta_m) - | \check{G}(\Delta_m) | \geq \delta$ and  the source with the largest LLR in  $(\Delta_m)^c \setminus \check{G}(\Delta_m) \neq \emptyset$ is in $V$.
	\end{itemize}
	Equivalently, in view of conditions \eqref{NGh_cons1}-\eqref{NGch_cons1}, it suffices to show that on the event \eqref{ev1} one of the following holds:
	\begin{itemize}
		\item $x(\Delta_m)>0$   and    the source with the smallest  LLR in $\Delta_m\setminus \hat{G}(\Delta_m) \neq \emptyset$ is  in $V$, \\
		
		\item $y(\Delta_m)>0$ and   the source with the largest  LLR in  $(\Delta_m)^c \setminus \hat{G}(\Delta_m) \neq \emptyset$ is in $V$.
	\end{itemize}
	By definition one of $x(D)$, $y(D)$ is positive, for every $D \in \cP_{\ell, u}$, which implies
	\begin{equation*}
		x(D)=0\;\Rightarrow\; y(D)>0, \quad \mbox{ and } \quad y(D)=0\;\Rightarrow\; x(D)>0.
	\end{equation*}
	Thus, it suffices to show that for every $D \in \cP_{\ell, u}$ and $G,\, G' \subseteq [M]$, on the event
	\begin{equation}\label{ev2}
		E_{A,V}(m) \cap \{ V \subseteq \hat{G}(\Delta_m)^c \cap \check{G}(\Delta_m)^c \}\cap \{\Delta_m=D\} \cap \{ \hat{G}(\Delta_m) = G\} \cap \{ \check{G}(\Delta_m) = G'\}
	\end{equation}
	the following claims hold:
	
	\begin{itemize}
		\item[(a)] If $D \setminus G \neq \emptyset$ and $y(D)=0$, then the source with the smallest LLR in $D \setminus G$ is in $V$,\\
		
		\item[(b)] If $D^c \setminus G' \neq \emptyset$ and $x(D)=0$, then the source with the largest LLR in $D^c \setminus G'$ is in  $V$,\\
		
		\item[(c)] If $x(D)>0$ and $y(D)>0$, then either the source with the smallest LLR in $D \setminus G$ or the source with the largest LLR in $D^c \setminus G'$ is in $V$.\\
	\end{itemize}
	
	In what follows, we assume that the event \eqref{ev2} is non-empty, otherwise the claim holds trivially. We note that this implies that  at least one of  the sets $D \setminus G$ and $D^c \setminus G'$ is non-empty.  Otherwise, we would have $D=G$, $D^c = G'$ and $V=\emptyset$, which would contradict the assumption that $A\cap V \neq \emptyset$. We continue with the proof of each one of the claims (a)-(c) on the event \eqref{ev2}.\\	
	
	(a) Suppose $D \setminus G \neq \emptyset$ and $y(D)=0$.  Then, by Proposition \ref{corl}(ii) it holds $|D|=u$. If $|D|=u >\ell$, by definition of $\Delta_m$ in \eqref{gi_decision_rule}, it follows that $D$ consists of the sources with the $u$ largest \textit{positive} LLRs at time $m$, and by definition of $E_{A,V}(m)$ it follows that at time $m$, we have $A \setminus V \subseteq D$ and $D \cap (A^c \setminus V)=\emptyset$. If $|D|=u =\ell$, by definition of $\Delta_m$ in \eqref{gap_decision rule}, it follows that $D$ consists of the sources with the $u$ largest LLRs at time $m$. By definition of $E_{A,V}(m)$, and since $|A|=|D|$ it follows that at time $m$, we have $D \subseteq (A \setminus V) \cup V \Leftrightarrow D \cap (A^c \setminus V)=\emptyset$. In both cases $D \cap (A^c \setminus V)=\emptyset$. Since $|A| \leq u$ and by assumption $A \cap V \neq \emptyset$, we have $|A \setminus V| < u =|D|$, which further implies that $D\cap V \neq \emptyset$. On the event \eqref{ev2}, we have $V \subseteq G^c \cap (G')^c$, and since $D\cap V \neq \emptyset$, we deduce that the source with the smallest LLR in $D \setminus G$ is in $V$.\\
	
	(b) Suppose $D^c \setminus G' \neq \emptyset$ and $x(D)=0$. Then, by Proposition \ref{corl}(i) it holds $|D|=\ell \Leftrightarrow |D^c|=M-\ell$. By the symmetric argument as that of case (a) we show that $A^c \setminus V \subseteq D^c$ and $D^c \cap (A \setminus V)=\emptyset$. Since $|A|\geq \ell \Leftrightarrow |A^c| \leq M-\ell$, we have $|A^c \setminus V| \leq M-\ell$. It suffices to show that $|A^c \setminus V| < M-\ell$, and since $|D^c|=M-\ell$ this would imply that $D^c\cap V \neq \emptyset$. Then, by the same argument as that of case (a), we deduce that the source with the largest LLR in $D^c \setminus G'$ is in $V$. To prove that $|A^c \setminus V| < M-\ell$, we proceed by contradiction. Indeed, if $|A^c \setminus V| = M-\ell$ then $A^c \setminus V = A^c$ and since $A^c \setminus V \subseteq D^c$ we deduce that $A^c \subseteq D^c$, which further implies that $|A^c|= |D^c| = M-\ell$ and thus $A^c = D^c$ or equivalently $A=D$. The latter implies that $x(A)=x(D)$, but $x(A)$ is assumed to be positive and $x(D)$ to be equal to $0$, which is a contradiction.\\
	
	(c) Let $x(D)>0$ and $y(D)>0$. If $|D|=u$, the result follows from (a), whereas if $|D|=\ell$ the result follows from (b). Thus, it suffices to consider the case where $\ell < |D| < u$. In this case, by definition of $\Delta_m$ in \eqref{gi_decision_rule}, it follows that $D$ consists of the sources with positive LLRs at time $m$. By the definition of $E_{A,V}(m)$ it follows that $A \setminus V \subseteq D$, and $A^c \setminus V \subseteq D^c$, where at least one of the previous two inclusions is strict $(\subset)$ because otherwise we would have $V=\emptyset$. This implies that $D \cap V = \emptyset$ or $D^c \cap V = \emptyset$, and since on the event \eqref{ev2} we have $V \subseteq G^c \cap (G')^c$, we deduce that either the source with the smallest LLR in $D \setminus G$ or the source with the largest LLR in $D^c \setminus G'$ is in $V$.\\ 
	
	(ii) By (i) it follows that there is a sequence $\{ Z_0(m)\,:\, m\in \mathbb{N} \}$ of iid Bernoulli random variables with parameter $\delta$ such that for every $m \in \bN$, 
	\begin{equation}\label{ez}
		E_{A,V}(m) \subseteq \left\{ \sum_{ i \in V} R_{i}(m+1) \geq Z_0(m)  \right\}. 
	\end{equation}	
	Let $\tau \in \cT$, $n \in \mathbb{N}$, and let $W_{V}(n/2;\tau)$ denote the total number of samples from the sources in $V$ during $[\tau+n/2,\tau+n]$, i.e.,
	\begin{equation}
		W_{V}(n/2;\tau):=\sum_{m=\lceil n/2 \rceil }^{n} \sum_{i \in V} R_{i}(\tau+m).
	\end{equation}
	Then, by (i) it follows that 
	\begin{equation}
		\bigcap_{m =\lceil n/2 \rceil}^{n} E_{A,V}(m:\tau) \subseteq \left\{ W_{V}(n/2;\tau) \geq \sum_{m=\lceil n/2 \rceil}^{n} Z_0(\tau+m)  \right\}.
	\end{equation}
	If  $\rho < \delta/(2|V|)$, then there exists $\epsilon \in (0,\delta)$ such that $\rho < (\delta - \epsilon)/(2|V|)$. Consequently, on the event $\{ \Pi_{V}(n:\tau)<\rho \}$ we  have 
	$$W_V(n/2;\tau) \leq \sum_{i \in V} n\,\pi_{i}(n:\tau) \leq \rho\, |V|\, n  \leq  (\delta - \epsilon) \lceil n/2 \rceil. $$
	
	Combining the above, we  obtain
	\begin{equation*}
		\left\{ \Pi_{V}(n:\tau)<\rho,\, \bigcap_{m =\lceil n/2 \rceil}^{n} E_{A,V}(m:\tau)  \right\} \subseteq  \left\{ (\delta - \epsilon) \lceil n/2 \rceil \geq  \sum_{m=\lceil n/2 \rceil}^{n} Z_0(m+\tau) \right\}.
	\end{equation*}
	Consequently, 
	\begin{equation}\label{ofe}
		\begin{aligned}
			&\Pro_{A}\left( \Pi_{V}(n:\tau)<\rho,\, \bigcap_{m =\lceil n/2 \rceil}^{n} E_{A,V}(m:\tau) \Big{|} \cF_{\tau} \right)\\
			&\leq \Pro_{A}\left( (\delta - \epsilon) \lceil n/2 \rceil \geq  \sum_{m=\lceil n/2 \rceil}^{n} Z_0(m+\tau) \Big{|} \cF_{\tau} \right) =	\Pro_{A}\left( (\delta - \epsilon) \lceil n/2 \rceil \geq  \sum_{m=\lceil n/2 \rceil}^{n} Z_0(m) \right),
		\end{aligned}
	\end{equation}
	where the equality holds because $\tau$ is a stopping time, $\{ Z_0(m)\,:\, m\in \bN_{0} \}$ is an iid sequence, and thus by \cite[Theorem 4.1.3]{durrett2010probability} it follows that $\{ Z_0(m+\tau)\,:\, m\in \bN_{0} \}$ is iid, independent of $\cF_{\tau}$, and  has the same distribution as $\{ Z_0(m)\,:\, m\in \bN_{0} \}$. The upper bound is independent of $\tau$, and by the Chernoff bound it follows that it is exponentially decaying, which completes the proof.\\
\end{IEEEproof}

\begin{IEEEproof}[Proof of Theorem \ref{gen_thmc}]
	In order to show \eqref{sh_c} for every $i \in A$ when $x(A) > 0$, and for every $i \notin A$ when $y(A) > 0$, it suffices to show that for each $V \subseteq [M]$ that satisfies \eqref{cond_vz}, there exist constants $\rho \in (0,1)$ and $C>0$ such that for all $n \in \bN$ and $\tau \in \cT$,
	\begin{equation*}
		\Pro_{A}\left( \Pi_{V}(n:\tau)<\rho \, \Big{|} \, \cF_{\tau} \right) \leq C\, n^{-\mathfrak{p}/2}, \quad \mbox{a.s.},
	\end{equation*}
	i.e., the left-hand term is uniformly $\mathfrak{p}/2$-polynomially decaying, as defined in \eqref{pol}. As a result, for  $V$ with size $|V|=1$, we obtain Theorem \ref{thm_c}.
	
	In view of Lemma \ref{V_Z}(ii), it suffices to show that for every $V \subseteq [M]$ that satisfies \eqref{cond_vz}, there exists a $\rho \in (0,1)$ such that
	\begin{equation}\label{kla}
		\Pro_{A}\left( \Pi_{V}(n:\tau)<\rho,\, \bigcup_{m =\lceil n/2 \rceil}^{n} E^{c}_{A,V}(m:\tau) \, \Big{|} \, \cF_{\tau} \right)
	\end{equation} 
	is uniformly $\mathfrak{p}/2$-polynomially decaying. 
	
	To this end, we observe that for any $\rho \in (0,1)$ and $V \subseteq [M]$,
	\begin{equation}\label{r_2r}
		\{ \Pi_{V}(n:\tau)<\rho \} \subseteq \bigcap_{m =\lceil n/2 \rceil}^{n} \{ \Pi_{V}(m:\tau)<2\rho \}.
	\end{equation}
	Indeed, on the event $\{ \Pi_{V}(n:\tau)<\rho \}$, we have $\pi_{i}(n:\tau)< \rho$ for all $i \in V$.  Thus, for any $m \in [n/2,\,n]$ we obtain
	\begin{equation*}
		\begin{aligned}
			(n/2) \,	\pi^{R}_{i}(m:\tau) \leq m \, \pi_{i}(m:\tau)=\sum_{u=1+\tau}^{m+\tau}R_{i}(u)\leq \sum_{u=1+\tau}^{n+\tau}R_{i}(u)=\pi_{i}(n:\tau) \, n<\rho \, n.
		\end{aligned}
	\end{equation*}
	Therefore, by \eqref{r_2r} it follows that the sequence in \eqref{kla} is bounded by
	\begin{equation*}
		\begin{aligned}
			\Pro_{A}&\left( \bigcap_{m =\lceil n/2 \rceil}^{n} \{ \Pi_{V}(m:\tau)<2\rho \},\, \bigcup_{m =\lceil n/2 \rceil}^{n} E^{c}_{A,V}(m:\tau) \, \Big{|} \, \cF_{\tau} \right)\\
			&\leq  \Pro_{A}\left( \bigcup_{m =\lceil n/2 \rceil}^{n} \{ \Pi_{V}(m:\tau)<2\rho \} \cap E^{c}_{A,V}(m:\tau) \, \Big{|} \, \cF_{\tau} \right),
		\end{aligned}
	\end{equation*}
	which is further bounded by the sum
	\begin{equation}\label{up_boun}
		\begin{aligned}
			\Pro_{A}&\left( \bigcup_{m =\lceil n/2 \rceil}^{n} \{ \Pi_{V}(m:\tau)<2\rho,\, \Pi_{V^c}(m:\tau)>\zeta \} \cap E^{c}_{A,V}(m:\tau) \, \Big{|} \, \cF_{\tau} \right)\\
			&+ \Pro_{A}\left( \bigcup_{m =\lceil n/2 \rceil}^{n} \{ \Pi_{V}(m:\tau)<2\rho\} \cap \{\Pi_{V^c}(m:\tau)>\zeta \}^c \, \Big{|} \, \cF_{\tau} \right),
		\end{aligned}
	\end{equation}
	for any choice of $\zeta \in (0,1)$. By Lemma \ref{e_c_g}, the first term in \eqref{up_boun} is uniformly $\mathfrak{p}/2$-polynomially decaying for any $\rho,\, \zeta >0$ such that $\zeta > q_{A}2\rho$. 
	
	However, in order to show that the second term in \eqref{up_boun} is also uniformly $\mathfrak{p}/2$-polynomially decaying, $\zeta$ and $\rho$ must be selected depending on the size of $V$, i.e., $\rho_{v}$, $\zeta_{v}$ for each $v \in [M]$. Of course, for each $v \in [M]$, $\rho_{v}$, $\zeta_{v}$ must satisfy
	\begin{equation}\label{cndv}
		\rho_{v}<\delta/(2v),\quad \mbox{ and } \quad  \zeta_{v} > 2 \rho_{v}q_{A},
	\end{equation}
	as the first condition guarantees that \eqref{cla} is uniformly exponentially decaying by Lemma \ref{V_Z}(ii), and the second  guarantees that the first term in \eqref{up_boun} is uniformly $\mathfrak{p}/2$-polynomially decaying by Lemma \ref{e_c_g}. We will show that for each $v \in [M]$, if  we select $\rho_{v}$ and $\zeta_{v}$ such that 
	in addition to \eqref{cndv},
	\begin{equation}\label{mq}
		\max\{4\rho_{v},\, 2\zeta_{v} \}< \rho_{v+1},
	\end{equation}
	where $\rho_{M+1}:=\infty$, then the second term in \eqref{up_boun} is uniformly $\mathfrak{p}/2$-polynomially decaying for every $V \subseteq [M]$ with $|V|=v$ that satisfies \eqref{cond_vz}. In order to show this, we will apply a backward  induction argument, starting from $v=M$ down to $v=1$. \\
	
	If $v=M$, then $V=[M]$ or, equivalently, $V^c=\emptyset$, and  as a result, the second term in \eqref{up_boun} is trivially equal to zero for any $\rho,\, \zeta \in (0,1)$. Thus, in order to guarantee that 
	$$ \Pro_{A}\left(\Pi_{[M]}(n:\tau)< \rho_{M} \, | \, \cF_{\tau} \right)$$ 
	is uniformly $\mathfrak{p}/2$-polynomially decaying, it suffices to select $\rho_{M},\, \zeta_{M}$ that satisfy \eqref{cndv}.
	
	Now, \textit{suppose that the claim holds for $v+1$}, i.e., there exist $\rho_{v+1},\, \zeta_{v+1} \in (0,1)$ such that the second term in \eqref{up_boun} is uniformly $\mathfrak{p}/2$-polynomially decaying for any $|V|=v+1$ that satisfies \eqref{cond_vz}. We will \textit{show that the claim also holds for $v$}. For this, we fix $\rho_{v}$, $\zeta_{v}$ that satisfy \eqref{cndv} and \eqref{mq}. In order to prove the claim for $v$, it suffices to show that 
	\begin{equation}\label{step_i}
		\begin{aligned}
			\bigcup_{m =\lceil n/2 \rceil}^{n} \{ \Pi_{V}(m:\tau)<2\rho_{v}\} \cap \{\Pi_{V^c}(m:\tau)>\zeta_{v} \}^c \subseteq \bigcup_{j \notin V}\left\{ \Pi_{V\cup \{j\} }(\lceil n/2 \rceil : \tau)  < \rho_{v+1} \right\},
		\end{aligned}
	\end{equation}
	as this further implies,  by Boole's inequality, that the second term in \eqref{up_boun} is bounded by 
	\begin{equation}\label{bol}
		\sum_{j \notin V} \Pro_{A}\left( \Pi_{V\cup \{j\} }(\lceil n/2 \rceil : \tau)  < \rho_{v+1} \, \Big{|} \, \cF_{\tau}\right).
	\end{equation}
	Since $V$ satisfies \eqref{cond_vz} and $V \subseteq V \cup \{j\}$, then $V \cup \{j\}$ satisfies \eqref{cond_vz}, and we also have  $|V \cup \{j\}|=v+1$ for every $j \notin V$. A direct implication of the induction hypothesis is that for any $|V|=v+1$ that satisfies \eqref{cond_vz},
	\begin{equation*}
		\Pro_{A}\left( \Pi_{V}(n:\tau)<\rho_{v+1} \, \Big{|} \, \cF_{\tau} \right)
	\end{equation*}
	is uniformly $\mathfrak{p}/2$-polynomially decaying, and as a result, each term in \eqref{bol} is uniformly $\mathfrak{p}/2$-polynomially decaying. Hence, we have concluded the step of the induction.
	
	It remains to show that \eqref{step_i} holds for any $\rho_{v}$, $\zeta_{v}$ that satisfy \eqref{mq}. Suppose that there exist  $m \in [n/2,n]$ and $j \notin V$ such that
	\begin{equation}
		\Pi_{V}(m: \tau)<2\rho_{v} \quad \mbox{and} \quad \pi_{j}(m: \tau) \leq \zeta_{v}.
	\end{equation}
	Then, in of view of \eqref{mq} we have
	\begin{equation}
		\begin{aligned}
			\pi_{j}(\lceil n/2 \rceil : \tau)\lceil n/2 \rceil &=\sum_{u=1+\tau}^{\lceil n/2 \rceil + \tau} R_{j}(u) \leq \sum_{u=1+\tau}^{m+ \tau} R_{j}(u) = \pi_{j}(m: \tau) \, m \\
			&\leq \zeta_{v} \, n < \rho_{v+1} \lceil n/2 \rceil,
		\end{aligned}
	\end{equation}
	and
	\begin{equation}
		\begin{aligned}
			\pi_{i}(\lceil n/2 \rceil: \tau)\lceil n/2 \rceil&=\sum_{u=1+\tau}^{\lceil n/2 \rceil+ \tau} R_{i}(u) \leq \sum_{u=1+\tau}^{m+ \tau} R_{i}(u) =\pi_{i}(m: \tau) \, m \\
			&\leq 2 \rho_{v} n < \rho_{v+1} \lceil n/2 \rceil, \quad \forall\, i\, \in\, V,\\
		\end{aligned}
	\end{equation}
	which together imply \eqref{step_i}.
\end{IEEEproof}

\section{}\label{vlbg11}
In this appendix, we prove Theorem \ref{prV}. In what follows we fix a set $A \in \cP_{l,u}$, and an ordering sampling rule $R$. Theorem \ref{prV} claims that for the sources in $A \setminus \hat{G}(A)$ the expected value of the distance between any two LLRs defined in \eqref{llrA}, for the sampling rule $R^{A,m}$ defined in \eqref{repreA}, is relatively small as described in \eqref{vlth}. The rule $R^{A,m}$ samples at time $u+1$ the sources
in $A \setminus \hat{G}(A)$ with the
\begin{equation}\label{flna}
\lfloor \hat{N}(A)  \rfloor -|\hat{G}(A) |+ \mathbf{1} \{\hat{Z}_{m+u} \leq \hat{N}(A) - \lfloor \hat{N}(A)  \rfloor \}
\end{equation}
smallest LLRs at time $u$. For the purposes of the proof we note that \eqref{flna} is equivalent to sampling at time $u+1$ the sources in $A \setminus \hat{G}(A)$ with the
\begin{equation}\label{flna2}
 \lceil \hat{N}(A) \rceil  -|\hat{G}(A)| -1 + \mathbf{1} \{ \hat{Z}_{m+u} \leq  \mathfrak{q}_A\}
\end{equation}
smallest LLRs at time $u$, where
\begin{equation}\label{qaa}
\mathfrak{q}_A:=		
\begin{cases}
	\hat{N}(A) - \lfloor \hat{N}(A) \rfloor,& \text{if} \;   \hat{N}(A) > \lfloor \hat{N}(A) \rfloor,  \\
	1,& \text{if} \; \hat{N}(A) = \lfloor \hat{N}(A) \rfloor.	
\end{cases}
\end{equation}
We observe that the number of sources we sample from $A \setminus \hat{G}(A)$ at time $u+1$ remains the same. This is because for any positive real number $x$, we have $x=\lceil x \rceil = \lfloor x \rfloor$ when $x$ is an integer, and $\lceil x \rceil = \lfloor x \rfloor +1$ otherwise. We prove Theorem \ref{prV} as a special case of a more general result whose proof proceeds by induction on the maximum number of sources sampled at each step, which is denoted by $\lambda \in \{1,\ldots,\lceil \hat{N}(A) \rceil  -|\hat{G}(A)|\}$. To carry out the induction, we introduce a more general sampling rule, denoted by $\cR$, that encompasses the family $\{R^{A,m} \, :\, m\in \bN \}$. Enlarging the probability space as needed the sampling rule
\begin{align}\label{newR}
	\cR \equiv \cR(\lambda, \mathfrak{q}, \cD, W, Y, Z)
\end{align}
samples at time $n+1$,
\begin{enumerate}
	\item [(i)] the $\lambda -1$ sources in the set $\mathcal{D}$ with the smallest LLRs at time $n$,
	
	\item [(ii)] and with probability $\mathfrak{q}$, the source in the set $\cD$ with the $\lambda^{th}$ smallest LLR at time $n$.
\end{enumerate}
For each $i \in [M]$, the LLR associated with the rule $\cR$ is defined as
\begin{equation*}
	\Lambda^{\cR}_{i}(n):= W_i + \sum_{m=1}^n \log \left( \frac{f_{1i} (Y_i(m)) }{f_{0i} (Y_i(m)) } \right) \,  \mathbf{1}\left\{ i \in \cR \right\}, 
\end{equation*}	 
and
\begin{itemize}
	\item $\lambda$ is an integer in $\{ 1,\ldots, \lceil \hat{N}(A) \rceil  -|\hat{G}(A)|\}$.
	
	\item $\mathfrak{q}$ is a number in $[0,1]$.
	
	\item $\mathcal{D}$ is a random set taking values in
	\begin{equation}\label{frakD}
		\left\{ D \subseteq A \setminus \hat{G}(A) :\, (\lambda - 1) + \mathfrak{q} < \sum_{i \in D} \frac{I^*(D)}{I_i} \right\},
	\end{equation} 
	and \eqref{frakD} must be non-empty in order for $\cR$ to be well-defined.
	
	\item $W:=\left\{ W_i :\, i \in [M] \right\}$, where $W_i$ is a real-valued random variable that stands for the initial value of the LLR of source $i$, i.e., $\Lambda^{\cR}_{i}(0):= W_i$, and it is not necessarily equal to $0$.
	
	\item $Y:=\{ Y_i :\, i \in [M] \}$, where $Y_i :=\{ Y_i(n)\,:\, n \in \bN \}$ is a data sequence of iid random variables with density $f_{i1}$, from which we take measurements according to rule $\cR$. For each $i \in [M]$, $Y_i$ has the same distribution as $X_i$ in \eqref{xi}, but they are not necessarily identical. The values of  $\cD$ and $W$ are assumed to be generated before we start observing the data of $Y$, and they are independent of $Y$.
	
	\item $Z:=\{ Z_n \, : \, n \in \bN_{0} \}$ is a sequence of independent, Uniform$[0,1]$ random variables, which are used for randomization purposes.  For each $n \in \bN$, the source in $\cD$ with the $\lambda^{th}$ smallest LLR at time $n$, is sampled at time $n+1$ if and only if $Z_n\leq \mathfrak{q}$. The sequence $Z$ is independent of $Y,W,\cD$.
\end{itemize}

Furthermore, we denote by $\left\{ \cF^{\cR}_{n}\,:\, n \in \bN_{0} \right\}$ the filtration induced by $\cR$, that is 
\begin{equation*}
\cF^{\cR}_{n} := \begin{cases}
	\sigma\bigg( Z_0,\mathcal{D},\left\{ W_i \,:\, i \in \mathcal{D} \right\} \bigg), & n =0,\\
	\sigma\bigg(\cF^{\cR}_{n-1} ,Z_n,\left\{ Y_{i}(n)\,:\, i \in \cR(n) \right\} \bigg), & n \in \bN,
\end{cases}
\end{equation*}
and by $\cT^{\cR}$ the class of $\Pro_{A}$-a.s. finite stopping times with respect to $\left\{ \cF^{\cR}_{n}\,:\, n \in \bN_{0} \right\}$. For each $n \in \bN_0$, the set $\cR(n+1)$ contains the 
\begin{equation*}
\lambda -1 + \mathbf{1}\{ Z_n\leq \mathfrak{q}\}
\end{equation*}
sources in $\cD$ with the smallest LLRs, and thus the sampling rule $\cR$ can be viewed as an ordering rule, which in comparison to Definition \ref{definition_ordering}, it has
\begin{equation}\label{rnd}
	\hat{N}(\mathcal{D})=\lambda -1+ \mathfrak{q}  \qquad \mbox{and} \qquad \hat{G}(\mathcal{D})=\emptyset.
\end{equation}
where the fact that $\hat{G}(\mathcal{D})=\emptyset$ follows by \eqref{frakD}. Consequently, all the results developed for the ordering rules such as the results of Appendix \ref{intrLLR}, and Theorem \ref{gen_thmc} can be applied in the development of the results for $\cR$. In order to simplify the notation we will refer to a sampling rule only by $\cR$ without repeating its arguments as long as they remain fixed, and they will be restored when we need to distinguish between rules with different arguments. In order to simplify the notation, we suppress the dependence on $\cR$, and for example we write $\Lambda_i, \cF, \cT$ instead of $\Lambda_i^{\cR}, \cF^{\cR}, \cT^{\cR}$ and we will restore the notation only when we need to distinguish between different rules.

In the following proposition we show how we must choose the arguments of $\cR$ so that the rules $R^{A,m}$ and $\cR$ coincide.

\begin{proposition}\label{eqv}
	We fix $m \in \bN$ and the rule $R^{A,m}$. For the sampling rule
	\begin{equation*}
		\cR \equiv \cR( \lceil \hat{N}(A) \rceil  -|\hat{G}(A)|, \mathfrak{q}_A,A \setminus \hat{G}(A), \Lambda_i^R(m), X,\hat{Z}),
	\end{equation*}
	it holds
	\begin{equation}\label{rmr}
		\cR(n) = R^{A,m}(n)\bigcap (A \setminus \hat{G}(A)), \quad \forall \, n \in \bN_{0}. 
	\end{equation}
\end{proposition}

\begin{IEEEproof}
	First, we need to show that the rule $\cR$ considered in \eqref{rmr} is well-defined in the sense that \eqref{frakD} is non-empty, or equivalently that for $D=A \setminus \hat{G}(A)$ it holds 
	\begin{equation*}
	\lceil \hat{N}(A) \rceil  -|\hat{G}(A)| -1 + \mathfrak{q}_{A} < \sum_{ i \in A \setminus \hat{G}(A)} \frac{I^*(A \setminus \hat{G}(A))}{I_i}.
	\end{equation*}
	Indeed, in both cases for $\mathfrak{q}_{A}$ in \eqref{qaa} we deduce that
	\begin{equation*}
		\hat{N}(A) -|\hat{G}(A)| < \sum_{i \in A \setminus \hat{G}(A)} \frac{I^*(A \setminus \hat{G}(A))}{I_i},
	\end{equation*}
	which is the assumption \eqref{mchat} of Theorem \ref{prV}. It remains to show \eqref{rmr}. By definition the rule $R^{A,m}$ at time $u+1$ samples the sources in $A \setminus \hat{G}(A)$ with the \eqref{flna} smallest LLRs, and $\cR$ the sources in $A \setminus \hat{G}(A)$ with the \eqref{flna2} smallest LLRs. Since \eqref{flna} and \eqref{flna2} are equal we show the claim.
\end{IEEEproof}

\subsection{The more general theorem}
For the statement of the more general theorem, we consider a finite sequence of decreasing positive numbers $\{ p_i \, :\, 1 \leq i \leq \lceil \hat{N}(A) \rceil  -|\hat{G}(A)| \}$ that will be determined explicitly later, and also a slightly larger sequence of decreasing positive numbers $\{ p^{+}_i \, :\, 1 \leq i \leq \lceil \hat{N}(A) \rceil  -|\hat{G}(A)| \}$ such that
\begin{equation*}
0< \ldots < p_{i+1} < p_{i+1}^{+} < p_i < \cdots < p_1 < p_1^{+} < \mathfrak{p}-1.
\end{equation*}
We proceed to the statement of the more general theorem from which we deduce Theorem \ref{prV}.

\begin{theorem}\label{Vl_b}
	For any sampling rule $\cR \equiv \cR(\lambda, \mathfrak{q}, \cD, W, Y,Z)$ such that
	\begin{equation}\label{s_a}
		\max_{i,j \in \mathcal{D}} |W_{i}-W_{j}| \in  \cL^{p_{\lambda}^{+}},
	\end{equation}
	there exists a strictly increasing sequence of random times $\{ \sigma_l\,:\, l \in \bN_{0} \} \in \cT^{\cR}$ with $\sigma_0:=0$ such that for the sequence 
	\begin{equation}\label{de1}
		V_l:=\max\limits_{\sigma_l \leq n < \sigma_{l+1}}\left( \max\limits_{i,j \in \mathcal{D}}\big{|}\Lambda^{\cR}_{i}(n)-\Lambda^{\cR}_{j}(n)\big{|} \right), \quad \forall \, l \in \bN_{0},
	\end{equation}
	there is a constant $C > 0$ independent of $l$ and of $\{ W_{i} \,:\, i \in \cD \}$ such that
	\begin{equation}\label{vlts}
		\sup_{l \in \bN_0} \Exp_A \left[ V^{p_{\lambda}}_l \right] \leq C\, \left( 1 + \Exp_A \left[  \max_{i,j \in \cD} \left| W_{i}-W_{j} \right|^{p_{\lambda}^+} \right] \right).
	\end{equation}
\end{theorem}

\textit{In what follows, when we refer to a constant we also imply independent of $l$, $\cF^{\cR}_0$, and of any stopping time in $\cT$.} In view of Proposition \ref{eqv}, we observe that if we choose $\mathfrak{p}$ that satisfies the moment condition \eqref{pfrk} of Theorem \ref{prV}, and also for $\lambda= \lceil \hat{N}(A) \rceil  -|\hat{G}(A)|$ we choose
\begin{equation}\label{fst}
	p_{\lambda} := 3 + \theta, \qquad p^{+}_{\lambda} :=  3 +\theta^{+},
\end{equation}
then we deduce Theorem \ref{prV}. The proof of Theorem \ref{Vl_b} is based on an inductive argument. We first show the claim for $\lambda =1$, and then inductively for all $\lambda \in \{2, \ldots, \lceil \hat{N}(A) \rceil  -|\hat{G}(A)|\}$. At each step of the induction, we establish the result for a moment less than the half of that of the previous step, i.e.,
\begin{equation}\label{pind}
	p_{i+1} < p_{i}/2, \qquad i \in \{1,\ldots,\lceil \hat{N}(A) \rceil -|\hat{G}(A)|\}.
\end{equation}
In order to satisfy the condition $p_1^{+} < \mathfrak{p}-1$, the final step requirement \eqref{fst}, and the requirement \eqref{pind}, we consider a sequence of decreasing positive numbers $\{ \theta_i \,:\, 0 \leq i \leq \lceil \hat{N}(A) \rceil -|\hat{G}(A)| \}$ such that
\begin{equation*}
0 < \ldots < \theta_{i+1} < \theta_{i} < \ldots < \ldots < \theta_{0} < (\mathfrak{p}-1)/2^{ \lceil \hat{N}(A) \rceil -|\hat{G}(A)|-1} -3,
\end{equation*}
and we define
\begin{equation*}
 \begin{aligned}
  p_i &:= 2^{ \lceil \hat{N}(A) \rceil -|\hat{G}(A)|-i} (3+\theta_i), \quad i \in \{0,\ldots, \lceil \hat{N}(A) \rceil -|\hat{G}(A)|\},\\
  p^{+}_{i+1} &:= p_i /2, \qquad \qquad \qquad \qquad \quad \; i \in \{0,\ldots, \lceil \hat{N}(A) \rceil -|\hat{G}(A)|-1\}.
 \end{aligned}
\end{equation*}
We verify that $p^{+}_{1}:=p_0/2 < \mathfrak{p}-1$, and \eqref{fst} is satisfied for $\theta = \theta_{\lceil \hat{N}(A) \rceil -|\hat{G}(A)|}$ and $\theta^{+} = \theta_{\lceil \hat{N}(A) \rceil -|\hat{G}(A)|-1}$. 

Let us fix a sampling rule $\cR \equiv \cR(\lambda, \mathfrak{q}, \cD, W, Y,Z)$. For each $l \in \bN_0$, we bound $V_{l}$ by
\begin{equation*}
	V_l\leq  H(\sigma_l)+  B(\sigma_l) +U(\sigma_l, \sigma_{l+1}),
\end{equation*}
where
\begin{itemize}
	\item $H(n)$ is the maximum distance between any two LLRs in $\cD$  at time $n \in \bN_{0}$, i.e.,
	\begin{equation}\label{eta0}
		H(n):= \max_{i,j \in \cD} \big{|} \Lambda_{i}(n) - \Lambda_{j}(n) \big{|},
	\end{equation}
	
	\item $B(n)$ the maximum draw-down of the LLRs of all sources in $\cD$ starting from time $n \in \bN_{0}$, i.e.,
	\begin{equation}\label{beta01}
		B(n):=\max_{i \in \cD}\left(\Lambda_{i}(n)-\inf_{m \geq n}\Lambda_{i}(m) \right),
	\end{equation}
	
	\item $U(n,m)$ the difference between the maximum LLR at time $n \in \bN_{0}$ and the maximum LLR at time $m-1$, where $m\in \bN$ and $m>n$, i.e.,
	\begin{equation}\label{ups}
		U(n,m):=\max_{i \in \cD} \Lambda_{i}(m-1) - \max_{i \in \cD} \Lambda_{i}(n).
	\end{equation}	      	 
\end{itemize}

Therefore, during the random interval $[\sigma_l, \sigma_{l+1})$, $V_l$ is bounded by the sum of the maximum distance of the LLRs at time $\sigma_l$, the maximum draw-down of the LLRs starting from time $\sigma_l$, and the increase in the maximum LLR between the times $\sigma_l$ and $\sigma_{l+1} -1$. In order to show \eqref{vlts}, by application of Jensen's inequality it follows that it suffices to show that this is the case for the following three terms,
\begin{equation*}
\begin{aligned}
\sup_{l \in \bN_0} \Exp_A \left[ H^{p_{\lambda}}(\sigma_l) \right],\; \sup_{l \in \bN_0} \Exp_A \left[ B^{p_{\lambda}}(\sigma_l) \right],\; \sup_{l \in \bN_0} \Exp_A \left[ U^{p_{\lambda}}(\sigma_l, \sigma_{l+1}) \right]\leq C\,\left( 1 + \Exp_A \left[  \max_{i,j \in \cD} \left| W_{i}-W_{j} \right|^{p_{\lambda}^+} \right] \right).
\end{aligned}
\end{equation*}
As the proof progresses we will see that for $\lambda \geq 2$ in order to prove the claim for $\{ U(\sigma_l, \sigma_{l+1})\,:\, l\in \bN_{0} \}$ we must show a slightly stronger claim for $\{ H(\sigma_l)\,:\, l\in \bN_{0} \}$, i.e., there is a constant $C >0$ such that
\begin{equation*}
	\begin{aligned}
		\sup_{l \in \bN_0} \Exp_A \left[ H^{p^{+}_{\lambda}}(\sigma_l) \right]\leq C\, \left( 1 + \Exp_A \left[  \max_{i,j \in \cD} \left| W_{i}-W_{j} \right|^{p_{\lambda}^+} \right] \right).
	\end{aligned}
\end{equation*}
which implies the former because $\Exp_A \left[ H^{p_{\lambda}}(n) \right] \leq \Exp_A \left[ H^{p^{+}_{\lambda}}(n) \right] +1$, for all $n \in \bN_{0}$. For the uniformity of the results we show this also for the case $\lambda =1$. By Lemma \ref{dd_l}, it follows that for any choice of $\{\sigma_l\,:\, l \in \bN_{0} \} \in \cT$, the sequence $\{ B(\sigma_l)\,: \, l\in \bN_{0} \}$ is bounded in $\cL^{p_{\lambda}}$ by a constant. However, this is not the case for $\{ H(\sigma_l)\,:\, l\in \bN_{0} \}$,  $\{ U(\sigma_l, \sigma_{l+1})\,:\, l\in \bN_{0} \}$ which are bounded in $\cL^{p^+_{\lambda}}$, $\cL^{p_{\lambda}}$ by a term that depends on $\Exp_A \left[  \max_{i,j \in \cD} \left| W_{i}-W_{j} \right|^{p_{\lambda}^+} \right]$. For example, when $l=0$ by definition \eqref{eta0} we have
\begin{equation*}
 \Exp_A \left[ H^{p^{+}_{\lambda}}(0) \right] = \Exp_A \left[  \max_{i,j \in \cD} \left| W_{i}-W_{j} \right|^{p^{+}_{\lambda}} \right]. 
\end{equation*}
Therefore, we need to define the sequence $\{\sigma_l\,:\, l \in \bN_{0} \}$ in a way that enables us to show 
\begin{equation}\label{ab2}
	\begin{aligned}
		\sup_{l \in \bN_0} \Exp_A \left[ H^{p^{+}_{\lambda}}(\sigma_l) \right],\;\; \sup_{l \in \bN_0} \Exp_A \left[ U^{p_{\lambda}}(\sigma_l, \sigma_{l+1}) \right]\leq C\, \left( 1 + \Exp_A \left[  \max_{i,j \in \cD} \left| W_{i}-W_{j} \right|^{p_{\lambda}^+} \right] \right).
	\end{aligned}
\end{equation}

\subsection{The special case $\lambda=1$} 
In this subsection we present the proof of Theorem \ref{Vl_b} for the case $\lambda=1$, which is the basis of our proof by induction for the general case $\lambda >1$. Also, it is the case on which the bibliography concerning the existing ordering sampling rules was focused on, under a particular setup with known number of anomalies as described in Subsection \ref{sec:special}. Our proof differs from the existing ones, and it is based on Lorden's excess inequality \cite[Theorem 3]{lorden1970excess}.

For the case $\lambda=1$, we define the $\{\sigma_l\,:\, l \in \bN_{0}\}$ as the sequence of times where the maximum LLR in $\cD$ changes. For this, we define the operator $J: \cT \to \cT$ that matches each $\nu \in \cT$ to the first time after $\nu$ that a change in the maximum LLR occurs, i.e.,
\begin{equation}\label{ttl}
	J(\nu):= \inf\left\{n > \nu:\, \max_{i \in \cD}\Lambda_{i}(n) > \max_{i \in \cD}\Lambda_{i}(\nu)  \right\},
\end{equation}
and the sequence $\{ \sigma_l\, :\, l \in \bN_{0}\}$ is defined recursively as
\begin{equation}\label{slp1}
	\sigma_{l+1}:=J(\sigma_l), \quad l \in \bN_{0}, \quad \mbox{where}\quad \sigma_0:=0.
\end{equation}
Under this definition of $\{\sigma_l\,:\, l \in \bN_{0}\}$ the maximum LLR remains the same during $[\sigma_l,\sigma_{l+1})$, which implies that
\begin{equation}\label{u0}
	U(\sigma_{l},\sigma_{l+1})=0, \quad \forall \; l \in \bN_{0}.
\end{equation}
Therefore, in order to prove \eqref{vlts}, it suffices to show \eqref{ab2} only for $\sup_{l \in \bN_0} \Exp_A \left[ H^{p^+_{1}}(\sigma_l) \right]$.
\begin{IEEEproof}[Proof of Theorem \ref{Vl_b} for $\lambda=1$]
First, we show by induction that $\{\sigma_{l}\,:\, l\in \bN_{0}\}$ is well-defined, i.e., $\{ \sigma_{l}\,:\, l\in \bN_{0} \} \in \cT$. By definition $\{\sigma_{l}\,:\, l\in \bN_{0}\}$ is a sequence of stopping times with respect to $\{ \cF_{n} \,:\, n \in \bN \}$, and it remains to show that $\sigma_{l}$ is $\Pro_{A}$-a.s. finite, for all $l\in \bN_{0}$. Indeed,	$\sigma_{0}=0 \in \cT$ and if $\sigma_{l} \in \mathcal{T}$ for some $l \in \bN_{0}$, then by Lemma \ref{tu}(i) there is a constant $C>0$ such that
\begin{equation*}
	\Exp_{A}[\sigma_{l+1} - \sigma_{l}] \leq C,
\end{equation*}
which implies that $\sigma_{l+1}$ is $\Pro_{A}$-a.s. finite, and thus $\sigma_{l+1} \in \mathcal{T}$.
	
Second, in view of \eqref{u0}, in order to prove \eqref{ab2} it suffices to show that there is a constant $C>0$ such that
\begin{equation*}
\sup_{l \in \bN_0} \Exp_A \left[ H^{p^{+}_{1}}(\sigma_l) \right] \leq C\, \left( 1 + \Exp_A \left[ \max_{i,j \in \cD} \left| W_{i}-W_{j} \right|^{p^{+}_{1}} \right] \right).
\end{equation*}
For $l=0$, by definition \eqref{eta0} we have
\begin{equation}\label{tta}
	\Exp_A \left[ H^{p^{+}_{1}}(0) \right] = \Exp_A \left[  \max_{i,j \in \cD} \left| W_{i}-W_{j} \right|^{p^{+}_{1}} \right]. 
\end{equation}
For $l \geq 1$, since $\{ \sigma_{l}\,:\, l\in \bN_{0} \} \in \cT$ and $p^{+}_{1} < \mathfrak{p}-1$ by Lemma \ref{tu}(iv), we conclude that there is a constant $C_{0}>0$ such that
\begin{equation}\label{ttb}
	\sup_{l \in \bN}\Exp_{A}\left[H^{p^{+}_{1}}(\sigma_{l}) \, | \, \cF_{\sigma_{l-1}}\right] \leq C_0.
\end{equation}
Adding up the upper bounds in \eqref{tta}, \eqref{ttb}, we show the claim for $C:=\max\{ 1, C_0\}$.
\end{IEEEproof}

\subsection{The general case $\lambda > 1$}
In this subsection, we provide the proof of Theorem \ref{Vl_b} for the general case $\lambda >1$, and in Subsection \ref{spl}, we include the auxiliary lemmas that support the proof. First, we define the sequence $\{\sigma_{l}\,:\, l\in \bN_{0}\}$ and then we proceed to the proof of Theorem \ref{Vl_b} for the defined sequence. The definition of $\{\sigma_{l}\,:\, l\in \bN_{0}\}$ is more complicated compared to that of case $\lambda =1$, because in the general case more than one sources are sampled simultaneously. For fixed $\lambda$, the definition of $\sigma_{l+1}$, given $\sigma_{l} \in \mathcal{T}$, is the same for all $l \in \bN_0$. For instance for $\lambda =2$, let us fix $l \geq 0$ and $\sigma_{l} \equiv \nu \in \cT$. Then, $\sigma_{l+1}$ is defined as the last element of a sequence of three stopping times
\begin{equation*}
\nu < \nu_1 < \nu_2,
\end{equation*}
where $\sigma_{l} \equiv \nu$, and $\sigma_{l+1}:=\nu_2$. The intermediate time is defined as $\nu_1 := J(\nu)$, where $J$ is defined in \eqref{ttl}, and it is the first time after $\nu$ that a change in the maximum LLR occurs. The last time $\nu_2$ is the first time after $\nu_1$ that the smallest LLR exceeds one of the others. For instance, if $u$ is the identity of the source with the smallest LLR at time $\nu_1$ then $\nu_{2} := F_{1}(\nu_1)$, where
\begin{equation*}
	F_{1}(\nu_1):= \inf \left\{ n > \nu_1 :\, \Lambda_{u}(n) > \min_{i \neq u} \Lambda_{i}(n) \right\}.
\end{equation*}
\begin{center}
	\begin{figure}[h]
		\subfloat[LLRs at time $\nu$]{
			\includegraphics[width=0.3\linewidth]{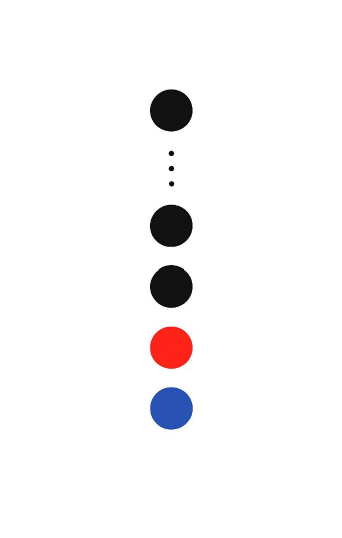}  
			\label{fig:nu}
		}
		\subfloat[LLRs at time $\nu_1$]{
			\includegraphics[width=0.3\linewidth]{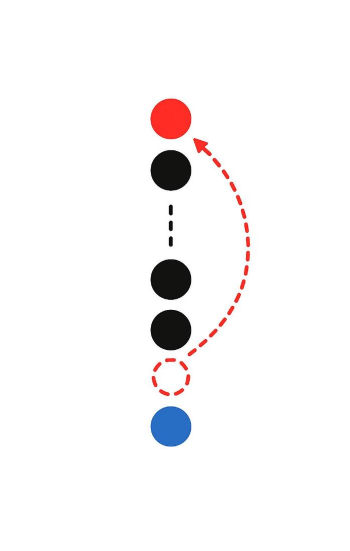}
			\label{fig:nu1} 
		} 
		\subfloat[LLRs at time $\nu_{2}$]{
			\includegraphics[width=0.32\linewidth]{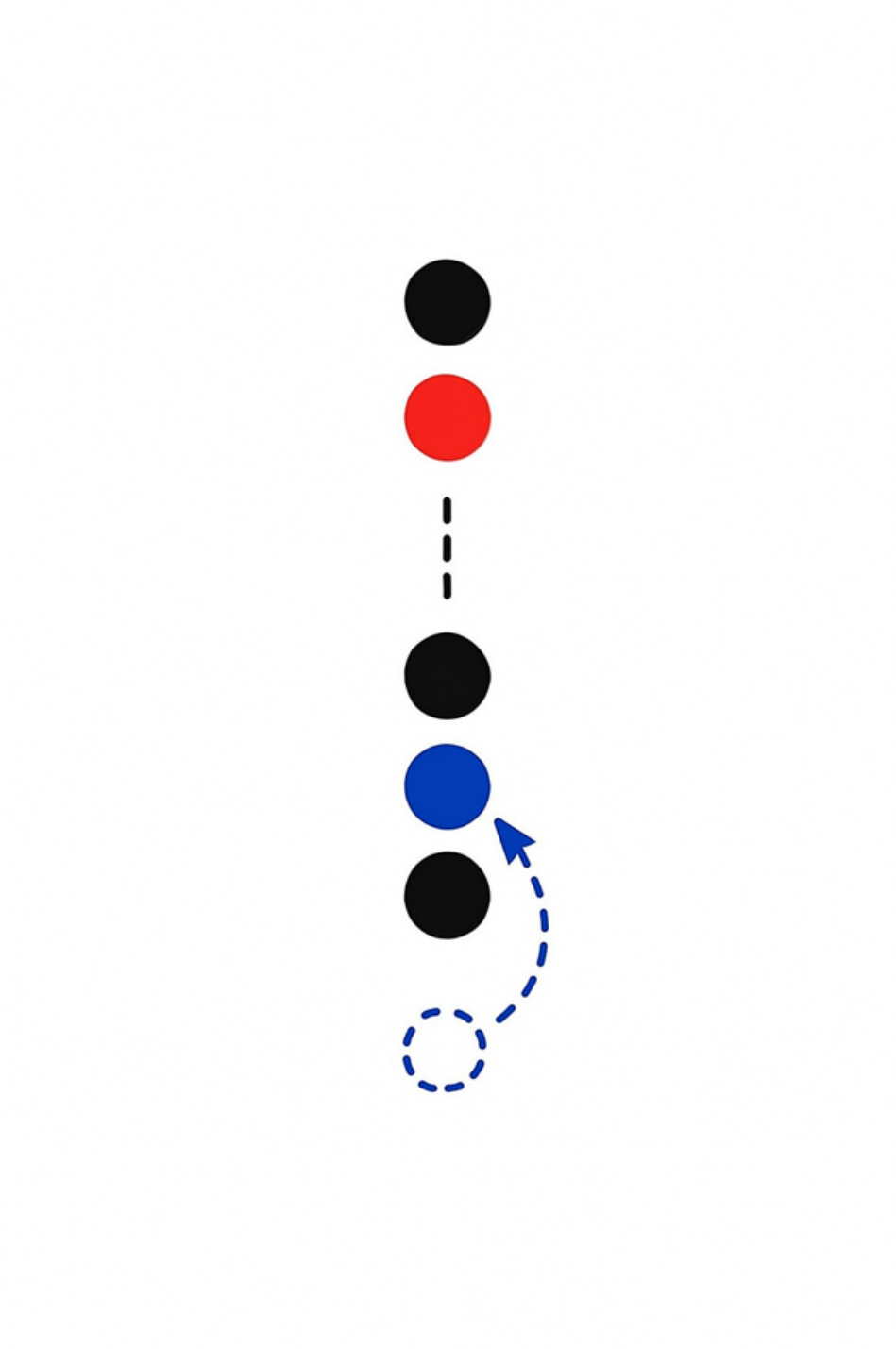}
			\label{fig:nu2} 
		} 
		\captionsetup{justification=raggedright, singlelinecheck=false}
		\caption{Indicative figure for the definition of $\{\sigma_{l}\,:\, l\in \bN_{0}\}$ for $\lambda=2$.}
	\end{figure}
\end{center}

To give some intuition behind this definition we provide an indicative figure. In Figure \ref{fig:nu}, we depict the values of the LLRs as balls, where the black balls are not sampled whereas the red and the blue are sampled because the represent the two smallest LLRs. Figure \ref{fig:nu1} is indicative of what happens at time $\nu_1$, where the red ball exceeds the current maximum. However, we point out that the identity of the red ball at time $\nu_1$ is not necessarily the same as that of the red ball at time $\nu$, because during $[\nu,\nu_1)$ the identity of red and blue balls, including the ordering within the black balls, may change, except from the identity of the black ball which corresponds to the maximum LLR. During $[\nu_1,\nu_2)$ the black balls together with the red behave as an independent group of sources that are sampled according to $\mathcal{R}$ with $\lambda=1$, whereas the blue ball is sampled at each instant $m \in  [\nu_1,\nu_2)$. In order to avoid the situation where the group of blacks with the red keep on increasing and the blue never reaches them, resulting into an increasing distance between the blue and the others, we introduce the last time $\nu_2$ which by definition requires the blue to reach the others. In Figure \ref{fig:nu2}, we show that at time $\nu_2$ the ball with smallest LLR (blue ball) reaches the others by getting between them. Also, the red ball which did the overshoot above the maximum LLR at time $\nu_1$ is no longer the one with the maximum LLR. This is because the black balls together with the red are sampled according to $\mathcal{R}$ with $\lambda=1$ and thus their ordering may change. In Subsection \ref{spl}, in Lemma \ref{tu}(i) we prove that the interval $\nu_1 -\nu$ is relatively ``short", and in Lemma \ref{vth}(iii) we prove that the interval $\nu_2 -\nu_1$ is relatively ``short" given that the initial distance of the LLRs, i.e. $H(\nu)$, is relatively ``small".

Based on this approach, we extend the definition of $\{\sigma_{l}\,:\, l\in \bN_{0}\}$ for $\lambda >2$, reassuring that no LLR is left behind. For this, we introduce the following notation. We denote by

\begin{itemize}
	\item $w_{i}(n)$ the identity of the source  in $\cD$ with the $i^{th}$ largest LLR at time $n \in \bN_{0}$, where  $i \in \{1, \ldots, |\cD|\}$, \\
	
	\item $\cD_{j}(n)$ the  subset of sources in $\cD$ with the $|\cD|-(\lambda-j)$ largest LLRs at time $n \in \bN_{0}$, i.e.,
	\begin{equation}\label{As}
		\cD_{j}(n) := \{ w_{i}(n) \,:\, i \leq |\cD|-(\lambda-j)\}, \quad \mbox{where} \quad j \in \{0,\ldots,\lambda\},
	\end{equation}
	 thus $\cD_{\lambda}(n)=\cD$, and $\cD \setminus \cD_{j}(n)$ is the subset of sources in $\cD$ with the $\lambda-j$ smallest LLRs at time $n \in \bN_{0}$, i.e.,
	\begin{equation*}
	 \cD \setminus \cD_{j}(n) = \{w_{i}(n)\, :\, i \geq |\cD|-(\lambda-j)+1 \},
	\end{equation*}
	
	\item  $H_j(n)$, for each $j \in \{0,\dots,\lambda\}$, the maximum distance of any two LLRs in $\cD_j(n)$ at time $n$, i.e.,
	\begin{equation}\label{til_ovrs}
		H_j(n):= \max_{i,k \in \cD_{j}(n)}\big{|} \Lambda_{i}(n) - \Lambda_{k}(n) \big{|}, \quad n \in \mathbb{N}_{0},
	\end{equation}	
	and thus $H_{\lambda} \equiv H$ defined in \eqref{eta0}, 
	
	\item $F_{j}: \mathcal{T} \to \mathcal{T}$, where  $j \in \{0,\ldots,\lambda-1\}$, the operator that matches each $\nu \in \cT$ to the first time after $\nu$ that at least one of the $\lambda-j$ smallest LLRs exceeds at least one of the $|D|- (\lambda-j)$ largest LLRs. Equivalently, $F_{j}(\nu)$ is the first time after $\nu$ that the LLR of at least one source from $\cD \setminus \cD_{j}(\nu)$ overshoots the LLR of at least one source from $\cD_{j}(\nu)$, i.e., 
	\begin{equation}\label{ff}
		F_{j}(\nu):= \inf \left\{ n > \nu :\, \max_{i \in \cD \setminus \cD_{j}(\nu)} \Lambda_{i}(n) > \min_{i \in \cD_j(\nu)} \Lambda_{i}(n) \right\}.
	\end{equation}
\end{itemize}
For fixed $\lambda$, the definition of $\sigma_{l+1}$, given $\sigma_{l} \in \mathcal{T}$, is the same for all $l \in \bN_0$. Let us fix  $l \geq 0$ and $\sigma_{l} \equiv \nu \in \cT$, then $\sigma_{l+1}$ is defined as the last element of an increasing sequence of $\lambda + 1$ stopping times $\nu$, $\{\nu_j \,:\, 1 \leq j \leq \lambda \}$, such that
\begin{equation}\label{sig1}
	\nu < \nu_1 < \ldots < \nu_{\lambda-1} < \nu_{\lambda},
\end{equation}
where $\sigma_{l} \equiv \nu$, and $\sigma_{l+1}:=\nu_{\lambda}$. The intermediate times are defined as follows,
\begin{equation*}
	\nu_1 := J(\nu),
\end{equation*}
where $J$ is defined in \eqref{ttl}, and for each $j \in \{1,\ldots,\lambda-1\}$,
\begin{equation*}
\nu_{j+1}:=F_{j}(\nu_j),
\end{equation*}
where $F_{j}$ is defined in \eqref{ff}. Therefore, for $\lambda \geq 1$, we define the operator
\begin{equation}\label{fl}
	\cG_{\lambda}:=
	\begin{cases}
		J,& \quad \lambda =1,\\
		F_{\lambda-1} \circ \ldots \circ F_{1} \circ J,& \quad \lambda \geq 2,
	\end{cases}
\end{equation}
and the sequence $\{\sigma_{l}\,:\, l\in \bN_{0}\}$ is defined recursively as
\begin{equation}\label{cons}
\sigma_{l+1}:= \cG_{\lambda}(\sigma_{l}), \qquad \sigma_{0}:=0.
\end{equation}
In the proof of Theorem \ref{Vl_b}, we will show that the existence claim in Theorem \ref{Vl_b} is fulfilled by $\{\sigma_{l}\,:\, l\in \bN_{0}\}$ defined according to \eqref{cons}. One of the implications of definition \eqref{cons} is that for $\lambda \geq 2$ $U(\sigma_{l},\sigma_{l+1})$ is not necessarily equal to zero. 

The following lemma is essential for the proof of Theorem \ref{Vl_b}, as it provides a condition under which $\{V_l\,:\, l \in \bN_{0} \}$ satisfies \eqref{vlts}. More precisely since $p_{\lambda} < p^{+}_{\lambda}$, and for $\lambda \geq 2$ it holds $p_{\lambda}, p^{+}_{\lambda} \in [2,\mathfrak{p}/2)$, we will show that under the aforementioned condition for any $p,\, q \in [2,\mathfrak{p}/2)$ such that $q < p$, there is a constant $C>0$ such that
\begin{equation}\label{vc1}
\sup_{l \in \bN_0} \Exp_A \left[ V^{q}_l \right] \leq C\, \left( 1 + \Exp_A \left[ H^{p}(0) \right] \right). 
\end{equation}
For this we use the supporting Lemmas \ref{bov}, \ref{vth} provided in the subsection after the proof of Theorem \ref{Vl_b}. We fix $p,\, q \in [2,\mathfrak{p}/2)$ such that $q < p$, and in view of \eqref{ab2}, in order to show \eqref{vc1}, it suffices to prove that
\begin{equation}\label{vc2}
	\begin{aligned}
		\sup_{l \in \bN_0} \Exp_A \left[ H^{p}(\sigma_l) \right],\;\; \sup_{l \in \bN_0} \Exp_A \left[ U^{q}(\sigma_l, \sigma_{l+1}) \right]\leq C\, \left( 1 + \Exp_A \left[ H^{p}(0) \right] \right).
	\end{aligned}
\end{equation}
By Lemma \ref{bov}(iii) (with $\nu \equiv \sigma_{l}$ and $w \equiv \sigma_{l+1}$), we have that there is a constant $C_1 >0$ such that
\begin{equation}\label{vc4}
	\sup_{l \in \bN_0}\Exp_{A}\left[U^{q}(\sigma_l,\sigma_{l+1})\right] \leq C_1 \, \left(1+ \sup_{l \in \bN_0}\Exp_{A}\left[(\sigma_{l+1}-\sigma_l)^{q}\right]\right),
\end{equation}
and by Lemma \ref{vth}(iv) (with $\nu \equiv\sigma_{l}$), we have that there is a constant $C_2 >0$ such that
\begin{equation}\label{vc3}
	\sup_{l \in \bN_0}\Exp_{A}\left[(\sigma_{l+1}-\sigma_l)^{q} \right] \leq C_2 \left(1 +  \sup_{l \in \bN_0}\Exp_{A}\left[H^{p}(\sigma_l)\right]\right).
\end{equation}
Therefore, in order to prove \eqref{vc1} and that $\{\sigma_{l} \,:\, l \in \bN_{0}\} \in \cT$, it suffices to show that there is a constant $C_3 >0$ such that
\begin{equation}\label{vc5}
\sup_{l \in \bN_0} \Exp_A \left[ H^{p}(\sigma_l) \right]\leq C_3\, \left( 1 + \Exp_A \left[  H^{p}(0) \right] \right).
\end{equation}

In Lemma \ref{suff}, we provide a condition under which $\{ H(\sigma_{l})\,:\, l\in \bN_{0} \}$ satisfies \eqref{vc5}. Also, we prove that if for the generic interval $[\nu,\widehat{\nu})$, where $\nu \in \cT$ and $\widehat{\nu}:=\cG_{\lambda}(\nu)$, a particular implication holds then the aforementioned condition is satisfied.
\begin{lemma}\label{suff}
	We fix $\lambda \geq 2$, $p,\, q \in [2,\mathfrak{p}/2)$ such that $q < p$, and the stopping time $\nu \in \cT$. We assume that $H(0) \in \mathcal{L}^{p}$.
	\begin{enumerate}[(i)]
		\item  If there is a constant $C>0$ such that 
		\begin{equation}\label{hl}
			\Exp_{A}\left[H^{p}(\sigma_{l+1})\right] \leq C \left( 1 + \left(\Exp_{A}\left[(\sigma_{l+1}-\sigma_{l})^{q}\right]\right)^{2/3} \right), \quad \forall\; l \in \bN_{0},
		\end{equation} 
		then $\{\sigma_{l}\,:\, l \in \bN_{0} \} \in \cT$, $\{ H_l\,:\, l \in \bN_{0} \}$ satisfies \eqref{vc5}, and as a result $\{ V_l\,:\, l \in \bN_{0} \}$ satisfies \eqref{vc1}.\\

		\item If there is a constant $C>0$ such that the following implication holds
		\begin{equation}\label{impl}
			\begin{aligned}
				\left(\widehat{\nu}-\nu \right) \in \mathcal{L}^{q} \quad \Rightarrow  \quad  \Exp_{A}\left[H^{p}(\widehat{\nu})\right] \leq C \left( 1 + \left(\Exp_{A}\left[(\widehat{\nu}-\nu)^{q}\right]\right)^{2/3} \right),
			\end{aligned}
		\end{equation} 
		then \eqref{hl} holds.
	\end{enumerate}
\end{lemma} 

\begin{IEEEproof}
	(i) As explained in \eqref{vc1}-\eqref{vc5}, in order to show that $\{\sigma_{l}\,:\, l \in \bN_{0} \} \in \cT$ and $\{  V_l\,:\, l \in \bN_{0} \}$ satisfies \eqref{vc1}, it suffices to show \eqref{vc5}. By assumption of \eqref{hl}, and Lemma \ref{vth}(iv) (with $\nu \equiv\sigma_{l}$) it follows that there is a constant $D >0$ such that $\Exp_{A}\left[(\sigma_{l+1}-\sigma_{l})^{q}\right] \leq D(1+\Exp_{A}\left[H^{p}(\sigma_{l})\right])$, and thus a constant $Q> 0$ such that
	\begin{equation}\label{ttn1}
		\begin{aligned}
			\Exp_{A}\left[H^{p}(\sigma_{l+1})\right] \leq  C + C\left( D + D \, \Exp_{A}\left[H^{p}(\sigma_{l})\right]\right)^{2/3} 
			\leq  Q + Q \left(\Exp_{A}\left[H^{p}(\sigma_{l})\right] \right)^{2/3}, \qquad \forall \, l \in \bN_0,
		\end{aligned}
	\end{equation}
	where for the second inequality we used the fact that $(x+y)^{2/3} \leq x^{2/3} + y^{2/3}$ for any $x,y \geq 0$. The sequence  
	\begin{equation*}
		a_{l+1} := Q + Q \left(a_{l}\right)^{2/3}, \quad l \in \bN_{0}, \quad  a_{0} := \Exp_{A}\left[H^{p}(0)\right],
	\end{equation*}
	converges to a limit, denoted by $L$, which is the unique root of the equation $L= Q +Q L^{2/3}$. If $\Exp_{A}\left[H^{p}(0)\right] > L$, then the sequence is decreasing and upper bounded by $\Exp_{A}\left[H^{p}(0)\right]$, otherwise it is non-decreasing and upper bounded by $L$. In both cases,
	\begin{equation*}
		\sup_{l \in \bN_0} \Exp_A \left[ H^{p}(\sigma_l) \right]\leq L + \Exp_A \left[  H^{p}(0) \right],
	\end{equation*}
	which proves \eqref{vc5}, for $C_3 := \max\{1,L\}$.\\
	
	(ii) We will show \eqref{hl} by induction on $l \in \bN_0$.
	
	\textit{Basis of induction}: For $l=0$, since $\sigma_0=0$ and $H(0) \in \mathcal{L}^{p}$, by application of Lemma \ref{vth}(iv) (with $\nu\equiv\sigma_0$) we have $\sigma_1 - \sigma_0 \in \mathcal{L}^{q}$, and by assumption of \eqref{impl} (with $\nu\equiv\sigma_0$) we obtain \eqref{hl} for $l=0$.
	
	\textit{Step of induction}: We fix $i \in \bN$. We assume that \eqref{hl} holds for each $l \leq i-1$, and we prove that it also holds for $l=i$. By assumption of \eqref{hl} for each $l \leq i-1$, it follows that \eqref{ttn1} holds for each $l \leq i-1$, and since $H(0) \in \mathcal{L}^{p}$ we deduce that 
	\begin{equation}\label{str1}
		\begin{aligned}
			H(\sigma_{l+1}) \in \mathcal{L}^{p}, \;\; \forall \; l \in \{-1,\ldots,i-1\}\quad \Leftrightarrow \quad H(\sigma_{l}) \in \mathcal{L}^{p}, \;\; \forall \; l \in \{0,\ldots,i\}.
		\end{aligned}
	\end{equation}
	Thus, by application of Lemma \ref{vth}(iv) (with $\nu\equiv\sigma_l$) for each $l \in \{0,\ldots,i\}$, it follows that
	\begin{equation}
		\{\sigma_{l+1}-\sigma_l\} \in \mathcal{L}^{q}, \quad \forall \; l \in \{0,\ldots,i\},
	\end{equation}
	which implies that $\sigma_l \in \cT$ for all $l \in \{0,\ldots,i+1\}$. Therefore, $\sigma_i \in \cT$ and $(\sigma_{i+1}-\sigma_i) \in \mathcal{L}^{q}$, which by assumption of \eqref{impl} (with $\nu\equiv\sigma_i$) implies that \eqref{hl} holds for $l=i$.\\
\end{IEEEproof}

For the proof of Theorem \ref{Vl_b} we will apply induction. The following lemma is essential for the establishment of the proof by induction of Theorem \ref{Vl_b}. We recall the definition of the sequence
\begin{equation*}
\nu < \nu_1 < \ldots < \nu_{\lambda-1} < \nu_{\lambda},
\end{equation*}
in \eqref{sig1}, and we fix $j \in \{1,\ldots,\lambda-1\}$. In Lemma \ref{cs_inq}, we provide a bound on a moment of $H_{j+1}(\nu_{j+1})$ given that the respective bound holds also for $H_{j}(\nu_{j})$. For the proof of the result we restrict the rule $\cR$ defined in \eqref{newR} on the set $\cD_{j}(\nu_{j})$ defining in this way the rule
\begin{equation}\label{crj}
	\cR_{j}:=\cR(j, \mathfrak{q}, \mathcal{D}_{j}(\nu_{j}), \widetilde{W}, \widetilde{Y}, \widetilde{Z}),
\end{equation}
where 
\begin{equation}
	\begin{aligned}
		\widetilde{W}_{i}&:=\Lambda^{\cR}_{i}(\nu_{j}), \quad &&\forall \; i \in \mathcal{D}_{j}(\nu_{j}), \\
		\widetilde{Y}_{i}(n)&:=Y_{i}(n+\nu_{j}), \quad &&\forall\; n \in \bN,  \quad \forall \; i \in \mathcal{D}_{j}(\nu_{j}),\\
		\widetilde{Z}_n&:= Z_{n+\nu_{j}}, \quad &&\forall\; n \in \bN,
	\end{aligned}
\end{equation}
and our basic assumption is that Theorem \ref{Vl_b} holds for the sources in $\mathcal{D}_{j}(\nu_{j})$ when sampled according to rule $\cR_{j}$. In particular, we assume that the sequence $\{ \sigma_{l}\, :\, l \in \bN_{0}\}$ is defined recursively, in this case by the operator $\cG_{j}$ in \eqref{fl}, with $\sigma_{0}:=0$, and that for the sequence
\begin{equation}\label{Vj}
	\left\{ V_{l}:=\max\limits_{\sigma_{l}\leq n< \sigma_{l+1}}  \max\limits_{i,z\in \cD_{j}(\nu_{j}) }\left| \Lambda^{\cR_{j}}_{i}(n)-\Lambda^{\cR_j}_{z}(n)\right| \, : \, l \in \bN_{0} \right\}
\end{equation}
there is a constant $C > 0$ independent of $\cF^{\cR}_{\nu_j}$ such that
\begin{equation*}
	\sup_{l \in \bN_0} \Exp_A \left[ V^{p_j}_l \right] \leq C \left( 1 + \Exp_A \left[H_{j}^{p^{+}_j}(\nu_j)\right] \right).
\end{equation*}
Each $V_l$ in \eqref{Vj} stands for the maximum distance of the LLRs of the sources in $\cD_{j}(\nu_{j})$ during $[\sigma_{l},\sigma_{l+1})$ when sampled according to $\cR_j$. We restore the superscript that indicates the sampling rule, in order to distinguish between $\cR$ and $\cR_j$.

\begin{lemma}\label{cs_inq}
	We fix $\lambda \geq 2$, $j \in \{1,\ldots,\lambda-1 \}$, and $\nu \in \cT$. We also assume that the sequence $\{\sigma_{l}\, :\, l \in \bN_{0}\}$ is defined recursively by the operator $\cG_{j}$ with $\sigma_{0}:=0$.
	
	\begin{enumerate}[(i)]
		\item The sampling rule $\cR_{j}:=\cR(j, \mathfrak{q}, \mathcal{D}_{j}(\nu_j), \widetilde{W}, \widetilde{Y}, \widetilde{Z})$, defined in \eqref{crj}, is well-defined.
		
		\item The restriction of $\cR$ on the set $\cD_{j}(\nu_{j})$ is equal to the sampling rule $\cR_j$ during the interval $[\nu_j,\nu_{j+1})$, i.e.,
		\begin{equation}\label{req}
			\cR_{j}(m) = \cR(m+\nu_j) \bigcap \mathcal{D}_{j}(\nu_j),\qquad \forall \, m\, \in [0, \nu_{j+1}-\nu_{j}).
		\end{equation}
		
		\item If for any $q >3$ there is a constant $C_0 >0$ that depends only on $q$ such that
		\begin{equation}\label{hbr}
			\Exp_{A} \left[H_{j}^{p^{+}_j}(\nu_{j})\right] \leq C_{0} \left( 1+ \left(\Exp_{A}\left[(\nu_{j}-\nu)^q \right]\right)^{2/3} \right),
		\end{equation}
		and for the sequence $\{ V_{l} \,:\, l \in \bN_{0} \}$ defined in \eqref{Vj}, there is a constant $C_1 >0$ such that
		\begin{equation}\label{hbrs}
			\sup_{l \in \bN_0} \Exp_A \left[ V^{p_j}_l \right] \leq C_1 \left(1 + \Exp_A \left[H_{j}^{p^{+}_j}(\nu_j)\right] \right),
		\end{equation}
		then for any $q >3$ there is a constant $C_2 >0$ that depends only on $q$ such that
		\begin{equation}\label{hbr+}
			\Exp_{A} \left[H_{j+1}^{p^{+}_{j+1}}(\nu_{j+1})\right] \leq C_{2} \left( 1+ \left(\Exp_{A}\left[(\nu_{j+1}-\nu)^q \right]\right)^{2/3} \right).
		\end{equation}
		In particular, for $j=\lambda-1$, $q=p_{\lambda}$, by \eqref{hbr+} we deduce that there is a constant $C>0$ such that
		\begin{equation*}
			\Exp_{A} \left[H^{p^{+}_{\lambda}}(\nu_{\lambda})\right] \leq C \left( 1+ \left(\Exp_{A}\left[(\nu_{\lambda}-\nu)^{p_{\lambda}} \right]\right)^{2/3} \right).
		\end{equation*}
	\end{enumerate}
\end{lemma}

\begin{IEEEproof}
	(i) Since $\nu_j \in \cT^{\cR}$, we have $\cD_{j}(\nu_{j}) \in \cF^{\cR}_{\nu_{j}}$ and $\widetilde{W}_{i} \in \cF^{\cR}_{\nu_{j}}$ for all $i \in \cD_{j}(\nu_{j})$. In accordance with the definition of a sampling rule in \eqref{newR}, in order to prove that $\cR_j$ is well-defined, we must show that $\{\widetilde{Y}_{i}(n)\, :\, n \in \bN\}$ is iid for all $i \in \mathcal{D}_{j}(\nu_j)$, $\{\widetilde{Z}_n \, : \, n \in \bN \}$ is idd, and also that
	\begin{equation}\label{inqq}
		j-1 + \mathfrak{q} < \sum_{i \in \mathcal{D}_{j}(\nu_j)} \frac{I^*(\cD_{j}(\nu_j))}{I_i}.
	\end{equation}
	Indeed, since $\nu_j \in \cT^{\cR}$, by \cite[Theorem 4.1.3]{durrett2010probability} it follows that for all $i \in \mathcal{D}_{j}(\nu_j)$ $\{\widetilde{Y}_{i}(n)\, :\, n \in \bN\}$ is iid, independent of $\cF^{\cR}_{\nu_{j}}$, and it has the same distribution as $\{Y_{i}(n)\, :\, n \in \bN\}$, and also $\{\widetilde{Z}_n \, : \, n \in \bN \}$ is idd, independent of $\cF^{\cR}_{\nu_{j}}$, and it has the same distribution as $\{Z_n \, : \, n \in \bN \}$. Since $I^*(\cD_{j}(\nu_j)) \geq I^*(\cD)$, the inequality \eqref{inqq} follows by Lemma \ref{vth}(i).\\
	
	(ii) In order to prove \eqref{req}, we recall that the random set $\cD_{j}(\nu_j)$ contains the sources with the $|\cD|-(\lambda - j)$ largest LLRs in $\cD$ at time $\nu_j$, and the random set $\cD\setminus\cD_{j}(\nu_j)$ contains the remaining $(\lambda - j)$ sources. Thus, during $[\nu_j, \nu_{j+1})$ the sampling rule $\cR$ samples the $(\lambda - j)$ sources in $\cD\setminus\cD_{j}(\nu_j)$, and the 
	\begin{equation*}
		j-1 + \mathbf{1}\{ Z_{n+\nu_j} \leq \mathfrak{q} \}
	\end{equation*}
	sources with the smallest LLRs in $\cD_{j}(\nu_j)$. On the other hand, the sampling rule $\cR_j$ samples the
	\begin{equation*}
		j-1 + \mathbf{1}\{ \widetilde{Z}_n \leq \mathfrak{q} \}
	\end{equation*}
	sources with the smallest LLRs in $\cD_{j}(\nu_j)$. Therefore, during $[\nu_j, \nu_{j+1})$ the sampling rule $\cR_j$ samples the same sources from $\cD_{j}(\nu_j)$ as the rule $\cR$, which proves \eqref{req}.\\
	
	(iii) By definition of $F_j(\nu_j)$ in \eqref{ff}, at time $\nu_{j+1}$ a LLR from $\cD \setminus \cD_{j}(\nu_j)$ overshoots one of the LLRs from $\mathcal{D}_{j}(\nu_j)$. This implies that the maximum distance of the sources in $\mathcal{D}_{j+1}(\nu_{j+1})$ at time $\nu_{j+1}$, i.e., $H_{j+1}(\nu_{j+1})$, is bounded by $V_{u^*}$ where $u^*$ is the index of the interval $[\sigma_{l},\sigma_{l+1})$ when the aforementioned overshoot occurs, i.e., $\nu_{j+1} \in [\sigma_{u^*},\sigma_{u^*+1})$, and
	\begin{equation}
		u^{*}:= \sup\{ u \in \bN_{0} \,:\, \nu_{j+1} \geq \nu_j + \sigma_{u}\}.
	\end{equation}
	In case at time $\nu_{j+1}$ there is an overshoot above the current maximum LLR, the upper bound of $H_{j+1}(\nu_{j+1})$ is also augmented by the size of this overshoot. We denote by $\eta(\nu_{j},\nu_{j+1})$ the number of times the maximum LLR changes during $[\nu_{j},\nu_{j+1}]$, and thus the aforementioned overshoot is the $\eta(\nu_{j},\nu_{j+1})^{th}$ overshoot above the maximum LLR during $[\nu_{j},\nu_{j+1}]$. Therefore, we upper bound $H^{p^{+}_{j+1}}_{j+1}(\nu_{j+1})$ by
	\begin{equation}
		H^{p^{+}_{j+1}}_{j+1}(\nu_{j+1}) \leq V^{p^{+}_{j+1}}_{u^*} + \sum_{i=1}^{\eta(\nu_{j},\nu_{j+1})} L^{p^{+}_{j+1}}(s_i),
	\end{equation}
	We fix $q >3$. By Jensen's inequality, in order to prove \eqref{hbr+}, it suffices to show that there are constants $D_1, D_2 >0$, such that
	\begin{align}
		\Exp_{A}\left[ \sum_{i=1}^{\eta(\nu_{j},\nu_{j+1})} L^{p^{+}_{j+1}}(s_i) \right] &\leq D_{1} \left( 1+ \left(\Exp_{A}\left[(\nu_{j+1}-\nu)^q \right]\right)^{2/3} \right),\label{al1} \\
		\Exp_{A}\left[V^{p^{+}_{j+1}}_{u^*}\right] &\leq D_{2} \left(\Exp_{A}\left[(\nu_{j+1}-\nu)^q \right]\right)^{2/3}, \label{al2}
	\end{align}
	For \eqref{al1}, by Lemma \ref{bov}(ii) it follows that there is a constant $D_{1} >0$ such that
	\begin{equation}
		\begin{aligned}
			\Exp_{A}\left[\sum_{i=1}^{\eta(\nu_{j},\nu_{j+1})} L^{p^{+}_{j+1}}(s_i) \right] \leq D_{1} \left( 1 + \Exp_{A}\left[\nu_{j+1}-\nu_{j}\right] \right) \leq D_{1} \left( 1 +\left(\Exp_{A}\left[(\nu_{j+1}-\nu)^q\right]\right)^{2/3} \right),
		\end{aligned}
	\end{equation}
	where the last inequality follows by Holder's inequality and the fact that $\nu_{j+1}-\nu \geq \nu_{j+1}-\nu_j \geq 1$, and $1/q \leq 2/3$. 
	
	For \eqref{al2}, and since $p^{+}_{j+1}=p_{j}/2$ we have
	\begin{equation}\label{b1}
		\begin{aligned}
			\Exp_{A}\left[V^{p^{+}_{j+1}}_{u^*}\right] \leq \Exp_{A}\left[\sum_{l=0}^{\infty} V^{p_j/2}_{l} \, \mathbf{1}\{ u^* \geq l \}\right] \leq \sum_{l=0}^{\infty} \Exp_{A}\left[V^{p_j/2}_{l} \, \mathbf{1}\{ u^* \geq l \}\right],
		\end{aligned}
	\end{equation}
	where the last inequality is deduced by the monotone convergence theorem. For each $l \geq 0$, by the Cauchy-Schwarz inequality we have
	\begin{equation}\label{b2}
		\Exp_{A}\left[V^{p_j/2}_{l} \, \mathbf{1}\{ u^* \geq l \}\right] \leq \sqrt{\Exp_{A}[V^{p_j}_{l}]}\sqrt{\Pro_{A}(u^* \geq l)} \leq \sqrt{\Exp_{A}[V^{p_j}_{l}]} \left(\Pro_{A}(u^* \geq l)\right)^{1/3},
	\end{equation}
	where the last inequality follows by the fact that $\Pro_{A}(u^* \geq l) \leq 1$, and $1/3 \leq 1/2$. By \eqref{hbr} and \eqref{hbrs}, it follows that there is a constant $D_3 > 0$ such that
	\begin{equation}\label{b3}
		\sup_{l \in \bN_{0}} \sqrt{\Exp_{A}\left[V^{p_j}_{l} \right]} \leq D_3 \left( 1+\left(\Exp_{A}\left[(\nu_{j+1}-\nu)^q\right]\right)^{1/3} \right),
	\end{equation}
	where we used the fact that $\sqrt{x+y} \leq \sqrt{x} + \sqrt{y}$ for any $x,y >0$, and that $\nu_{j+1}-\nu \geq \nu_{j}-\nu$. Therefore,
	\begin{equation}
		\Exp_{A}\left[V^{p^{+}_{j+1}}_{u^*}\right] \leq  D_3 \left( 1+\left(\Exp_{A}\left[(\nu_{j+1}-\nu)^q\right]\right)^{1/3} \right)   \sum_{l=0}^{\infty} \left(\Pro_{A}(u^* \geq l)\right)^{1/3}.
	\end{equation}
	By definition $u^* \leq \nu_{j+1} - \nu_{j} \leq \nu_{j+1} - \nu$, thus by Markov's inequality we have
	\begin{equation}
		\Pro_{A}\left( u^* \geq l\right) \leq \frac{\Exp_{A}\left[(\nu_{j+1}-\nu)^q\right]}{l^q}, \quad l \in \bN.
	\end{equation}
	Since $\nu_{j+1} - \nu \geq 1$, for $l=0$ we also have
	\begin{equation}
		\Pro_{A}( u^* \geq 0) = 1 \leq \Exp_{A}\left[(\nu_{j+1}-\nu)^q\right].
	\end{equation}
	As a result,
	\begin{equation*}
		\sum_{l=0}^{\infty} \left(\Pro_{A}(u^* \geq l)\right)^{1/3} \leq \left( 1+ \sum_{l=1}^{\infty} \frac{1}{l^{q/3}}\right) \left(\Exp_{A}\left[(\nu_{j+1}-\nu)^q\right]\right)^{1/3}.
	\end{equation*}	
	Since $q>3$,
	\begin{equation*}
		D_4 := 1 +  \sum_{l=1}^{\infty} \frac{1}{l^{q/3}} < \infty,
	\end{equation*}
	and thus
	\begin{equation*}
		\begin{aligned}
			\Exp_{A}\left[V^{p^{+}_{j+1}}_{u^*}\right] &\leq D_3 D_4 \left(\Exp_{A}\left[(\nu_{j+1}-\nu)^q\right]\right)^{2/3} + D_3 D_4 \left(\Exp_{A}\left[(\nu_{j+1}-\nu)^q\right]\right)^{1/3}\\
			&\leq 2D_3 D_4 \left(\Exp_{A}\left[(\nu_{j+1}-\nu)^q\right]\right)^{2/3},
		\end{aligned}
	\end{equation*}
	where the last inequality follows by the fact that $x^{1/3} \leq x^{2/3}$, for any $x \geq 1$, and the claim \eqref{al2} follows with $D_2:=2D_3 D_4 $.
\end{IEEEproof}

We proceed to the proof of Theorem \ref{Vl_b} for $\lambda > 1$.

\begin{IEEEproof}[Proof of Theorem \ref{Vl_b} for $\lambda > 1$]
 We fix $\lambda_0 \in \bN$, a sampling rule $\cR$ defined in \eqref{newR} with $\lambda=\lambda_0$, and we assume that \eqref{s_a} holds. According to Lemma \ref{suff}(ii) in order to prove the claim it suffices to show that for the generic interval $[\nu,\hat{\nu})$, where $\nu \in \cT$ and $\hat{\nu} := \cG_{\lambda_0}(\nu)$, the following implication holds
 \begin{equation}\label{impl1}
 	\begin{aligned}
 		\left(\widehat{\nu}-\nu \right) \in \mathcal{L}^{p_{\lambda_0}} \quad \Rightarrow  \quad  \Exp_{A}\left[H^{p^{+}_{\lambda_0}}(\widehat{\nu})\right] \leq C \left( 1 + \left(\Exp_{A}\left[(\widehat{\nu}-\nu)^{p_{\lambda_0}}\right]\right)^{2/3} \right).
 	\end{aligned}
 \end{equation}
Recalling the definition \eqref{sig1} of the sequence
\begin{equation*}
	\nu < \nu_1 < \ldots < \nu_{\lambda-1} < \hat{\nu} := \nu_{\lambda_0},
\end{equation*}
in order to prove \eqref{impl1} it suffices to show that for each $j \in \{1,\ldots,\lambda_0\}$ there is a constant $C_j$ such that
\begin{equation}\label{hjp}
	\Exp_{A} \left[H_{j}^{p^{+}_j}(\nu_{j})\right] \leq C_{j} \left( 1+ \left(\Exp_{A}\left[(\nu_{j}-\nu)^{p_j} \right]\right)^{2/3} \right),
\end{equation}
and for $j=\lambda_0$ we prove \eqref{impl1}. In particular, we will prove the stronger claim that for each $j \in \{1,\ldots,\lambda_0\}$ and any $q>3$ there is a constant $C_j >0$ that depends only on $q$ such that
\begin{equation*}
	\Exp_{A} \left[H_{j}^{p^{+}_j}(\nu_{j})\right] \leq C_{j} \left( 1+ \left(\Exp_{A}\left[(\nu_{j}-\nu)^{q} \right]\right)^{2/3} \right).
\end{equation*}
For this, we proceed by induction on $j$.

\textit{Basis of the induction}: For $j=1$, since $\nu_{1}:=J(\nu)$ by Lemma \ref{tu}(iv) with $p = p^{+}_{1}$ it follows that there is a constant $C_1>0$ such that
\begin{equation}\label{ast1}
\Exp_A \left[ H^{p^{+}_1}_1(\nu_{1}) \right] \leq C_1 \leq C_1 \left( 1+ \left(\Exp_{A}\left[(\nu_{1}-\nu)^{q} \right]\right)^{2/3} \right),
\end{equation}
where the second inequality holds for any $q>3$. We point out that $C_1$ is independent of $\cF_{\nu}$, meaning that for any interval $[\sigma_{l}, \sigma_{l+1})$ we will get the same constants $\{ C_j \, :\, 1 \leq j \leq \lambda_0\}$ independently of $l$. Also, for $\lambda =1$ we have already proven Theorem \ref{Vl_b} in the ``Proof of Theorem \ref{Vl_b} for $\lambda=1$". This implies that for the sequence $\{ V_{l} \,:\, l \in \bN_{0} \}$ defined in \eqref{Vj} for $j=1$, there is a constant $D_1 >0$ such that
\begin{equation}\label{ast2}
	\sup_{l \in \bN_0} \Exp_A \left[ V^{p_1}_l \right] \leq D_1 \left(1 + \Exp_A \left[H_{1}^{p^{+}_1}(\nu_1)\right] \right).
\end{equation}
The basis of the induction is fulfilled by \eqref{ast1} and \eqref{ast2}. Since the step of the induction is complicated we first show the case $j=2$ (the first step), and then an arbitrary step of the induction.

\textit{The case $j=2$ (the first step)}: In view of \eqref{ast1} and \eqref{ast2}, by Lemma \ref{cs_inq}(iii) it follows that for any $q>3$ there is a constant $C_2 > 0$ such that
\begin{equation*}
	\Exp_A \left[ H^{p^{+}_2}_2(\nu_{2}) \right] \leq C_2 \left( 1+ \left(\Exp_{A}\left[(\nu_{2}-\nu)^{q} \right]\right)^{2/3} \right).
\end{equation*}
For $q=p_2$, we have \eqref{hjp}, and thus for the case $\lambda=2$, we proved the implication \eqref{impl1}. Hence, by Lemma \ref{suff} it follows that Theorem \ref{Vl_b} holds for $\lambda=2$, which implies that for the sequence  $\{ V_{l} \,:\, l \in \bN_{0} \}$ defined in \eqref{Vj} for $j=2$, there is a constant $D_2 >0$ such that
\begin{equation*}
	\sup_{l \in \bN_0} \Exp_A \left[ V^{p_2}_l \right] \leq D_2 \left(1 + \Exp_A \left[H_{2}^{p^{+}_2}(\nu_2)\right] \right).
\end{equation*}

\textit{Step of the induction}: We fix $i \in \{2,\ldots,\lambda_0 -1\}$, and we assume that for any $q >3$ there is a constant $C_i > 0$ such that
\begin{equation*}
\Exp_{A} \left[H_{i}^{p^{+}_i}(\nu_{i})\right] \leq C_{i} \left( 1+ \left(\Exp_{A}\left[(\nu_{i}-\nu)^{q} \right]\right)^{2/3} \right),
\end{equation*}
and for the sequence  $\{ V_{l} \,:\, l \in \bN_{0} \}$ defined in \eqref{Vj} for $j=i$, there is a constant $D_i >0$ such that
\begin{equation*}
	\sup_{l \in \bN_0} \Exp_A \left[ V^{p_i}_l \right] \leq D_i \left(1 + \Exp_A \left[H_{i}^{p^{+}_i}(\nu_i)\right] \right).
\end{equation*}
Then, by Lemma \ref{cs_inq}(iii) it follows that for any $q>3$ there is a constant $C_{i+1} > 0$ such that
\begin{equation*}
\Exp_{A} \left[H_{i+1}^{p^{+}_{i+1}}(\nu_{i+1})\right] \leq C_{i+1} \left( 1+ \left(\Exp_{A}\left[(\nu_{i+1}-\nu)^{q} \right]\right)^{2/3} \right).
\end{equation*}
For $q=p_{i+1}$, we have \eqref{hjp}, and thus for the case $\lambda=i+1$, we proved the implication \eqref{impl1}. Hence, by Lemma \ref{suff} it follows that Theorem \ref{Vl_b} holds for $\lambda=i+1$, which implies that for the sequence  $\{ V_{l} \,:\, l \in \bN_{0} \}$ defined in \eqref{Vj} for $j=i+1$, there is a constant $D_{i+1} >0$ such that
\begin{equation*}
\sup_{l \in \bN_0} \Exp_A \left[ V^{p_{i+1}}_l \right] \leq D_{i+1} \left(1 + \Exp_A \left[H_{i+1}^{p^{+}_{i+1}}(\nu_{i+1})\right] \right).
\end{equation*}

For $i=\lambda_0 -1$, we do the last step of the induction and we complete the proof of Theorem \ref{Vl_b} for rule $\cR$ with $\lambda=\lambda_0$.
\end{IEEEproof}

\subsection{Supporting Lemmas}\label{spl}
In this subsection, we state and prove lemmas that are used in the proof of Theorem \ref{Vl_b}. We recall that $\cR$ samples at most $\lambda$ sources that correspond to the $\lambda$ smallest LLRs, and all the remaining $|\cD|-\lambda$ with higher LLRs remain constant until a change occurs when one of the $\lambda$ smallest LLRs exceeds one of the $|\cD|-\lambda$ largest LLRs, and sometimes the overshoot exceeds even the current maximum. In Lemma \ref{tu}, we obtain upper bounds for moments of the size of intervals defined by the times when such changes occur, and of the size of the respective overshoots. For this, we need to introduce the following notation, where $\nu$ is a stopping time in $\cT$. \textit{In all supporting lemmas, when we refer to a constant we also imply independent of $l$, $\cF^{\cR}_0$, and of any stopping time in $\cT$.}

We recall the definition of the operator $J$ defined in \eqref{ttl}, and we denote by $\widetilde{\nu}$ the first time after $\nu$ that a change in the maximum LLR occurs, i.e., 
	\begin{equation}\label{wnu}
		\widetilde{\nu}:=J(\nu).
	\end{equation}
We denote by $L(\widetilde{\nu})$ the size of the overshoot at time $\widetilde{\nu}$, above the current maximum LLR we have since time $\nu$, i.e.,
	\begin{equation}\label{llw}
		L(\widetilde{\nu}):= \max_{i \in \cD} \Lambda_{i}(\widetilde{\nu}) - \max_{i \in \cD} \Lambda_{i}(\nu).
	\end{equation}

We denote by $\nu_0$ is the first time after $\nu$ when the current $(|\mathcal{D}|-\lambda)^{th}$ largest LLR changes, or equivalently the LLR of a source in $\cD \setminus \cD_0(\nu)$ exceeds the LLR of at least one source in $\cD_0(\nu)$, where $\cD_0(\nu)$ is the set that contains the sources with the $(|\mathcal{D}|-\lambda)$ largest LLRs at time $\nu$, i.e.,
\begin{equation*}
\nu_0 := F_{0}(\nu),
\end{equation*}
where
\begin{equation*}
F_{0}(\nu):= \inf \left\{ n > \nu :\, \max_{i \in \cD \setminus \cD_{0}(\nu)} \Lambda_{i}(n) > \min_{i \in \cD_0(\nu)} \Lambda_{i}(n) \right\}.
\end{equation*}

At time $\widetilde{\nu}$  a change in both the current $(|\mathcal{D}|-\lambda)^{th}$ largest LLR and the current maximum LLR occurs, but before that some changes in the current $(|\mathcal{D}|-\lambda)^{th}$ largest LLR might have already occured, and as a result
\begin{equation*}
\nu_0 \leq \widetilde{\nu}.
\end{equation*}
Also, we denote by
	\begin{equation} \label{def_S}
		S(\nu_0) := \max_{j \in \cD \setminus \cD_{0}(\nu)} \Lambda_{j}(\nu_0) - \min_{j \in \cD_0(\nu)} \Lambda_{j}(\nu)
	\end{equation}	
	the size of the overshoot at time $\nu_0$, above the $(|\mathcal{D}|-\lambda)^{th}$ largest LLR we have since time $\nu$. 
	
Starting from $\nu$, we denote by $\{\tau_i \, :\, i \in \bN_0\}$ the sequence of stopping times, when the current $(|\mathcal{D}|-\lambda)^{th}$ largest LLR changes, i.e.,
	 \begin{equation}\label{sf0}
	 	\tau_{i+1}:=F_{0}(\tau_{i}), \quad i \in \bN_{0}, \quad \mbox{where }\; \tau_0:=\nu \in \cT,
	 \end{equation}
and by $\{S(\tau_{i}) \, :\, i \in \bN_{0} \}$ the sequence of sizes of the respective overshoots defined according to \eqref{def_S}.
	  
By the definition of $\cR$, all sources in $\cD \setminus \cD_{0}(\nu)$ are sampled at every time $m \in [\nu,\nu_0]$ either with probability $1$ or with probability $\mathfrak{q}$. In order to point out the identity of the source whose LLR exceeds the boundary $\min_{j \in \cD_0(\nu)} \Lambda_{j}(\nu)$, for each $i \in \cD \setminus \cD_{0}(\nu)$ we denote by $\nu_{0,i}$ the first time after $\nu$ that the LLR of the source $i$ exceeds $\min_{j \in \cD_0(\nu)} \Lambda_{j}(\nu)$, i.e.,
\begin{align*}
	\nu_{0,i}:= \inf\left\{ n> \nu :\, \Lambda_{i}(\nu)+ \sum_{m=\nu+1}^n \log  \left( \frac{f_{1i}(Y_{i}(m))}
	{f_{0i}(Y_{i}(m)) } \right)\cR_i(m)  > \min_{j \in \cD_0(\nu)} \Lambda_{j}(\nu) \right\},
\end{align*}
where $\cR_i(m)$ is the indicator of whether source $i$ is sampled at time $m$, and by $S_{i}(\nu_{0,i})$  the size of the overshoot above $\min_{j \in \cD_0(\nu)} \Lambda_{j}(\nu)$ at time $\nu_{0,i}$ by the LLR of source $i$, i.e., 
\begin{equation*}
	S_{i}(\nu_{0,i}):=  \Lambda_{i}(\nu)+ \sum_{m=\nu+1}^{\nu_{0,i}} \log  \left( \frac{f_{1i}(Y_{i}(m)) }
	{f_{0i}(Y_{i}(m)) } \right) \cR_i(m) - \min_{j \in \cD_0(\nu)} \Lambda_{j}(\nu).
\end{equation*}
Clearly, 
\begin{equation}\label{y_max}
	S(\nu_0) \leq \max_{i \in \cD \setminus \cD_{0}(\nu)} S_{i}(\nu_{0,i}).
\end{equation}

We proceed to the statement of Lemma \ref{tu}. The proof of Theorem \ref{Vl_b} for $\lambda =1$ is based entirely on Lemma \ref{tu}.\\

\begin{lemma}\label{tu}	
    We fix $\nu \in \cT$.	
	\begin{enumerate}[(i)]
		\item For any $p \in [1, \mathfrak{p}/2)$, there is a constant $C>0$ such that
		\begin{equation}\label{n0n} 
			\Exp_{A}\left[ (\widetilde{\nu}-\nu)^{p} \, \big{|} \, \cF_{\nu} \right] \leq  C.
		\end{equation}
		Also, $\nu_0, \widetilde{\nu} \in \cT$, and for $\tau_0 := \nu$ the sequence $\{\tau_i \, :\, i \in \bN_0\} \in \cT$.
		
		\item The number of elements of the sequence $\{\tau_i \, :\, i \in \bN_0\}$, with $\tau_0 := \nu$, that lie within $[\nu, \widetilde{\nu}]$, i.e.,
		\begin{equation}\label{zz}
			z := \max\{ i \in \bN_0 \, :\, \tau_{i} \leq \widetilde{\nu}\}
		\end{equation}
		is a $\Pro_{A}$-a.s. finite stopping time with respect to filtration $\{\cF_{\tau_i} \, :\, i \in \bN_0\}$.
		
		\item For any $p \in [1, \mathfrak{p}-1 ]$, there is a constant $C>0$ such that
		\begin{equation*}
			\Exp_{A} \left[S^{p}(\nu_0) \, \big{|} \, \cF_{\nu} \right] \leq C.
		\end{equation*}
		
		\item For any $p \in [1, \mathfrak{p}-1 ]$, there is a constant $C>0$ such that
		 \begin{equation}\label{lhy}
		 	\Exp_{A}\left[ L^p(\widetilde{\nu})  \, \big{|} \, \cF_{\nu} \right] \, \leq \, \Exp_{A}\left[ H_{1}^p(\widetilde{\nu})  \, \big{|} \, \cF_{\nu} \right] \leq C,
		 \end{equation}    
		where $H_1$ is defined in \eqref{til_ovrs}, and for the case $\lambda =1$ we have $H_{1} \equiv H$.
	\end{enumerate}	
\end{lemma}

\begin{IEEEproof}
	(i) We fix $p \in [1, \mathfrak{p}/2)$. In order to prove \eqref{n0n}, it suffices to show that 
	\begin{equation}\label{ndx}
		\sum_{m=1}^{\infty}m^{p-1} \Pro_{A}\left(\widetilde{\nu} - \nu > m \, \big{|} \, \cF_{\nu}\right) \leq C,
	\end{equation}
	where $C>0$ is a constant. For each $m \in \bN_{0}$, the event $\{ \widetilde{\nu}-\nu > m  \}$ implies that during $[\nu,\, \nu + m]$ there has not occurred a change in the maximum LLR. As a result, there exists an $i \in \mathcal{D}$ such that $\Lambda_{i}(m:\nu) \leq 0$ (for example, the one which corresponds to the maximum LLR at time $\nu$). Hence, by the union bound we have 
	\begin{equation}\label{tnn}
		\Pro_{A}\left(\widetilde{\nu}-\nu >m \, \big{|} \, \cF_{\nu} \right) \leq \sum_{i \in \cD}  \Pro_{A}\left( \Lambda_{i}(m:\nu) \leq 0 \,  \big{|} \, \cF_{\nu} \right).
	\end{equation}	
	For each $i \in \cD$, $n \in \bN$, and $\rho > 0$, we further have 
	\begin{equation}\label{ttz}
		\begin{aligned}
			\Pro_{A}\left( \Lambda_{i}(m:\nu) \leq 0  \, \big{|} \, \cF_{\nu} \right) &\leq \,  \Pro_{A}\left( \Lambda_{i}(m:\nu) \leq 0,\, \pi_{i}(m:\nu)\geq \rho \,  \big{|} \, \cF_{\nu} \right)\\
			&+\Pro_{A}\left(\pi_{i}(m:\nu)<\rho \, \big{|}  \, \cF_{\nu} \right).
		\end{aligned}
	\end{equation}
	Since $\nu \in \mathcal{T}$, the first term on the right hand side of \eqref{ttz} is uniformly exponentially decaying for any $\rho \in (0,1]$, by Lemma \ref{lambda_tau}(i). It remains to show that there is a $\rho \in (0,1)$, sufficiently small, such that the second term on the right hand side of \eqref{ttz} is uniformly $\mathfrak{p}/2$- polynomially decaying with respect to $m$, and since $(p-1) - \mathfrak{p}/2 < -1$ we will show \eqref{ndx}.
	
   As explained in \eqref{rnd}, the rule $\cR$ can be viewed an ordering rule with $\hat{N}(\mathcal{D})=\lambda -1+ \mathfrak{q}$, $\hat{G}(\mathcal{D})=\emptyset$, which implies $\hat{N}(\mathcal{D}) > |\hat{G}(\mathcal{D})|$. Since for $\cR$ the set of anomalous sources is fixed to $\cD \subseteq A\setminus \hat{G}(A)$, the condition \eqref{NGh_cons1} of Theorem \ref{gen_thmc} is satisfied, no matter whether $x(\cD)$ is positive or zero. Therefore, by Theorem \ref{gen_thmc} there exists a $\rho \in (0,1)$ sufficiently small such that for each $i \in \mathcal{D}$, the second term on the right hand side of \eqref{ttz} is  uniformly $\mathfrak{p}/2$- polynomially decaying.
   
   Since $\nu \in \cT$, by definition $\widetilde{\nu}$, $\nu_0$, and  $\{\tau_i \, :\, i \in \bN_0\}$ are stopping times with respect to $\{ \cF_{n} \, :\, n \in \bN \}$, it remains to show that they are also $\Pro_{A}$-a.s. finite. Since $\nu \in \cT$ by \eqref{n0n} we deduce that $\widetilde{\nu} \in \cT$, and by $\nu_0 \leq \widetilde{\nu}$ we have $\nu_0 \in \cT$. Since $\{\tau_i \, :\, i \in \bN_0\}$ is defined by recursive application of $F_{0}$ (same as for $\nu_0$), then for $\tau_0:=\nu \in \cT$ we prove by induction that $\{\tau_i \, :\, i \in \bN_0\} \in \cT$.\\
	
   (ii) Since $z \leq \widetilde{\nu} - \nu$, and $\widetilde{\nu} \in \cT$ by part (i), then $z$ is $\Pro_{A}$-a.s. finite. At time $\widetilde{\nu}$, the LLR that exceeds the current maximum LLR, it also exceeds the current $(|\mathcal{D}|-\lambda)^{th}$ largest LLR, and all others in between. Therefore, at time $\widetilde{\nu}$ a change in both the current $(|\mathcal{D}|-\lambda)^{th}$ largest LLR and the current maximum LLR occurs, and by definition of $z$ we have
   \begin{equation*}
   	\tau_{z} = \widetilde{\nu}.
   \end{equation*}	
   Since, $\widetilde{\nu} \in \cT$, and $\{\tau_i \, :\, i \in \bN_0\} \in \cT$, we have 
    \begin{equation}
    	\{ z = k \} =\{ \tau_{k} = \widetilde{\nu} \} \in \cF_{\tau_{k}}, \quad \forall \; k \in \bN,
    \end{equation}	
	which shows the claim.\\
	
	(iii) We fix $p \in [1, \mathfrak{p}-1]$. In view of inequality
	      \begin{equation*}
	      	S(\nu_0) \leq \max_{i \in \cD \setminus \cD_{0}(\nu)} S_{i}(\nu_{0,i}),
	      \end{equation*}
	     introduced in \eqref{y_max}, it suffices to show that there is a constant $C >0$ such that for all $i \in \cD \setminus \cD_{0}(\nu)$, we have
		\begin{equation}\label{oef}
		\Exp_{A}\left[ S^{p}_{i}(\nu_{0,i}) \, \big{|} \, \cF_{\nu} \right] \leq C.
	    \end{equation}
	    We fix $i \in \cD \setminus \cD_{0}(\nu)$. Since $\nu \in \cT$, the quantities $\Lambda_{i}(\nu)$, $\min_{j \in \cD_0(\nu)} \Lambda_{j}(\nu)$ and the set $\cD \setminus \cD_{0}(\nu)$ are $\cF_{\nu}$-measurable. Also, by \cite[Theorem 4.1.3]{durrett2010probability} the sequence
	    \begin{equation*}
	    	\left\{  \log  \left( \frac{f_{1i}(Y_{i}(m+\nu))}{f_{0i}(Y_{i}(m+\nu))} \right) \, :\,  m \in \bN \right\}
	    \end{equation*}
	    is iid, independent of $\cF_{\nu}$, and it has the same distribution as
	     \begin{equation*}
	    	\left\{ \log \left( \frac{f_{1i}(Y_{i}(m))}{f_{0i}(Y_{i}(m))} \right) \, :\,  m \in \bN \right\}.
	    \end{equation*}
	    We set
	    \begin{equation*}
	    g(m) := \log \left( \frac{f_{1i}(Y_{i}(m+\nu))}{f_{0i}(Y_{i}(m+\nu))}\right), \quad m \in \bN,
	    \end{equation*}
	    and we consider the sequence of stopping times $\{t_{k} \, :\, k \in \bN \}$, where
	    \begin{equation}\label{tkd}
	     t_{k} := \inf\left\{ n \geq \nu : \sum_{m=\nu}^{n} \cR_{i}(m+1) = k\right\}, \quad k \in \bN,
	    \end{equation}
	    and $t_{k} +1$ is the first time we have collected $k$ samples from source $i$, after time $\nu$. Also, for each $k \in \bN$ we consider the sum
	    \begin{equation}\label{sg}
	    \sum_{m=1}^{k} g(t_m +1), \qquad k \in \bN.
	    \end{equation}
	    Since $\{\cR_{i}(n) \, :\, n \in \bN\}$ is a predictable sequence, i.e., $\cR_{i}(n+1) \in \cF_{n}$ for all $n \in \bN_{0}$, it follows that for each $k \in \bN$, $t_k$ is a stopping time with respect to filtration $\{ \cF_{n} \, :\, n \in \bN\}$, i.e., $\{t_k =n\} \in \cF_{n}$. Also, on the event that source $i$ is sampled at each instant after $\nu$ at least with probability $\mathfrak{q}$, it follows that  for each $k \in \bN$, $t_k$ is also $\Pro_{A}$-a.s. finite, i.e., $t_k \in \cT$. By definition of $t_{k}$, at time $t_{k} +1$ we collect a sample from source $i$, and $g(t_k +1)$ is included in the sum \eqref{sg}. 
	    
	    We denote by $\kappa$ the first index that the sum \eqref{sg} exceeds the boundary $\min_{j \in \cD_0(\nu)} \Lambda_{j}(\nu) - \Lambda_{i}(\nu)$, i.e.,
	    \begin{equation*}
	     \kappa := \inf\left\{ k \geq 1 \,:\, \sum_{m=1}^{k} g(t_m +1) > \min_{j \in \cD_0(\nu)} \Lambda_{j}(\nu) - \Lambda_{i}(\nu) \right\},
	    \end{equation*}
	    and by definiton \eqref{tkd} of the sequence $\{t_{k} \, :\, k \in \bN \}$, it follows that the excess of the sum \eqref{sg} above the boundary $\min_{j \in \cD_0(\nu)} \Lambda_{j}(\nu) - \Lambda_{i}(\nu)$ is equal to $S_{i}(\nu_{0,i})$, i.e.,
	    \begin{equation*}
	    S_{i}(\nu_{0,i}) = \sum_{m=1}^{\kappa} g(t_m +1) - \left(\min_{j \in \cD_0(\nu)} \Lambda_{j}(\nu) - \Lambda_{i}(\nu)\right).
	    \end{equation*}
	    By \cite[Theorem 4.1.3]{durrett2010probability}, for each $k \in \bN$, $g(t_k +1)$ is independent of $\cF_{t_k}$ and it has the same distribution as $\log\left(f_{1i}(Y_{i}(1))/f_{0i}(Y_{i}(1))\right)$. Therefore, the sequence $\{g(t_k +1) \, :\, k \in \bN\}$ is iid, and the inequality \eqref{oef} follows by Lorden's excess inequality \cite[Theorem 3]{lorden1970excess}.\\
	    
	 (iv) We fix $p \in [1, \mathfrak{p}-1]$. Since at time $\widetilde{\nu}$ a change in both the current $(|\mathcal{D}|-\lambda)^{th}$ largest LLR and the current maximum LLR occurs, we have $\tau_{z} = \widetilde{\nu}$, and
	 \begin{equation}
	 	L(\widetilde{\nu}) \leq H_1(\widetilde{\nu}) \leq S(\tau_{z}),
	 \end{equation}   
	where $z$ is defined in \eqref{zz}, $L$ in \eqref{llw}, $H_1$ in \eqref{til_ovrs}, and $S$ in \eqref{def_S}. This means that at time $\widetilde{\nu}$ the overshoot above the current $(|\mathcal{D}|-\lambda)^{th}$ largest LLR (or equivalently the minimum LLR of the sources in $\cD_0(\tau_{z-1})$), i.e., $S(\tau_{z})$, exceeds the relative distance of all LLRs greater than or equal to the current $(|\mathcal{D}|-\lambda)^{th}$ largest LLR (or equivalently the LLRs from $\cD_1(\widetilde{\nu})$), i.e., $H_1(\widetilde{\nu})$, which by definition is larger than the size of the overshoot measured above the maximum LLR, i.e., $L(\widetilde{\nu})$. Consequently, 
	\begin{equation}
		L^{p}(\widetilde{\nu}) \leq H^{p}_1(\widetilde{\nu}) \leq S^{p}(\tau_{z}) \leq \sum_{k=1}^{z} S^{p}(\tau_{k}).
	\end{equation}
	In order to prove \eqref{lhy}, it suffices to show that there is a constant $C>0$ such that
	\begin{equation}\label{cll1}
		\Exp_{A} \left[ \sum_{k=1}^{z} S^{p}(\tau_{k}) \, \big{|}  \, \cF_{\nu} \right] \leq C.
	\end{equation}
	We have 
	\begin{equation}\label{ist}
		\begin{aligned}
			\Exp_{A} \left[ \sum_{k=1}^{z} S^{p}(\tau_{k}) \, \big{|} \,\cF_{\nu} \right] &= \Exp_{A} \left[\sum_{k=1}^{\infty} S^{p}(\tau_{k}) \, \mathbf{1}\{ z \geq k \} \, \big{|} \, \cF_{\nu} \right]\\
			&=\sum_{k=1}^{\infty} \Exp_{A} \left[S^{p}(\tau_{k}) \,  \mathbf{1}\{ z \geq k \} \, \big{|} \, \cF_{\nu} \right],
		\end{aligned}
	\end{equation}
	where the last inequality is deduced by the monotone convergence theorem. In part (ii), we proved that $z$ is a $\Pro_{A}$-a.s. finite stopping time with respect to filtration $\{\cF_{\tau_i} \, :\, i \in \bN_0\}$, which implies that $\{ z \geq k \} \in \cF_{\tau_{k-1}}$ for all $k \in \bN$, and since $\tau_{k-1} \geq \nu$ for all $k \in \bN$, by the law of iterated expectation we get
	\begin{equation}
		\begin{aligned}
			\Exp_{A}\left[S^{p}(\tau_{k}) \, \mathbf{1}\{ z \geq k \} \, \big{|} \, \cF_{\nu} \right] =\Exp_{A} \left[ \Exp_{A} \left[S^{p}(\tau_{k})\, \big{|} \,  \cF_{\tau_{k-1}} \right]\, \mathbf{1}\{ z \geq k \}  \, \big{|} \, \cF_{\nu} \right].
		\end{aligned}
	\end{equation}
	In part (iii), we proved that there is a constant $C_0 > 0$ such that for any $k \in \bN$,
	\begin{equation*}
	\Exp_{A} \left[S^{p}(\tau_{k}) \, \big{|}\,  \cF_{\tau_{k-1}} \right] \leq C_{0},
	\end{equation*}
	which implies that
	\begin{equation*}
	\Exp_{A} \left[ \sum_{k=1}^{z} S^{p}(\tau_{k}) \, \big{|} \, \cF_{\nu} \right] \leq C_{0} \, \Exp_{A}[z \, |\, \cF_{\nu}]\leq  C_{0} \, \Exp_{A}[\widetilde{\nu} - \nu \, |\, \cF_{\nu}],
	\end{equation*}
	and the claim \eqref{cll1} follows by part (i).\\
\end{IEEEproof}

In Lemma \ref{bov}, we obtain upper bounds for the moments of growth of the maximum LLR during $[\nu,w)$, where $\nu,w \in \cT$ 
and $w \geq \nu$, i.e. $U(\nu,w)$ defined in \eqref{ups}. We also obtain upper bounds on moments of the maximum distance between any two LLRs at time $w$, i.e., $H(w)$ defined in \eqref{eta0}, in terms of moments of $w-\nu$ and of the corresponding maximum distance at time $\nu$, i.e., $H(\nu)$. For this we need to introduce the following notation, where $\nu,w \in \cT$, and $w \geq \nu$.

We denote by $\{s_i\,:\, i \in \bN_{0}\}$ with $s_{0}:=\nu$, the sequence of random times starting from $\nu$ at which a change in the maximum LLR occurs, i.e.,
\begin{equation}\label{ssu}
	s_{i+1}:=J(s_{i}), \quad i \in \bN_{0}, \quad \mbox{where}\quad s_{0}:= \nu,
\end{equation}
where $J$ is defined in \eqref{ttl}. By definition $\{s_i\,:\, i \in \bN_{0}\}$ is a sequence of stopping times with respect to $\{ \cF_{n} \, :\, n \in \bN \}$. Also, since $s_0 := \nu$ is $\Pro_{A}$-a.s. finite, by application of Lemma \ref{tu}(i) we show by induction that $s_i$ is $\Pro_{A}$-a.s. finite for each $i \in \bN_0$, and thus $\{s_i\,:\, i \in \bN_{0}\} \in \cT$. We also denote by $\eta(\nu,w)$ the number of times the maximum LLR changes during $[\nu,w]$, i.e.,
\begin{equation}\label{hht}
	\eta(\nu,w) := \max\{ i\in \bN_0\,:\, s_{i} \leq w\}.
\end{equation}
Since the maximum LLR changes only whenever there is an overshoot above the current maximum, we have 
\begin{equation}\label{uu}
	U(\nu,w) \leq \sum_{i=1}^{\eta(\nu,w)} L(s_{i}),
\end{equation}
where $L$ is defined in \eqref{llw}. We proceed to the statement of Lemma \ref{bov}.\\

\begin{lemma}\label{bov}
We fix $\nu, w \in \cT$ such that $w \geq \nu$.	Let $\widetilde{w}:=J(w)$, where $J$ is defined in \eqref{ttl}.
\begin{enumerate}[(i)]
		\item The number of times the maximum LLR changes during $[\nu,\widetilde{w}]$, i.e., $\eta(\nu,\widetilde{w})$, is a $\Pro_{A}$-a.s. finite stopping time with respect to $\{\cF_{s_i} \, :\, i \in \bN_0\}$. Also, for any $p \in [1, \mathfrak{p}/2)$ there is a constant $C >0$ such that
		\begin{equation}\label{eta}
			\Exp_{A}\left[\eta^{p}(\nu,\widetilde{w})\right] \leq C \left( 1+ \Exp_{A}\left[ (w - \nu)^{p} \right] \right). 
		\end{equation}
		
		\item For any $p \in [1, \mathfrak{p}-1]$ there is a constant $C > 0$ such that
		\begin{equation}
			\Exp_{A}\left[ \sum_{i=1}^{\eta(\nu,w)}L^{p}(s_{i})\right]\leq C\, \left(1+\Exp_{A}[w -\nu]\right).
		\end{equation}
		
		\item For any $p \in [2, \mathfrak{p}/2)$ there is a constant $C> 0$ such that
		\begin{equation}\label{ccl}
			\Exp_{A}\left[U^{p}(\nu,w)\right] \leq C \, \left(1+ \Exp_{A}\left[(w-\nu)^{p}\right]\right).
		\end{equation}
		
		\item For any $p \in (2,\mathfrak{p}-1]$ and $q \in [2, \mathfrak{p}/2)$ such that $q < p$ there is a constant $C>0$ such that
		\begin{equation}\label{ttp}
			\Exp_{A}\left[H^{q}(w)\right] \leq C \left( 1 + \Exp_{A}\left[H^{p}(\nu)\right] + \Exp_{A}\left[(w-\nu)^{q} \right] \right).
		\end{equation}
	\end{enumerate}
\end{lemma}

\begin{IEEEproof} 
	(i) Since $\eta(\nu,\widetilde{w}) \leq \widetilde{w} -\nu$, and $\widetilde{w} \in \cT$ by Lemma \ref{tu}(i), then $\eta(\nu,\widetilde{w})$ is $\Pro_{A}$-a.s. finite. Also, since $\widetilde{w}$, $\{s_{i}\,:\, i \in \bN_{0}\}$ are stopping times with respect to $\{\cF_{n}\,:\, n \in \bN_{0}\}$, and at time $\widetilde{w}$ there is a change in the maximum LLR, we have 
	\begin{equation}\label{ett}
		\{ \eta(\nu,\widetilde{w}) = i \}=\{ s_{i} = \widetilde{w} \} \in \cF_{s_{i}}, \quad \forall \; i \in \bN,
	\end{equation}
	which shows the first claim. For the second claim, we fix $p \in [1, \mathfrak{p}/2)$. By Jensen's inequality, we have
	\begin{equation}
		\Exp_{A}\left[\eta^{p}(\nu,\widetilde{w})\right] \leq 2^{p-1}\left(\Exp_{A}\left[(\widetilde{w}-w)^{p}\right] + 
		\Exp_{A}\left[(w - \nu)^{p}\right] \right),
	\end{equation}	
	and by Lemma \ref{tu}(i) there is a constant $C_0 > 0$ such that $\Exp_{A}\left[(\widetilde{w}-w)^{p}\right] \leq C_0$. Therefore, for $C:=2^{p-1}\max\{C_0,1\}$ we conclude the claim.\\
	
	(ii) We fix $p \in [1, \mathfrak{p}-1]$. Since $w \leq \widetilde{w}$, it also holds $\eta(\nu,w) \leq \eta(\nu,\widetilde{w})$ , and thus
	\begin{equation}\label{stinq}
	\sum_{i=1}^{\eta(\nu,w)}L^{p}(s_{i}) \leq \sum_{i=1}^{\eta(\nu,\widetilde{w})}L^{p}(s_{i}).
	\end{equation}
	Therefore, it suffices to show that there is a $C>0$, independent of $w$ and $\nu$, such that 
	\begin{equation*}
		\Exp_{A}\left[\sum_{i=1}^{\eta(\nu,\widetilde{w})}L^{p}(s_{i})\right] \leq C\, \left(1+\Exp_{A}[w -\nu]\right).
	\end{equation*}
	Indeed, we have
	\begin{equation}\label{aa1}
		\begin{aligned}
			\Exp_{A}\left[\sum_{i=1}^{\eta(\nu,\widetilde{w})}L^{p}(s_{i})\right] = \Exp_{A}\left[\sum_{i=1}^{\infty}L^{p}(s_{i}) \, \mathbf{1}\{\eta(\nu,\widetilde{w}) \geq i\}\right]
			=\sum_{i=1}^{\infty} \Exp_{A}\left[L^{p}(s_{i}) \, \mathbf{1}\{\eta(\nu,\widetilde{w}) \geq i\}\right],
		\end{aligned}
	\end{equation}	
	where the last equality follows by monotone convergence theorem. In part (i), we showed that $\eta(\nu,\widetilde{w})$ is a $\Pro_{A}$-a.s. finite stopping time with respect to $\{\cF_{s_i} \, :\, i \in \bN_0\}$, which implies that $\{ \eta(\nu,\widetilde{w}) \geq i \} \in \cF_{s_{i-1}}$ for all $i \in \bN$, and by the law of iterated expectation we have 
	\begin{equation}
		\Exp_{A}\left[L^{p}(s_{i}) \, \mathbf{1}\{\eta(\nu,\widetilde{w}) \geq i\}\right] = \Exp_{A}\left[ \Exp_{A}\left[ L^{p}(s_{i})|\cF_{s_{i-1}}\right] \, \mathbf{1}\{\eta(\nu,\widetilde{w}) \geq i\}\right].
	\end{equation}
	By Lemma \ref{tu}(iv), there is a constant $C_0 > 0$ such that
	\begin{equation}\label{lst}
		\sup_{i \in \bN} \Exp_{A}\left[ L^{p}(s_{i}) \, | \, \cF_{s_{i-1}}\right] \leq C_0,
	\end{equation}
	which implies that
	\begin{equation}\label{llst}
		\begin{aligned}
			\Exp_{A}\left[\sum_{i=1}^{\eta(\nu,\widetilde{w})}L^{p}(s_{i})\right]  & \leq C_0 \, \Exp_{A}\left[ \eta(\nu,\widetilde{w})\right] = C_0 \left( \Exp_{A}\left[\widetilde{w}-w\right] + \Exp_{A}\left[w - \nu\right] \right),
		\end{aligned}
	\end{equation}
	and by Lemma \ref{tu}(i) there is a constant $C_1 > 0$ such that $\Exp_{A}\left[ \widetilde{w}-w\right] \leq C_1$. Therefore, for $C:=\max\{C_0 C_1, C_0\}$ we conclude the claim.\\
	
	(iii) We fix $p \in [2, \mathfrak{p}/2)$. In view of \eqref{uu} and \eqref{stinq}, we deduce that
	\begin{equation}\label{mg}
		\begin{aligned}
			U(\nu,w) \leq \sum_{i=1}^{\eta(\nu,\widetilde{w})} L(s_{i})
			=\sum_{i=1}^{\eta(\nu,\widetilde{w})} \xi_{i} + \sum_{i=1}^{\eta(\nu,\widetilde{w})} \Exp_{A}[L(s_{i}) \, |\, \cF_{s_{i-1}}],
		\end{aligned}
	\end{equation}
	where
	\begin{equation*}
		\xi_{i}:=L(s_{i})-\Exp_{A}[L(s_{i}) \, |\,  \cF_{s_{i-1}}], \quad i \in \bN.
	\end{equation*}
	In order to prove \eqref{ccl}, by application of Jensen's inequality it follows that it suffices to show that there are $C_0,\, C_1 >0$, independent of $w$ and $\nu$ such that
	\begin{equation}\label{cc0}
		\Exp_{A}\left[\left(\sum_{i=1}^{\eta(\nu,\widetilde{w})} \Exp_{A}[L(s_{i})\,|\, \cF_{s_{i-1}}] \right)^{p}\,\,\right] \leq  C_{0}\,\Exp_{A}\left[\eta^{p}(\nu,\widetilde{w})\right],
	\end{equation}
	\begin{equation}\label{cc01}
		\Exp_{A}\left[\left(\sum_{i=1}^{\eta(\nu,\widetilde{w})} \xi_{i} \right)^{p}\,\,\right] \leq  C_{1}\,\Exp_{A}\left[ \eta^{p/2}(\nu,\widetilde{w}) \right],
	\end{equation}	
	then since $\eta(\nu,\widetilde{w}) \geq 1$ we further have
	\begin{equation*}
		\Exp_{A}\left[ \eta^{p/2}(\nu,\widetilde{w}) \right] \leq \Exp_{A}\left[\eta^{p}(\nu,\widetilde{w})\right],
	\end{equation*}
	and by part (i) we deduce the claim. 
	
	For inequality \eqref{cc0}, we observe that it follows directly by \eqref{lst} (with $p=1$). In order to show \eqref{cc01}, we note that $\{ \sum_{i=1}^{m} \xi_{i},\, m \in \bN \}$ is a zero-mean martingale with respect to $\{\cF_{s_m}\,:\, m \in \bN_0 \}$, and we recall that $\eta(\nu,\widetilde{w})$ is a $\Pro_{A}$-a.s. finite stopping time with respect to $\{\cF_{s_m} \, :\, m \in \bN_0\}$ by part (i). By application of Rosenthal's inequality \cite[Theorem 2.12]{Hall_Heyde} (for stopping times \cite[Theorem 1]{ren2003rosenthal}) it follows that there is a constant $C_2 > 0$ such that 
	\begin{equation}\label{ros_xi}
		\begin{aligned}
			\Exp_{A}\left[\left(\sum_{i=1}^{\eta(\nu, \widetilde{w})} \xi_{i} \right)^{p}\,\,\right] \leq \, C_{2}\, \Exp_{A}\left[\left(\sum_{i=1}^{\eta(\nu, \widetilde{w})} \Exp_{A}[\xi_{i}^{2} \, |\, \cF_{s_{i-1}}]\right)^{p/2}\,\,\right] + C_{2}\, \Exp_{A}\left[\sum_{i=1}^{\eta(\nu, \widetilde{w})} \xi_{i}^{p}\right].
		\end{aligned}
	\end{equation}
	By application of Jensen's inequality and \eqref{lst}, it follows that there is a constant $C'_{0} > 0$, that depends only on $p$, such that 
	\begin{equation*}
		\sup_{i \in \bN} \Exp_{A}\left[ \xi^{p}_{i} \,|\, \cF_{s_{i-1}}\right] \leq C'_0,
	\end{equation*}
	which implies that the first term on the right hand of \eqref{ros_xi} is bounded by $\Exp_{A}\left[\eta^{p/2}(\nu, \widetilde{w})\right]$ up to a multiplicative constant. By part (i) we have $\{ \eta(\nu,\widetilde{w}) \geq i \} \in \cF_{s_{i-1}}$, and thus by the law of iterated expectation it follows that
	\begin{equation}
		\begin{aligned}
			\Exp_{A}\left[\sum_{i=1}^{\eta(\nu, \widetilde{w})} \xi^{p}_{i} \,\,\right] = \sum_{i=1}^{\infty} \Exp_{A}\left[ \Exp_{A}\left[ \xi^{p}_{i} \,|\,\cF_{s_{i-1}}\right]\mathbf{1}\{\eta(\nu,\widetilde{w})\geq i\}  \right]
			\leq C'_{0} \;  \Exp_{A}\left[\eta(\nu,\widetilde{w})\right],
		\end{aligned}
	\end{equation} 
	and since $p/2 \geq 1$ and $\eta(\nu,\widetilde{w}) \geq 1$ we deduce \eqref{cc01}.\\
	
	(iv) We fix $p \in (2,\mathfrak{p}-1]$, $q \in [2, \mathfrak{p}/2)$ such that $q < p$. In view of definitions \eqref{eta0}-\eqref{ups}, the maximum relative distance of the LLRs at time $w$ is bounded by
	\begin{equation}
		H(w) \leq H(\nu)+ B(\nu) + U(\nu,w).
	\end{equation}
	By Jensen's inequality we have
	\begin{equation}\label{b_nq}
		\Exp_{A}\left[ H^{q}(w) \right] \leq 3^{q-1}\left( \Exp_{A}\left[H^{q}(\nu)\right] + \Exp_{A}\left[B^{q}(\nu)\right] + \Exp_{A}\left[U^{q}(\nu,w)\right] \right).
	\end{equation}
	For the first term on the right hand side of \eqref{b_nq}, since $q < p$ we obtain
	\begin{equation}\label{ltin}
		\Exp_{A}\left[ H^{q}(\nu) \right] \leq \Exp_{A}\left[ H^{p}(\nu)\right] + 1,
	\end{equation}
	where the last inequality follows by $x^{q} \leq x^p +1$ for any $x>0$. For the second term, since $w \in \cT$, by Lemma \ref{dd_l} there is a constant $C_0 > 0$ such that 
	\begin{equation}
		\Exp_{A}\left[B^{q}(\nu)\right] \leq C_0.
	\end{equation}
	For the third term, since $w,\, \nu \in \cT$, by part (iii) there is a constant $C_1 > 0$ such that
	\begin{equation}
		\Exp_{A}\left[U^{q}(\nu,w)\right] \leq C_1 \left( 1 + \Exp_{A}\left[(w-\nu)^q\right]\right).
	\end{equation}
	Adding up the upper bounds, we show the claim for $C:= 3^{q-1}(1+C_0 +C_1)$.\\ 
\end{IEEEproof}

In Lemma \ref{vth}, we show that the length of the interval $[\sigma_{l}, \sigma_{l+1})$, where the sequence $\{\sigma_{l}\,:\, l\in \bN_{0}\}$ is defined according to \eqref{cons}, is relatively ``short" in the sense that it has some moments finite, given that the maximum distance of the LLRs at the beginning of the interval, i.e., $H(\sigma_{l})$, is relatively ``small" in the sense that it has some moments finite. To this end, we obtain upper bounds on the moments of the subintervals defined by the consecutive times in \eqref{sig1}. For this, we introduce the following notation. We fix $\lambda \geq 2$, a stopping time $w \in \cT$, and for each $j \in \{1, \ldots, \lambda-1\}$ we set
\begin{equation}\label{wj}
w_j := F_j(w), \quad j \in \{1, \ldots, \lambda-1\},
\end{equation}
where $F_j$ is defined in \eqref{ff}, and $w_j$ is the first time after $w$ that at least one of the $\lambda-j$ smallest LLRs exceeds at least one of the sources with the $|D|- (\lambda-j)$ largest LLRs. Equivalently, $w_j$ is the first time after $w$ that the LLR of at least one source from $\cD \setminus \cD_{j}(w)$ overshoots the LLR of at least one source from $\cD_{j}(w)$. Since the definition of $\sigma_{l+1}$, given $\sigma_{l} \in \cT$, is the same for all $l \in \bN_0$, in order to study the moments of $\sigma_{l+1}-\sigma_{l}$, we fix $\lambda \geq 2$ and $\sigma_{l} \equiv \nu \in \cT$ and we set
\begin{equation}\label{gl}
\widehat{\nu}:=\cG_{\lambda}(\nu),
\end{equation}
where $\cG_{\lambda}$ is defined in \eqref{fl}, and $\sigma_{l+1} := \widehat{\nu}$. Therefore, $[\nu,\widehat{\nu})$ stands for a generic interval, and by  \eqref{sig1}
\begin{equation}\label{vst}
\widehat{\nu} - \nu = (\nu_1 -\nu) + \ldots +  (\nu_{i+1} - \nu_{i}) + \ldots + (\widehat{\nu}-\nu_{\lambda-1})
\end{equation}
where $\nu_1 :=J(\nu)$, $\nu_{i+1}:=F_{i}(\nu_i)$ for all $i \in \{1,\ldots,\lambda-1\}$, and by definition $\widehat{\nu}=\nu_{\lambda}$.\\

\begin{lemma}\label{vth}
We fix $\lambda \geq 2$, and the stopping times $w, \nu \in \cT$.

\begin{enumerate}[(i)]
 
\item There are $\epsilon, \delta >0$ independent of $w$ such that for all $j \in \{1,\ldots,\lambda-1\}$ it holds
	  \begin{equation}\label{ed}
	  	\frac{j-1 + \mathfrak{q} +\epsilon}{\sum_{i \in \mathcal{D}_{j}(w)} 1/I_i} < I^*(\cD) -3\delta. 
	  \end{equation}

\item There are $\epsilon, \delta >0$ independent of $w$ such that for all $j \in \{1,\ldots,\lambda-1\}$ and for all $m \in \bN$, it holds
	\begin{equation*}
	 \begin{aligned}
	 \{ w_j - w > m \} \subseteq &\left\{ \sum_{i \in \cD_j(w)} \pi_{i}(m:w) > j -1 +\mathfrak{q}+\epsilon, w_j - w > m\right\}\\
	 &\cup  \left\{ \exists\, i \in \cD_j(w) \,:\, \bar{\Lambda}_{i}(m:w)  > m \, \delta \right\}\\
	 &\cup \left\{ \bar{\Lambda}_{i}(m:w) \leq - m \, \delta, \; \forall \, i \in \cD \setminus \cD_j(w) \right\}\cup \left\{ H(w) \geq m \delta \right\},
	 \end{aligned}
	\end{equation*}
where $\bar{\Lambda}_{i}(m:w)$ is defined in \eqref{lnmb}, and $w_j$ in \eqref{wj}.
	
\item For any $p \in (1,\mathfrak{p}-1]$ and $q \in [1, \mathfrak{p}/2)$ such that $q < p$, there is a constant $C>0$ such that
	\begin{equation}\label{ttpp2}
		\Exp_{A}\left[(w_{j}-w)^{q}\right] \leq C \left(1 +  \Exp_{A}\left[H^{p}(w)\right]\right), \quad \forall \; j \in \{1,\ldots,\lambda-1\},
	\end{equation}
where $w_j$ is defined in \eqref{wj}, and as a result for $\{\nu_i \,:\, 1 \leq i \leq \lambda\}$ in \eqref{vst} there is a constant $C>0$ such that
\begin{equation}\label{ttvv3}
	\Exp_{A}\left[(\nu_{i+1}-\nu_{i})^{q}\right] \leq C \left(1 +  \Exp_{A}\left[H^{p}(\nu_i)\right]\right), \quad \forall \; i \in \{1,\ldots,\lambda-1\}.
\end{equation}

\item For any $p,\, q \in [2,\mathfrak{p}/2)$ such that $q < p$, there is a constant $C>0$ such that
\begin{equation}\label{tosh1}
	\Exp_{A}\left[(\widehat{\nu}-\nu)^{q} \right] \leq C \left(1 +  \Exp_{A}\left[H^{p}(\nu)\right]\right),
\end{equation} 
where $\widehat{\nu}$ is defined in \eqref{gl}.
\end{enumerate}	
\end{lemma}

\begin{IEEEproof}
(i) We fix 	$j \in \{1,\ldots,\lambda-1\}$. By requirement \eqref{frakD}, for the set $\cD$ it follows that
\begin{equation}
\lambda -1 + \mathfrak{q} < \sum_{ i \in \cD} \frac{I^*(\cD)}{ I_i}.
\end{equation}
Subtracting  $|\cD \setminus \mathcal{D}_{j}(w)|=\lambda - j$ from both sides, we have
\begin{equation*}
	\begin{aligned}
		j-1 + \mathfrak{q} &<  \sum_{ i \in \cD} \frac{I^*(\cD)}{ I_i} - |\cD \setminus \mathcal{D}_{j}(w)|\\
		&= \sum_{ i \in \mathcal{D}_{j}(w)} \frac{I^*(\cD)}{ I_i}+ \left( \sum_{ i \in \cD \setminus \mathcal{D}_{j}(w)} \frac{I^*(\cD)}{ I_i}  - |\cD \setminus \mathcal{D}_{j}(w)| \right) \leq \sum_{ i \in \mathcal{D}_{j}(w)} \frac{I^*(\cD)}{ I_i} ,
	\end{aligned}
\end{equation*}
where for the last inequality we use the fact that $ I^*(\cD) /I_i \leq 1$ for every $i \in \cD$. Therefore, 
\begin{equation*}
\frac{j-1 + \mathfrak{q}}{\sum_{ i \in \mathcal{D}_{j}(w)} 1/I_i} \leq \frac{j-1 + \mathfrak{q}}{\min\limits_{D} \sum_{ i \in D} 1/I_i} < I^*(\cD),
\end{equation*}	
where the minimum is considered over all $D \subseteq \cD$ such that $|D|=|\cD|-(\lambda -j)$, which proves the claim.\\
	
(ii) We fix $j \in \{1,\ldots,\lambda-1\}$, $m \in \bN$, and $\epsilon, \delta > 0$ that satisfy part (i). We have
\begin{equation*}
 \begin{aligned}
  \{ w_j - w > m \} =&  \left\{ \sum_{i \in \cD_j(w)} \pi_{i}(m:w) > j -1 +\mathfrak{q}+\epsilon, w_j - w > m\right\}\\
                    &\cup \left\{ \sum_{i \in \cD_j(w)} \pi_{i}(m:w) \leq j -1 +\mathfrak{q}+\epsilon, w_j - w > m\right\}\\
                     \subseteq & \left\{ \sum_{i \in \cD_j(w)} \pi_{i}(m:w) > j -1 +\mathfrak{q}+\epsilon, w_j - w > m\right\}\\
                    &\cup  \left( \left\{ \sum_{i \in \cD_j(w)} \pi_{i}(m:w) \leq j -1 +\mathfrak{q}+\epsilon, w_j - w > m\right\} \cap \left\{ \bar{\Lambda}_{i}(m:w)  \leq m \, \delta, \;  \forall \, i \in \cD_j(w) \right\} \right)\\
                    &\cup \left\{ \exists \, i \in \cD_j(w) \,:\, \bar{\Lambda}_{i}(m:w)  > m \, \delta \right\}.
 \end{aligned}
\end{equation*}
Therefore, in order to prove the claim it suffices to show that
\begin{equation*}
\begin{aligned}
&\left\{ \sum_{i \in \cD_j(w)} \pi_{i}(m:w) \leq j -1 +\mathfrak{q}+\epsilon, w_j - w > m\right\} \cap \left\{ \bar{\Lambda}_{i}(m:w)  \leq m \, \delta, \;  \forall \, i \in \cD_j(w) \right\}\\
& \; \subseteq  \left\{ \bar{\Lambda}_{i}(m:w) \leq - m \, \delta, \; \forall \, i \in \cD \setminus \cD_j(w) \right\} \cup \left\{ H(w) \geq m \delta \right\}.
\end{aligned}
\end{equation*}	
By definition of $F_j$ in \eqref{ff}, it follows that on the event $\{ w_j - w > m \}$, for all $u \in \cD \setminus \cD_j(w)$ it holds
\begin{equation}\label{ineqq}
	\Lambda_{u}(m:w) \leq \min_{i \in \cD_j(w)} \left\{ \Lambda_{i}(m:w) \right\} + H(w), 
\end{equation} 	
meaning that starting from time $w$, the increase in $\Lambda_{u}$ is smaller that the minimum increase of all $\Lambda_{i}$, where $i \in \cD_j(w)$, augmented by the maximum distance of all LLRs at time $w$, i.e. $H(w)$, otherwise we would have $\{ w_j - w \leq m \}$. Recalling the property \eqref{decompose} of $\bar{\Lambda}_{i}(m:w)$, by moving $H(w)$ to the left-hand side and dividing both sides by $m$, we have
\begin{equation*}
	\frac{\bar{\Lambda}_{u}(m:w) -H(w)}{m} + \pi_{u}(m:w)I_{u} \leq \min_{i \in \cD_j(w)} \left\{ \frac{\bar{\Lambda}_{i}(m:w)}{m} + \pi_{i}(m:w)I_{i} \right \}.
\end{equation*}	
On the event $\{ w_j - w > m \}$, at each instant during $[w,w+m]$ the ordering rule $\cR$ samples all sources in $\cD \setminus\cD_j(w)$ because they correspond to the $\lambda - j$ smallest LLRs, which implies that
\begin{equation}
\pi_{u}(m:w) = 1, \quad \forall \, u \in \cD \setminus\cD_j(w),
\end{equation}	
and since we intersect by the event $\left\{ \bar{\Lambda}_{i}(m:w)  \leq m \, \delta, \;  \forall \, i \in \cD_j(w) \right\}$, we deduce that for all $u \in \cD \setminus \cD_j(w)$ it holds
\begin{equation}
	\frac{\bar{\Lambda}_{u}(m:w) -H(w)}{m} + I_{u} \leq \delta + \min_{i \in \cD_j(w)} \left\{ \pi_{i}(m:w)I_{i} \right \}.
\end{equation}
Clearly,
\begin{equation*}
\min_{i \in \cD_j(w)}\left\{\pi_{i}(m:w)I_{i}  \right \} \leq \max_{v_i} \min_{i \in \cD_j(w)}\left\{v_i I_{i}  \right \},
\end{equation*}
where the maximum is considered over all $v_i \in [0,1]$, and since we intersect by the event $\{ \sum_{i \in \cD_j(w)} \pi_{i}(m:w) \leq j -1 +\mathfrak{q}+\epsilon\}$, we must further have
\begin{equation*}
\sum_{i \in \cD_j(w)} v_{i} \leq j -1 +\mathfrak{q}+\epsilon.
\end{equation*}
The solution of the constrainted max-min problem requires $v_i I_i = v_k I_k$ for all $i,\, k \in \cD_j(w)$, which implies that
\begin{equation}
	\max_{v_i} \min_{i \in \cD_j(w)}\left\{v_i I_{i} \right \} = \frac{j -1 +\mathfrak{q} + \epsilon}{ \sum_{i \in \cD_j(w)} 1/I_i} < I^*(\cD) -3\delta.
\end{equation}
where the inequality follows by part (i), and since $I^*(\cD) \leq I_{u}$ for all $u \in \cD \setminus \cD_j(w)$, we deduce that 
\begin{equation}
	\begin{aligned}
		\frac{\bar{\Lambda}_{u}(m:w) - H(w)}{m}  < -2 \delta, \quad \forall \, u \in \cD \setminus \cD_j(w).
	\end{aligned}
\end{equation}
Clearly,
\begin{equation*}
  \begin{aligned}
  &\left\{\frac{\bar{\Lambda}_{u}(m:w) - H(w)}{m}  < -2 \delta, \; \forall \, u \in \cD \setminus \cD_j(w)  \right\}\\
  &\subseteq \{\bar{\Lambda}_{u}(m:w) \leq -m \delta, \; \forall \, u \in \cD \setminus \cD_j(w) \} \cup \left\{ H(w) \geq m \delta \right\},
  \end{aligned}
\end{equation*}
which proves the claim.\\

(iii) We fix $p \in (1,\mathfrak{p}-1]$, $q \in [1, \mathfrak{p}/2)$ such that $q < p$, and $j \in \{1,\ldots,\lambda-1\}$. In order to show the claim it suffices to show that there is a constant $C>0$ such that
\begin{equation*}
\sum_{m=1}^{\infty} m^{q-1} \, \Pro\left( w_{j}-w > m \right) \leq  C \left(1 +  \Exp\left[H^{p}(w)\right]\right). 
\end{equation*}
By part (ii) and conditional Boole's inequality  we deduce that for every $m \in \bN$
\begin{equation}\label{blin}
\begin{aligned}
        \Pro_{A}\left( w_{j}-w >m \, |\,  \cF_{w}\right)
		\leq& \Pro_{A}\left( \sum_{i \in \cD_j(w)} \pi_{i}(m:w) > j -1 +\mathfrak{q}+\epsilon,\, w_{j}-w >m \, |\, \cF_{w}\right)\\
		&+	\sum_{i \in \cD_j(w)}  \Pro_{A}\left( \bar{\Lambda}_{i}(m:w) > \delta\, m \, |\, \cF_{w} \right) \\
		&+ \Pro_{A}\left( \bar{\Lambda}_{u}(m:w) \leq - m \, \delta, |\, \cF_{w}\right)+ \mathbf{1}\{ H(w) \geq \delta\, m\},
\end{aligned}
\end{equation}
where for the third term on the right-hand side we have fixed a $u \in \cD \setminus \cD_j(w)$, and for the forth term we applied the fact that $\{H(w) \geq \delta\, m\} \in \cF_{w}$. Therefore, it suffices to show that there are constants $C_1$, $C_2$, $C_3$ such that
\begin{equation}\label{st1}
	\sum_{m=1}^{\infty} m^{q-1} \, \Pro_{A}\left( \sum_{i \in \cD_j(w)} \pi_{i}(m:w) {>} j-1 +\mathfrak{q}+\epsilon, \, w_{j}-w >m \, | \, \cF_{w}\right) \leq C_{1},
\end{equation}
\begin{equation}\label{st2}
\sum_{m=1}^{\infty} m^{q-1} \Pro_{A}\left( \bar{\Lambda}_{i}(m:w) > \delta m \,| \, \cF_{w}\right) \leq C_{2}, \qquad \forall \, i \in \cD_j(w),
\end{equation}
\begin{equation}\label{st3}
	\sum_{m=1}^{\infty} m^{q-1} \, \Pro_{A}\left( \bar{\Lambda}_{u}(m:w)  \leq -m \, \delta  \,| \,  \cF_{w}\right) \leq C_{3}.
\end{equation}
Indeed, having those bounds we deduce that
\begin{equation}\label{tp0}
\Exp_{A}\left[ (w_j-w)^{q} | \cF_{w} \right] \leq C_1 + M C_2 + C_3 +  \sum_{m=1}^{\infty} m^{q-1} \, \mathbf{1}\{ H(w)\geq \delta \, m\},
\end{equation}
where $M$ is the total number of sources. Taking expectations on both sides, we have
	\begin{equation}
	\begin{aligned}
		\Exp_{A}\left[ (w_j-w)^{q}\right] &\leq C_1 + M C_2 + C_3 + \sum_{m=1}^{\infty} m^{q-1} \, \Pro_{A}\left(H(w)\geq \delta \, m\right)\\
		&\leq C_1 + M C_2 + C_3 + \frac{\Exp_{A}\left[ H^{p}(w)\right]}{\delta^p} \sum_{m=1}^{\infty} m^{(q-p)-1}, 
	\end{aligned}
\end{equation}
where for the first inequality we apply the monotone convergence theorem, and for the second  Markov's inequality. Since $p > q$, we conclude \eqref{ttpp2} for $C:=\max\{ C_1 + M C_2 + C_3, (1/\delta^p) \sum_{m=1}^{\infty} m^{(q-p)-1} \}$. We proceed to the proof of \eqref{st1}, \eqref{st2} and \eqref{st3}.

For \eqref{st1}, we note that on the event $\{ w_j-w >m \}$, at each time instant $n \in [w,w +m]$ the  ordering rule  $\cR$ samples all sources in  $\cD \setminus\cD_j(w)$, because they correspond to the $\lambda-j$ smallest LLRs in $\cD$, and also the $j-1+ \mathbf{1}\{ Z_n \leq \mathfrak{q}\}$ sources with the smallest LLRs in $\cD_j(w)$, and as a result
\begin{equation*}
	\sum_{i \in \cD_j(w)} \pi_{i}(m:w) = j-1 + \frac{1}{m} \sum_{n=0}^{m-1} \mathbf{1}\{Z_{n+w} \leq \mathfrak{q}\},
\end{equation*}
which implies
\begin{equation*}
	\begin{aligned}
		\Pro_{A}&\left( \sum_{i \in \cD_j(w)} \pi_{i}(m:w) > j -1 +\mathfrak{q}+\epsilon,\, w_j-w >m \,|\, \cF_{w}\right)\\
		&\leq \Pro_{A}\left( \frac{1}{m} \sum_{n=0}^{m-1} Z_{n+w} > \mathfrak{q}+\epsilon \,|\, \cF_{w} \right)\leq \Pro_{A}\left(\frac{1}{m} +\frac{1}{m-1} \sum_{n=1}^{m-1} Z_{n+w} > \mathfrak{q}+\epsilon \,|\, \cF_{w} \right)
	\end{aligned}
\end{equation*}
By \cite[Theorem 4.1.3]{durrett2010probability}, we know that that the sequence $\{Z_{n+w} \,:\, n \geq 1\}$ is iid, independent of $\cF_{w}$, and it has the same distribution as $\{Z_{n} \,:\, n \geq 1\}$. For all $m > 1/\epsilon \Leftrightarrow \epsilon > 1/m $, the right hand side is exponentially decaying by Chernoff bound, which proves \eqref{st1}.

For \eqref{st2}, we fix $i \in \cD_j(w)$. By Lemma \ref{polyn}, it follows that $\Pro_{A}\left( \bar{\Lambda}_{i}(m:w) > \delta m \, |\,  \cF_{w}\right)$ is uniformly $\mathfrak{p}/2$-polynomially decaying. Since, by assumption $q < \mathfrak{p}/2$, we conclude \eqref{st2}.

For \eqref{st3}, we note that by Lemma \ref{exp_tau} it follows that $\Pro_{A}\left( \bar{\Lambda}_{u}(m:w)  \leq -\delta m \, | \, \cF_{w} \right)$ is uniformly exponentially decaying, which proves \eqref{st3}.\\

(iv) We fix $p,\, q \in [2,\mathfrak{p}/2)$ such that $q < p$. Without loss of generality, we assume that $H(\nu) \in \cL^p$, otherwise the inequality holds trivially. For the purpose of the proof we consider the decreasing sequence $\{t_i:\, 1 \leq i \leq \lambda \}$ such that
\begin{equation}
	p=t_1 > \ldots > t_{\lambda}=q.
\end{equation}
In view of \eqref{vst}, in order to prove the claim, it suffices to show that for each $i \in \{ 1,\ldots,\lambda\}$ there is a constant  $C_{i} > 0$ such that 
\begin{equation}\label{ci}
\Exp_{A} \left[ (\nu_{i} - \nu_{i-1})^{t_{i}}  \right] \leq C_{i}\left( 1 + \Exp_{A}\left[ H^{p}(\nu)\right]   \right), \quad \forall \; i \in \{ 1,\ldots,\lambda\},
\end{equation}
where for simplicity we renamed $\nu \equiv \nu_0$. Indeed, if this is true then by Jensen’s inequality we have
\begin{align}\label{toz}
	\begin{split}
		\Exp_{A} \left[ ( \widehat{\nu} -\nu)^{q}\right] =
		\Exp_{A} \left[ \left(\sum_{i=1}^\lambda (\nu_i - \nu_{i-1} ) \right)^{q}\right] &\leq
		\lambda^{q-1} \, \sum_{i=1}^{\lambda} \Exp\left[ \left( \nu_i - \nu_{i-1}  \right)^{q} \right]\\
		  &\leq \lambda^{q-1} \, \sum_{i=1}^{\lambda}  \Exp_{A} \left[ \left( \nu_i - 
		\nu_{i-1}  \right)^{t_{i}} \right]\\
		& \leq  \left( \lambda^{q-1} \, \sum_{i=1}^{\lambda} C_{i} \right) \,	 \left( 1 + \Exp_{A} \left[ H^{p}(\nu)\right]   \right),
	\end{split}
\end{align}
where the second inequality follows by the fact that $\nu_i - \nu_{i-1} \geq 1$, and $t_i \geq q$ for all $i \in \{1,\ldots,\lambda\}$. In order to show \eqref{ci}, we apply induction on $i \in \{1,\ldots,\lambda\}$.

\textit{Basis of the induction}: For $i=0$, by definition $\nu_1:=J(\nu)$, and $t_1 = p\in [2,\mathfrak{p}/2)$. By Lemma \ref{tu}(i) it follows that there is a constant $C_1 >0$ such that  $\Exp_{A} \left[ (\nu_{1} - \nu)^{t_{1}}  \right] \leq C_{1}$, which implies \eqref{ci} for $i=0$.

\textit{Step of the induction}: We fix $k \leq \lambda$, we assume that \eqref{ci} is satisfied for all $i \in \{1,\dots,k-1\}$, and we show that it is also satisfied for $i=k$. Since $t_k, t_{k-1} \in [2,\mathfrak{p}/2)$ and  $t_k < t_{k-1}$, by \eqref{ttvv3} of part (iii) it follows that there is a constant $D_{0} > 0$ such that
\begin{equation}
\Exp_{A}\left[(\nu_{k}-\nu_{k-1})^{t_k}\right] \leq D_0 \left(1 +  \Exp_{A}\left[H^{t_{k-1}}(\nu_{k-1})\right]\right),
\end{equation}
and thus in order to show \eqref{ci} for $i=k$, it suffices to show that there is a constant $D_{1} > 0$ such that
\begin{equation}\label{x3}
	\Exp\left[ H^{t_{k-1}}(\nu_{k-1})\right] \leq D_{1} \left( 1 + \Exp\left[H^{p}(\nu)\right] \right).
\end{equation}   	
Since $\nu \in \cT$, by Lemma \ref{bov}(iv) we deduce that if also $\nu_{k-1} \in \cT$ then there is a constant $D_2 >0$ such that
\begin{equation*}
\Exp_{A}\left[  H^{t_{k-1}}(\nu_{k-1})\right] \leq D_2 \left( 1 + \Exp_{A}\left[H^{p}(\nu)\right] + \Exp_{A}\left[(\nu_{k-1}-\nu)^{t_{k-1}} \right] \right).
\end{equation*}
Therefore, in order to show \eqref{ci} for $i=k$, it suffices to show that there is a constant $D_{3} > 0$ such that 	
\begin{equation}\label{x0}
\Exp_{A}\left[(\nu_{k-1}-\nu)^{t_{k-1}}\right] \leq D_{3} \left(1 + \Exp_{A}\left[ H^{p}(\nu)\right]\right),
\end{equation}	
and since $\nu \in \cT$, and $H(\nu) \in \cL^{p}$, the \eqref{x0} will also imply that $\nu_{k-1} \in \cT$. Indeed, by the induction hypothesis that \eqref{ci} is satisfied for all $i \in \{1,\dots,k-1\}$, Jensen's inequality, and the fact that $t_i \geq t_{k-1}$ for all $i \leq k-1$, we deduce that
\begin{align*}
	\begin{split}
		\Exp_{A} \left[(\nu_{k-1} -\nu)^{t_{k-1}}\right] \leq (k-1)^{t_{k-1}-1} \, \sum_{i=1}^{k-1}  \Exp_{A} \left[ \left( \nu_i - 
		\nu_{i-1}  \right)^{t_{i}} \right]
		\leq  \left( (k-1)^{t_{k-1}-1} \, \sum_{i=1}^{k-1} C_{i} \right) \,	 \left( 1 + \Exp_{A} \left[ H^{p}(\nu)\right]   \right),
	\end{split}
\end{align*}
which proves \eqref{x0}.
\end{IEEEproof}

\end{document}